\documentclass[11pt,openany,leqno]{article}
\usepackage{amsmath,amsthm,amsfonts,amssymb,amscd}
\usepackage{graphics}

\usepackage{enumerate}
\usepackage{stackrel}
\usepackage{epsfig} 
\usepackage{hyperref}
\begin{document}

\def\Max{\text{max}}
\def\P{\mathbb P}
\def\R{\mathbb R}
\def\T{\mathbb{T}}
\def\N{\mathbb N}
\def\Z{\mathbb Z}
\def\C{\mathbb C}
\def\D{\mathbb D}
\def\M{\mathbb M}
\def\a{{\underline a}}
\def\b{{\underline b}}
\def\n{{\underline n}}
\def\Log{\text{log}}
\def\loc{\text{loc}}
\def\inta{\text{int }}
\def\det{\text{det}}
\def\exp{\text{exp}}
\def\Re{\text{Re}}
\def\lip{\text{Lip}}
\def\leb{\text{Leb}}
\def\dom{\text{Dom}}
\def\diam{\text{diam}\:}
\def\supp{\text{supp}\:}
\newcommand{\ovfork}{{\overline{\pitchfork}}}
\newcommand{\ovforki}{{\overline{\pitchfork}_{I}}}
\newcommand{\Tfork}{{\cap\!\!\!\!^\mathrm{T}}}
\newcommand{\whforki}{{\widehat{\pitchfork}_{I}}}
\newcommand{\marginal}[1]{\marginpar{{\scriptsize {#1}}}}
\def\sR{{\mathfrak R}}
\def\sM{{\mathfrak M}}
\def\sA{{\mathfrak A}}
\def\sB{{\mathfrak B}}
\def\sY{{\mathfrak Y}}
\def\sE{{\mathfrak E}}
\def\sP{{\mathfrak P}}
\def\sG{{\mathfrak G}}
\def\sa{{\mathfrak a}}
\def\sb{{\mathfrak b}}
\def\sc{{\mathfrak c}}
\def\se{{\mathfrak e}}
\def\sg{{\mathfrak g}}
\def\sd{{\mathfrak d}}
\def\sr{{\mathfrak {sr}}}
\def\ss{{\mathfrak {s}}}
\def\sD{{\mathfrak {p}}}
\def\sp{{\mathfrak {p}}}
\def\arr{\overleftarrow}
\def\u{\underline}
\def \S{\mathbb S}
\def \A{\mathbb A}
\def \Diff{{\rm Diff}}
\def \qand{\quad \text{and} \quad}
\newtheorem{prop}{Proposition} [section]
\newtheorem{thm}[prop] {Theorem}
\newtheorem{conj}[prop] {Conjecture}
\newtheorem{defi}[prop] {Definition}
\newtheorem{lemm}[prop] {Lemma}

\newtheorem{prob}[prop] {Problem}

\newtheorem{sublemm}[prop] {Sub-Lemma}
\newtheorem{cor}[prop]{Corollary}
\newtheorem{theo}{Theorem}
\newtheorem{theoprime}{Theorem}
\newtheorem{Claim}[prop]{Claim}
\newtheorem{fact}[prop]{Fact}

\newtheorem{coro}[theo]{Corollary}
\newtheorem{defprop}[prop]{Definition-Proposition}
\newtheorem{propdef}[prop]{Proposition-Definition}
\newtheorem{RL}[prop]{Rescaling Lemma}

\newtheorem{question}[prop]{Question}
\newtheorem{conjecture}[prop]{Conjecture}

\theoremstyle{remark}
\newtheorem{exam}[prop]{Example}
\newtheorem{rema}[prop]{Remark}

\renewcommand{\thetheo}{\Alph{theo}}
\renewcommand{\thetheoprime}{\Alph{theo}$'$}
\title{On Herman's Positive Entropy Conjecture}

\author{Pierre Berger\footnote{CNRS-Universit\'e Paris 13}\; and Dimitry Turaev\footnote{Imperial College London}}

\date{March 22, 2019}

\maketitle

\begin{abstract} 
We show that any area-preserving $C^r$-diffeomorphism of a two-dimensional surface displaying an elliptic fixed point can be $C^r$-perturbed to one exhibiting a chaotic island whose metric entropy is positive, for every $1\le r\le \infty$. This proves a conjecture of Herman stating that the identity map of the disk can be $C^\infty$-perturbed to a conservative diffeomorphism with  positive metric entropy. This implies also that the 
Chirikov standard map for large and small parameter values can be $C^\infty$-approximated by  a conservative diffeomorphisms displaying a positive metric entropy (a weak version of Sinai's positive metric entropy conjecture).  Finally, this sheds light onto a Herman's question on the density of $C^r$-conservative diffeomorphisms displaying a positive metric entropy: we show the existence of a dense set formed 
by conservative diffeomorphism which either are weakly stable (so, conjecturally, uniformly hyperbolic) or display a chaotic island of positive metric entropy. 
\end{abstract}
\tableofcontents

\section*{Introduction}

Consider a diffeomorphism $f$ of a two-dimensional surface $\M$. The \emph{maximal Lyapunov exponent} of $x\in \M$ is
\begin{equation}\label{mlyap}
\lambda(x)= \limsup_{n\to \infty} \frac1n \log \|Df^n(x)\|\; .
\end{equation}
It quantifies the sensitivity to the initial conditions: if $\lambda(x)$ is positive, then the forward orbits of most of the points from a neighborhood of $x$ 
will diverge exponentially fast from the orbit of $x$. 

Let $f$ preserve a smooth area form $\omega$ on $\M$. The {\em metric entropy}\footnote{We employ here Pesin formula for the Kolmogorov-Sinai entropy \cite{Pe77}.} 
of $f$ is the integral
\begin{equation}\label{entrode}
h_\omega(f):=\int_\M  \lambda(x) \;\omega(dx).
\end{equation}
Whenever the metric entropy of a dynamical system is positive, points in $\M$ display a positive Lyapunov exponent {\em with non-zero probability}.

One of the most fundamental questions in conservative dynamics is

\begin{question} How typical are conservative dynamical systems with positive metric entropy?
\end{question}

Note that a different notion of topological entropy is one of the basic tools in describing chaotic dynamics: positive topological entropy indicates the presence
of uncountably many orbits with a positive maximal Lyapunov exponent \cite{Ka80}. However, the positivity of the metric entropy is a much stronger property, as it
ensures positive maximal Lyapunov exponent for a non-negligible set of initial conditions. While numerical evidence for a large set of initial conditions corresponding to seemingly chaotic behavior in area-preserving maps is abundant, a rigorous proof for the positivity of metric entropy is available only for
a small set of specially prepared examples (see Section 1). Currently no mathematical technique exists for answering Question 0.1 in full generality.  

Several prominent conjectures are related to this question. In order to formulate them, let us recall the topologies involved.
For $1\le r\le \infty$, let ${\rm {\rm Diff}}_\omega^r(\M)$ be the space of diffeomorphisms which keep the area form $\omega$ invariant. When $r<\infty$, 
the space ${\rm {\rm Diff}}_\omega^r(\M)$ is endowed with the uniform $C^r$-topology. The space ${\rm {\rm Diff}}_\omega^\infty(\M)$ is endowed with the projective limit
topology whose base is formed by all $C^r$-open subsets for all finite $r$. Let us fix a metric $d_r$ compatible with the $C^r$-topology (the space 
${\rm {\rm Diff}}_\omega^r(\M)$ with the metric $d_r$ is complete). The $C^\infty$-topology in ${\rm {\rm Diff}}_\omega^\infty(\M)$ is defined by
the following metric:
$$d_\infty(f,g)=\sum_{r=0}^\infty \frac1{r!} \min(1,d_r(f,g))$$
(note that ${\rm {\rm Diff}}_\omega^\infty(\M)$ with the metric $d_\infty$ is complete).

Consider the two-dimensional disc $\D:=\{(x,y)\in \R^2:x^2+y^2\le 1\}$. Let $id$ be the identity map of $\D$. One of the favorite conjectures of Herman can be formulated as follows.

\begin{conjecture}[Herman \cite{He98}]
For every $\varepsilon>0$ there exists $f\in {\rm {\rm Diff}}^\infty_\omega(\D)$ such that $d_\infty(f,id) <\varepsilon$ and the metric entropy of $f$ is
positive:  $h_\omega(f)>0$.
\end{conjecture}
It is linked to his question:
\begin{question}[Herman \cite{He98}]
Is the set of diffeomorphisms $f$ with positive metric entropy $h_\omega(f)$ dense in ${\rm {\rm Diff}}^\infty_\omega(\D)$?
\end{question}

In this work we prove Herman's Conjecture 0.2. This also could be a step towards a positive answer to Question 0.3.
Recall that a periodic point $P$ of $f$ is \emph{hyperbolic} if the eigenvalues of $Df^p(P)$ (where $p$ is the period of $P$) are 
not equal to $1$ in the absolute value. The main result of this work is the following

\begin{theo}\label{main}
For any surface $(\M,\omega)$, if a diffeomorphism $f\in {\rm {\rm Diff}}^\infty_\omega(\M)$ has a periodic point which is
 not hyperbolic, then there is a
$C^\infty$-small (as small as we want) perturbation of $f$ such that the perturbed map $\hat f\in  {\rm {\rm Diff}}^\infty_\omega(\M)$
has positive metric entropy: $h_\omega (\hat f)>0$.
\end{theo}

If $f=id$, then every point in $\M$ is a non-hyperbolic fixed point of $f$, so Theorem A implies Herman's conjecture immediately. 
Another immediate consequence employs the notion of the weak stability \cite{Ma82}. The map  $f\in {\rm {\rm Diff}}^r_\omega (\M)$ is 
$C^r_\omega$-{\em weakly stable} if all the periodic points of any $C^r$-close to $f$ map $\hat f\in {\rm {\rm Diff}}^r_\omega(\M)$ are hyperbolic. 

\begin{coro}
A diffeomorphism $f\in {\rm {\rm Diff}}^\infty_\omega(\M)$ is either $C^\infty_\omega$-weakly stable or $C^\infty$-approximated by 
a diffeomorphism from ${\rm {\rm Diff}}^\infty_\omega(\M)$ which has positive metric entropy.
\end{coro}

This statement suggests that the answer to Question 0.3 has to be positive. The reason is that the common belief among dynamicists is that
any $C^\infty$-weakly stable map of a closed manifold is uniformly hyperbolic (see Section 1.7). If this conjecture is true, then every $C^\infty_\omega$-weakly stable
diffeomorphism has positive metric entropy (and, moreover, no $C^\infty_\omega$-weakly stable
diffeomorphisms exist when $\M=\D$).

Another metric entropy conjecture regards the very popular Chirikov standard map family. This is a one-parameter family of area-preserving diffeomorphisms of $\T^2$ defined for $a \in \R$ by
\begin{equation}\label{chirikov} T_a(x,y) = (2x -y + a \sin 2 \pi x, x).\end{equation}
For $a\in (0,\frac{2}{\pi})$ this map has an elliptic fixed point at $(x=1/2, y=1/2)$ (a period-$p$ point $P$ of an area-preserving map $f$
is called \emph{elliptic} if the eigenvalues of $Df^p(P)$ are equal to $e^{\pm i\alpha}$ where $\alpha\in (0,\pi)$, i.e. they are complex
and lie on the unit circle).
When $a$ increases, the elliptic fixed point loses stability and, at $a$ large enough, the numerically obtained phase portraits display a large set
where the dynamics is apparently chaotic (the so-called ``chaotic sea'' \cite{Ch79}). A conjecture due to Sinai can be formulated as
follows (cf. \cite{Si94} P.144):
\begin{conj}\label{sinaiconj}
There exists a set $\Lambda\subset\R$ of positive Lebesgue measure such that, for $a \in \Lambda$, the metric entropy of $T_a$ is positive.
\end{conj}
This conjecture is still completely open despite intense efforts, see e.g. \cite{GL00}. 
However, it is shown by Duarte \cite{Du94} that the map $T_a$ 
has elliptic periodic points for an open and dense set of sufficiently large values of parameter $a$. Hence, our main theorem implies
the following ``approximative version'' of Sinai's Conjecture 0.4:
\begin{coro}
For every sufficiently large or sufficiently small $a\in \R$, there exists a $C^\infty$-small perturbation $\hat T\in {\rm {\rm Diff}}^\infty_\omega$ of the map $T_a$ 
such that $\hat T$ has positive metric entropy.
\end{coro}

We note that a large set (of almost full Lebesgue measure) in a neighborhood of a generic elliptic point of an area-preserving $C^r$-diffeomorphism 
with $r\geq 4$ consists of points with zero Lyapunov exponent (the points on KAM-curves \cite{Ko54}). Our result, nevertheless, shows that the Lebesgue 
measure of the points with positive maximal Lyapunov exponent in a neighborhood of any
elliptic point can be positive too. 

The proof of Theorem A occupies Sections 2-6 of this paper. In Section 1 we remind certain background information pertinent to
this work.

\thanks{This work was partially financed by the project BRNUH of USPC university, the Steklov institute, the Royal Society, the 
ERC project of S. Van Strien, and the EPSRC project by D.Turaev. The second author was supported by the grant 14-41-00044 
of the RSF. The first author thanks M.-C. Arnaud for informing him about the Herman's conjecture. The second author is grateful
to L. M. Lerman who attracted his attention to Przytycki example. We are also grateful to M. Chaperon, S. Crovisier,  B.Fayad, 
F. Przytycki, and E. Pujals for important comments.
}
\section{Selected results and conjectures around the positive metric entropy conjecture}

The study of the instability and chaos in conservative dynamics enjoys a long tradition since the seminal work by Poincar\'e \cite{Po87} 
on the tree-body problem. 

\subsection{Uniformly hyperbolic maps}

An invariant compact set $K$ of the diffeomorphism $f$ of a manifold $\M$ is uniformly hyperbolic if the restriction $T\M|_K$ of the tangent bundle of $\M$ 
to $K$ splits into two $Df$-invariant continuous sub-bundles $E^s$ and $E^u$ such that $E^s$ is uniformly contracted and $E^u$ is uniformly expanded:
\[T\M|_K=E^s\oplus E^u\; , \quad \exists N\ge 1, \quad  \|Df^N|_{E^s}\|<1\; , \quad  \|Df^{-N}|_{E^u}\|<1\; .\]
Whenever $K=\M$, the diffeomorphism $f$ is called \emph{uniformly hyperbolic} or Anosov. Note that the maximal Lyapunov exponent
of a uniformly hyperbolic dynamical system is positive at every point. Hence, whenever the dynamics is conservative (i.e. $f$ keeps invariant a volume form $\omega$), the metric entropy is positive.

Such dynamics are very well understood. However, the existence of the splitting of $T\M$ into two non-trivial continuous sub-bundles imposes strong
restrictions on the topology of $\M$. For instance, there exists no uniformly hyperbolic diffeomorphism of a closed two-dimensional surface different
from the torus $\T^2$. 

\subsection{Stochastic island}
The first example of a {\em non-uniformly hyperbolic} area-preserving map of the disk was given by Katok \cite{Ka79}. His construction 
started with an Anosov diffeomorphism of $\T^2$ with 4 fixed points and a certain symmetry. Then the diffeomorphism is modified so that the fixed 
points become non-hyperbolic and sufficiently flat (while the hyperbolicity is preserved outside of the fixed points). After that, the torus is projected
to a two-dimensional disc.  
The singularities of this projection correspond to the fixed points, and their flatness allows for making the resulting map of the disc a diffeomorphism.
The positivity of the metric entropy is inherited from the original Anosov map.

This example has been pushed forward by Przytycki \cite{Pr82}, where instead of making the fixed points non-hyperbolic he made a surgery 
to replace each of these fixed points by an elliptic island (in this case - a neighborhood of an elliptic fixed point filled by closed invariant curves)
bounded by a heteroclinic link (as defined below). 
Przytycki's example was put in a more general context by Liverani \cite{Li04}.

Recall that given a hyperbolic periodic point $P$ of a diffeomorphism $f$, the following sets are immersed smooth submanifolds: the \emph{stable} and,
respectively, \emph{unstable manifolds}
\[W^s(P; f):=\{x\in \M: d(f^n (x), P)\underset{n\to+\infty}{\to} 0\}\quad\text{and}\quad W^u(P; f):=\{x\in \M: d(f^n (x), P)\underset{n\to-\infty}{\to} 0\}.\]
We will call a $C^0$-embedded circle $L$ a \emph{heteroclinic $N$-link} if there exists $N$ hyperbolic periodic points
$P_1,\dots, P_N \in L$ satisfying $L\subset \cup_{i=1,\dots,N}  W^s(P_i)\cup W^u(P_i)$. Note that a heteroclinic $N$-link is a piecewise $C^\infty$-curve with possible break points at the periodic points $P_i$. 

\begin{figure}[h!]
	\centering
		\includegraphics[width=5cm]{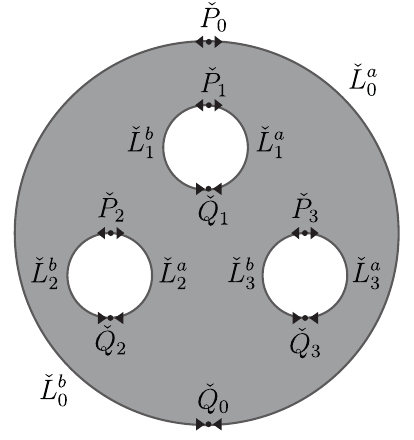}
\caption{Stochastic island}\label{fig:island}
\label{fig:oldcoord}
\end{figure}

Przytycki construction gives an example of the stochastic island in the followings sense.
\begin{defi} A \emph{stochastic island} is a two-dimensional domain $\mathcal I$ bounded by finitely many heteroclinic links such
that every point in $\mathcal I$ has positive maximal Lyapunov exponent. 
\end{defi}

In this paper (see Section \ref{section1}) we build one more example of a map $f\in {\rm {\rm Diff}}^\infty_\omega(\D)$ with
a stochastic island (where $\omega$
is the standard area form $dx\wedge dy$ in the unit two-dimensional disc $\D$. To this aim, we adapt to the conservative setting the Aubin-Pujals
blow-up construction \cite{AP09}.

One of the main difficulties in the Herman's entropy conjecture is that no other examples of conservative maps of a disc with positive metric entropy are known. All of these constructions are very fragile (sensitive to perturbations): for example, no $C^r$-generic finite-parameter family of area-preserving diffeomorphisms can have a parameter value for which a heteroclinic or homoclinic link exists; no entire diffeomorphism (including e.g. the standard map and any polynomial diffeomorphism) can have
a heteroclinic or homoclinic link \cite{Us80}. Still, we prove our main theorem by showing that stochastic islands appear near any elliptic point of an area-preserving diffeomorphism after a $C^\infty$-small perturbation.

\subsection{Strong regularity}
With the aim to extend the available examples of the non-uniformly hyperbolic behavior, Yoccoz launched a program called \emph{Strong Regularity} in his first lecture at Coll\`ege de France \cite{BY14}. The objective was to give a geometric-combinatorial definition of the non-uniformly hyperbolic dynamics which would serve both the one-dimensional (strongly dissipative) case, like e.g. in Jakobson theorem \cite{Ja81} and the 2-dimensional case (e.g. for the positive entropy  conjecture).  So far there are three examples of such dynamics: one-dimensional quadratic maps (e.g. implying Jakobson theorem) \cite{Yo95}, a non-uniformly hyperbolic horseshoe of dimension close to 6/10 \cite{PY09}, and H\'enon-like endomorphisms \cite{berhen} (implying Benedicks-Carleson Theorem \cite{BC2}). 

\subsection{Isotopy to identity and renormalization}
It is well-known that any symplectic diffeomorphism $F:\D\to \R^2$ is isotopic to identity \cite{Ch83}, which implies that it can always be represented as a composition $F=f_n\circ \cdots \circ f_1$ 
of $n$ symplectic diffeomorphisms $f_i$, each of which is uniformly $O(1/n)$-close to the identity map. Thus, one could try to prove the Herman's conjecture by using the 
Ruelle-Takens construction \cite{RT71}: take as $F$ in this formula an appropriate area-preserving map with a stochastic island, then consider $n$-disjoint $\varepsilon$-disks 
$\sqcup_i D_i=\psi_i(\D)$ inside $\D$ (where $\psi_i$ are uniform affine contractions) and take a perturbation $f$ of the identity map such that $f(D_i)=D_{i+1}$ and $f|_{D_i}$
is smoothly conjugate to $f_i$, i.e., $f|_{D_i}=\psi_{i+1}\circ f_i\circ \psi_i^{-1}$ for $i=1,\dots, n-1$, and
$f|_{D_n}=\psi_{1}\circ f_n\circ \psi_n^{-1}$. Then, $f^n|_{D_1}$ will be smoothly conjugate to $F$ and, hence, would have positive metric entropy. 
However, since the conjugacies $\psi_i^{-1}$ expand with a rate at least $\varepsilon^{-1}$, we observe that the $C^r$-norm of $f$ is then $\asymp 1/(n\varepsilon^r)$. 
As the $n$ discs $D_i$ are disjoint, it follows that $n \varepsilon^{2}\le 1$, so $f$ can, a priori, be $\asymp \varepsilon^{2-r}$ far, in the $C^r$-norm, from the identity. 

Therefore, this construction does not produce the result for $r\geq 2$, although one can create $C^1$-close to identity maps with positive metric entropy in this way (in fact, every area-preserving
dynamics can be realized by iterations of $C^1$-close to identity maps exactly by this procedure). In \cite{NRT64}, Newhouse-Ruelle-Takens pushed forward the argument to obtain the $C^2$-case for torus maps. 
Fayad \cite{Fayad} also proposed a trick to cover the $C^{2}$-case for disk maps, but his method does not work in the $C^r$-case if $r\ge 3$. 

We bypass the problem by using symplectic polynomial approximations of \cite{Tu03} instead of the isotopy. In this way, one $C^r$-approximates
any symplectic diffeomorphism $F$ by the product $f_n\circ \cdots \circ f_1$ of symplectic diffeomorphisms of a very particular form (H\'enon-like maps).
For these maps, the conjugating contractions $\psi_i$ can be made very non-uniform, allowing for an arbitrarily good approximation of every dynamics by
iterations of $C^r$-close to identity maps for all $r$, see \cite{T15}. The product $f_n\circ \cdots \circ f_1$ is only an approximation of $F$, and it is still 
not known if every area-preserving dynamics can be exactly realized by iterations of $C^r$-close to identity maps.
The main technical novelty of this paper is to show that some maps
with stochastic islands can.

\subsection{Stochastic sea and elliptic islands}
The main motivation for the positive metric entropy conjecture is the amazing complexity of dynamics of a typical area-preserving map.
Let us stress that no conservative dynamics are understood with certainty, except for those which are semi-conjugate
to a rotation \cite{AK70} or to an Anosov map. The reason is that hyperbolic and non-hyperbolic elements are often inseparable. 

Thus, it was discovered by Newhouse \cite{Ne68,Ne79} that a uniformly-hyperbolic Cantor set can be {\em wild}, i.e., its stable and unstable manifolds
can have tangencies, and these non-transverse intersections cannot all be removed by any $C^2$-small perturbation of the map.
Moreover, Newhouse showed \cite{Ne77} that a $C^r$-small perturbation of an area-preserving map with a wild set creates elliptic periodic orbits  
which accumulate to the wild hyperbolic set. 
 
Newhouse theory was applied and further developed by Duarte. He showed in \cite{Du94} that for all $a$
large enough the ``chaotic sea'' observed in the standard map (\ref{chirikov}) contains a wild hyperbolic set $K$,
and for a Baire generic subset of this interval of $a$ values the map has infinitely many generic elliptic periodic points, which accumulate 
on $K$. Recall that a generic elliptic point of period $k$ for a map $f$ is surrounded, in its arbitrarily small neighborhood, by uncountably 
many smooth circles (KAM-curves), invariant with respect to $f^k$. The map $f^k$ restricted to such curve is smoothly conjugate to
an irrational rotation (so the Lyapunov exponent is zero). The set occupied by the KAM curves has positive Lebesgue measure, and their density
tends to $1$ as the elliptic point is approached. An invariant curve bounds an invariant region that contains the elliptic point,
such regions are called \emph{elliptic islands}. 

In \cite{Duarte99, Duarte08}, Duarte showed that small perturbations, within the class ${\rm {\rm Diff}}^r_\omega$, of any 
area-preserving surface diffeomorphism with a {\em homoclinic tangency}
(the tangency of the stable and unstable manifolds of a saddle periodic orbit) 
lead to creation of a wild hyperbolic set and to infinitely many coexisting elliptic points (and elliptic islands).
In turn, it was shown in \cite{MR97, GT10} that near any elliptic point a homoclinic tangency to some saddle periodic orbit 
can be created by a $C^r$-small perturbation. 

Altogether, this gives a quite complicated picture of generic conservative dynamics: within the stochastic sea there are elliptic islands,
inside elliptic islands there are small stochastic seas, etc. By \cite{GST07,GT10}, it is impossible to describe such dynamics in full detail.
In fact, even most general features are presently not clear: for instance, we have no idea if the observed stochastic sea represents
a transitive invariant set, or if it has a positive Lebesgue measure (though, by Gorodetski \cite{Go12}, it may contain uniformly-hyperbolic subsets 
of Hausdorff dimension arbitrarily close to 2).

The inherent inseparability of the hyperbolic and elliptic behavior even suggests the following provocative question, communicated 
to us by Fayad.
\begin{question}\label{qfayad}
Does an open set of area-preserving $C^\infty$-diffeomorphisms exist
with the following property: for each diffeomorphism belonging to this subset the complement to the union of the KAM curves has zero Lebesgue measure?
\end{question}

Maps with this property have zero metric entropy, so our Theorem A implies the negative answer to this question (a KAM curve is always a limit of non-hyperbolic periodic points).
However, it is still possible that a
generic (i.e., belonging to a countable intersection of open and dense subsets) non-hyperbolic map from ${\rm Diff}^\infty_\omega$ has zero metric entropy. 
 
In \cite{AB16}, the complement $U$ to the set of all essential invariant curves of a symplectic twist map was considered. Any connected component of $U$
(the Birkhoff ``instability zone'') is bounded by two invariant topological circles $C_1$ and $C_2$. It is shown in \cite{AB16}, that either $C_i$ 
 is a heteroclinic link (a scenario as much improbable as exhibiting a stochastic island, e.g., it is impossible for entire maps), or the Lyapunov 
exponent of any invariant measure supported by $C_i$ is zero ($i=1,2$). By the results of Furman \cite{Fu97}, the latter alternative implies that the 
convergence\footnote{
By the works of Birkhoff, Mather, and Le Calvez  \cite{Bi24, Ma91, LeC87}, there exists an orbit whose $\alpha$-limit set is in $C_1$ and the 
$\omega$-limit set is in $C_2$.} of any orbit in the instability zone to one of these curves $C_i$ is at most sub-exponential. 
Thus, it is hard to see how a transitive invariant set with strictly positive maximal Lyapunov exponents can have $C_i$ in its closure,
if $C_i$ is not a heteroclinic link.
 
The twist property is fulfilled near a generic elliptic point, hence the results of \cite{AB16} hold true there. Therefore, it seems probable
that if the Sinai conjecture is correct, then
the positive metric entropy is achieved by sets distant from KAM curves. This seems to be consistent with the numerical observations \cite{Me94}.

\subsection{Stochastic perturbation of the standard map}
In higher dimension, more possibilities exist for creating examples with positive metric entropy. Thus, it was shown in \cite{BC14}
that, for large values of the parameter $a$, a skew product of the standard map over an Anosov map is non-uniformly hyperbolic
and displays non-zero Lyapunov exponents for Lebesgue almost every point. 
Recently, Blumenthal-Xue-Young \cite{BXY17} used a similar argument for random perturbations of the standard map with
large $a$ and also showed the positivity of metric entropy. 

\subsection{Genericity results}
A recent breakthrough by Irie and Asaoka \cite{AI16} showed, from a cohomological argument, that for any closed surface $(\M,\omega)$
a generic map from ${\rm {\rm Diff}}^\infty_\omega(\M)$ has a dense set of periodic points. Hence, if such map is weakly-stable, then it 
has a dense set of hyperbolic periodic points. A natural conjecture is, then, the structural stability of weakly-stable maps from ${\rm {\rm Diff}}^\infty_\omega(\M)$; 
this would be a counterpart of the ``Lambda lemma'' from holomorphic dynamics \cite{MSS, Ly84, LD13,BD14}. 

Another natural conjecture would be that the weakly-stable maps from ${\rm {\rm Diff}}^r_\omega(\M)$ are uniformly hyperbolic, $1\le r\le \infty$. For $r=1$ this result
have been proven by Newhouse \cite{Ne77}. For any $r\ge 2$ this question is open, as well as its dissipative counterpart -- a conjecture by Ma\~n\'e \cite{Ma82}, which is also proven only for $r=1$ \cite{Ao92}. Since uniformly hyperbolic maps from ${\rm {\rm Diff}}^r_\omega(\M)$ have positive metric entropy, this conjecture and our Theorem A
would imply that maps with positive metric entropy are dense in ${\rm {\rm Diff}}^\infty_\omega(\M)$.

Because of the meagerness of the heteroclinic links, the genericity of positive metric entropy does not follow from our result. In fact, one
can conjecture that a $C^r_\omega$-generic surface diffeomorphism is either uniformly hyperbolic, or of zero entropy. We do not have an opinion in this regard.
In the $C^1$-topology, this statement was a conjecture by Ma\~n\'e, now proven by Bochi \cite{B02}; in higher regularity it is completely open. 

A milder version of this problem can be formulated as the following question due to Herman \cite{He98}:
\begin{question}\label{eli}
Given a surface $(\M,\omega)$, is there an open subset of ${\rm Diff}^r_\omega(\M)$ where maps with zero metric entropy
are dense?
\end{question}
A candidate for such dense set could be a hypothetical set of maps from Question \ref{qfayad} (the maps for which the union of
all KAM curves would have full Lebesgue measure). Note that by the upper semi-continuity of the maximal Lyapunov exponent, 
a positive answer to Question \ref{eli} would also imply the local genericity of maps with zero metric entropy.

\subsection{Universal dynamics}
In \cite{Tu03}, the richness of chaotic dynamics in area-preserving maps was characterized by the concept of a {\em universal map}.
Given a $C^r_\omega$-diffeomorphism $f$ ($r=1,\dots,\infty$) of a two-dimensional surface $(\M,\omega)$, 
its behavior on ever smaller spatial scales
can be described by its {\em renormalized iterations} defined as follows. Let $Q$ be a $C^r$-diffeomorphism into $\M$ from some disc in 
$\R^2$. Assume that the domain of definition of $Q$ contains the unit disc $\D$ and the domain of $Q^{-1}$ in $\M$ contains
$f^n(Q(\D))$ for some $n\geq 0$. We also assume that the Jacobian $\det\; DQ$ is constant in the chart $(x,y)$ on $\M$ where the area-form $\omega$ is standard: $\omega=dx\wedge dy$.
\begin{defi}\label{renit}
The map $\D \to\R^2$ defined as
\[\hat F_{Q,n} = Q^{-1}\circ f^n|_{Q(\D)}\circ Q\]
is a renormalized iteration of $f$.
\end{defi}
Note that since the Jacobian of $Q$ is constant, all renormalized iterations of $f$ preserve the standard area-form in $\R^2$. 
\begin{defi}[Universal map]
A diffeomorphism $f\in {\rm {\rm Diff}}^\infty_\omega(\M)$ is universal if the set of its renormalized iterations is $C^\infty$-dense
among all orientation-preserving, area-preserving diffeomorphisms $\D \to\R^2$.
\end{defi}
By this definition, the dynamics of a single universal map approximate, with arbitrarily good precision, all symplectic maps of the unit disc.

In the general non-conservative context this notion was used in \cite{Tu10,T15}. In $C^1$ category, the concept of universal dynamics was independently proposed by Bonatti and Diaz \cite{BD02}. 

The universal dynamics might sound difficult to materialize, but it is not. It is shown in \cite{GST07} that an arbitrarily small, in $C^\infty_\omega$,  perturbation of any area-preserving map with a homoclinic tangency can create universal dynamics. Moreover, universal maps form a Baire generic subset of the Newhouse domain - the open set in ${\rm Diff}^\infty_\omega(\M)$ comprised of maps with wild hyperbolic sets. In \cite{GT10}, it
was shown that a $C^\infty_\omega$-generic diffeomorphism of $\M$ with an elliptic point is universal\footnote{The results in \cite{GST07,GT10} were also proven in the space of real-analytic area-preserving maps.}.

Consequently, any diffeomorphism $f\in {\rm Diff}^\infty_\omega(\M)$ with an elliptic point can be perturbed in such a way that its iterations
would approximate any given area-preserving dynamics and, in particular, the dynamics with positive metric entropy. We stress that this
observation {\em is not sufficient for a proof of our Theorem A}. Indeed, if $f$ is a $C^\infty_\omega$-diffeomorphism with an elliptic point and $g$ is a $C^\infty_\omega$-diffeomorphism of $\D$ with positive metric entropy, the only thing we can conclude from \cite{GST07,GT10} 
is that arbitrarily close to $f$ in ${\rm Diff}^\infty_\omega(\M)$ there exists a diffeomorphism whose iteration restricted to a certain disc is smoothly
conjugate to a map $G$ which is as close as we want to $g$ in ${\rm Diff}^\infty_\omega(\D)$. However, this map $G$ does not need to inherit
the positive metric entropy from $g$. Overcoming this problem is the main technical point of this paper.

\section{Proof of the main Theorem}\label{section0}

We start with constructing a map with a stochastic island with certain additional properties.
In section \ref{section1}, we give a precise description of the construction similar to those in \cite{Ka79, AP09,Pr82}, which
produces a $C^\infty_\omega$-diffeomorphism $\check F: \D\to\D$ with a stochastic island $\mathcal I$  bounded by four heteroclinic 
bi-links $\{\check L_i^a\cup \check L^b_i: 0\le i\le 3\}$. Each $\check L_i^a\cup \check L_i^b$ is a $C^\infty$-embedded circle included
in the stable and unstable manifolds of hyperbolic fixed points   $\check P_i, \check Q_i$:
\[\check L_i^a\cup \check L_i^b \subset   W^u(\check P_i; \check F)\cup W^s(\check Q_i; \check f)\; .\]
The island of $\check F$ is depicted in Fig. \ref{fig:island}. 
For every $F$ which is $C^1$-close to $\check F$, for every $0\le i\le 3$, the \emph{hyperbolic continuations} of $\check P_i$
and $\check Q_i$  are the uniquely defined hyperbolic $F$-periodic orbits close to $\check P_i$ and $\check Q_i$. 

We also show the following
\begin{prop}\label{KPAP}
For every conservative map $F$ which is $C^2$-close to $\check F$, let $P_i$ and $Q_i$ be hyperbolic
continuations of $\check P_i, \check Q_i$. If $\{W^u(P_i; F)\cup W^s(Q_i;F): 0\le i \le 3\}$ define four
heteroclinic bi-links $\{L_i^a\cup  L^b_i: 0\le i\le 3\}$ which are $C^2$-close to $\{\check L_i^a\cup  \check L^b_i: 0\le i\le 3\}$, then they bound a stochastic island. In particular, the metric entropy of $F$ is positive.  
\end{prop}
The proof is given by Corollary \ref{remf2} of Theorem \ref{claim7} in Section \ref{section1}. It follows from a new and short argument which implies also some results of \cite{AP09}.

We will call a stochastic island {\em robust relative link preservation} if it satisfies this property: for every $C^r$-small perturbation of the map, if the stable and unstable manifolds that form the heteroclinic link do not split, then they bound a stochastic island. Proposition \ref{KPAP} shows that
the stochastic island $\check I$ of the map $\check F$ satisfies this robustness property (with any $r\geq 2$), and the same holds true
for the islands bounded by the four heteroclinic bi-links $\{L_i^a\cup  L^b_i: 0\le i\le 3\}$ (if these links exist) of any map $C^2$-close to $\check F$.

In Section \ref{section2} (see Proposition \ref{suitpro}), we construct a coordinate transformation 
$\mathring \phi\in {\rm Diff}^\infty_\omega(\R^2)$ such that
$\mathring {\mathcal I}:= \mathring \phi(\mathcal I)$ is a ``\emph{suitable}'' island for $\mathring F:= \mathring \phi\circ \check F\circ \mathring \phi^{-1}$.
The suitability conditions are described in Definition \ref{suitabledef}. They include the requirement that
certain segments of the stable and unstable manifolds of the hyperbolic fixed points $\mathring P_i:= \mathring \phi(P_i)$ and $\mathring Q_i=\mathring \phi(Q_i)$, $0\le i\le 3$,
are strictly horizontal, i.e., they lie in the lines $y=const$ where $(x,y)$ are coordinates in $\R^2$. 
Moreover, the map $\mathring F$ near these segments has a particular form, which allows us to establish the following result
in Section \ref{section3} (this is the central point of our construction):
\begin{prop}\label{closing_bi-link}
Given any finite $r\ge 2$, for every $F\in {\rm Diff}^{r+8}_\omega(\D)$ which is $C^{r+8}$-close to $\mathring F$, there exists a $C^r$-small function $\psi: \R\to\R$
such that the map $\bar F$ defined as 
$$\bar F= S_\psi \circ F\;, \;\mbox{    where     } S_\psi:= (x,y)\mapsto (x,y+\psi(x))\;,$$
has the following property: For the hyperbolic continuations $P_i$ and $Q_i$ of the fixed points $\mathring P_i$ and, respectively,
$\mathring Q_i$, the union $W^s(P_i; \bar F)\cup W^u(Q_i; \bar F)$ 
defines a heteroclinic bi-link $L_i^a\cup L_i^b$ which is $C^r$-close to $\mathring L_i^a\cup \mathring L_i^b$, for each $i=0,\dots,3$.
\end{prop}

\begin{rema}\label{postbilink} We notice that by Proposition \ref{KPAP}, the map $\bar F=S_\psi\circ F$ has
a stochastic island, robust relative link preservation, and its metric entropy is positive.
\end{rema}

With this information, we can now complete the proof of the main theorem.

\noindent{\bf Proof of Theorem \ref{main}.}
Let $f\in {\rm Diff}^\infty_\omega(\M)$ have a non-hyperbolic periodic point. By an arbitrarily small perturbation of $f$ one can make this point
elliptic. 
Then, by \cite{GT10}, by a $C^\infty_\omega$-small perturbation of
$f$, one can create, in a neighborhood of the elliptic point, a
hyperbolic periodic cycle whose stable and unstable manifolds coincide (those define two heteroclinic links), see Fig. \ref{elband}.

\begin{figure}[h!]
	\centering
		\includegraphics[width=5cm]{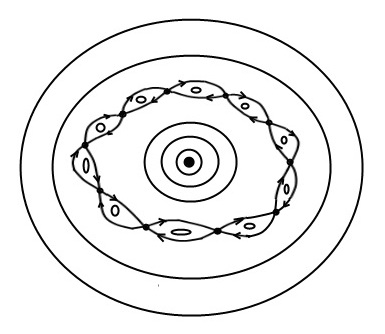}
\caption{
Construction of a flat homoclinic tangency from \cite{GT10}: Given an area-preserving map with an elliptic periodic point $O$,
one can add a $C^\infty$-small perturbation such that the first-return map in a small neighborhood of $O$ will be integrable.
A change of rotation number at $O$ leads to the birth of a resonant garland where the stable and unstable manifolds of a hyperbolic $q$-periodic point coincide.}\label{elband}
\end{figure}
The important thing here is that by a small perturbation of the original map $f$, we create a periodic point with a flat homoclinic tangency.
After that, we apply the following result, proven in Section \ref{section5} (see the proof after Corollary \ref{qq}).

\begin{prop}\label{proprescalling}
Let $f\in {\rm Diff}^{\infty}_{\omega} (\M)$ have a hyperbolic periodic point with a flat homoclinic tangency. Then, there exists a $C^\infty$-dense (residual)
subset $\cal F$ of ${\rm Diff}^\infty_\omega(\D,\R^2)$ such that for every $F\in {\cal F}$, every $r\geq 2$, every $C^r$-smooth function $\psi:\R\to\R$,
and every $\varepsilon>0$, there exists a diffeomorphism $\hat f\in {\rm Diff}^r_{\omega} (\M)$ such that\\ 
$\bullet$ $d_r(f,\hat f) <\varepsilon$, where $d_r$ is the $C^r$-distance,\\
$\bullet$ the composition $S_\psi\circ F$ is equal to a renormalized iteration of $\hat f$, where $S_\psi(x,y)= (x,y+\psi(x))$. 
\end{prop}
This statement is an enhanced version of 
the ``rescaling lemma'' (Lemma 6) of \cite{GST07}\footnote{While in \cite{GST07} the rescaling was done for a single round near the homoclinic tangency, here we do many rounds, similar to \cite{T15}.}.

This proposition gives us much freedom in varying the renormalized iteration of $\hat f$ without perturbing $F$ (by composing with $S_\psi$ for an arbitrarily functional parameter $\psi$). 

The renormalized iterations are described by Definition \ref{renit}. Since a renormalized iteration is $C^r$-conjugate to an actual iteration of
the map $\hat f$ restricted to some small disc, it follows that by taking $F$ and $\psi$ exactly as in Proposition \ref{closing_bi-link}, (so that
$S_\psi\circ F$ will have a stochastic island, see Remark \ref{postbilink}), we will obtain that the map $\hat f$ has a stochastic island too, 
robust relative link preservation.

The stochastic island for the map $\hat F:= S_\psi\circ F$ is bounded by $4$-bi-links, each of which is equal to a $C^r$-embedded circle. As $S_\psi\circ F$
is smoothly conjugate to $\hat f^n$ for some $n>0$, it follows that the stochastic island for the map $\hat f$ is bounded $m=4n$ heteroclinic bi-links and is
robust relative link preservation. We denote them by $(L_i^a\cup L_i^b)_{i=1}^{m}$.  We prove the
following result in Section \ref{proofrtoinf}:

\begin{prop}\label{rtoinf}
Given any map $\hat f\in {\rm Diff}^r_\omega(\M)$ with a stochastic island $\mathcal I$ bounded by bi-links $(L^a_i\cup L_i^b)_{i=1}^m$ such that each bi-link $L^a_i\cup L_i^b$
is a $C^r$-embedded circle, arbitrarily close in $C^r$ to $\hat f$ there exists a map
$\hat f_\infty\in {\rm Diff}^\infty_\omega(\M)$ for which the bi-links persist (i.e., the hyperbolic continuations of the stable and unstable manifolds forming each bi-link $(L_i^a, L_i^b)$ 
comprise a heteroclinic bi-link for the map $\hat f_\infty$, $C^r$-close to $(L_i^a, L_i^b)$ ).
\end{prop}
The map $\hat f$ with the stochastic island lies in the $\varepsilon$-neighborhood of the original map $f$ in ${\rm Diff}^r_\omega(\M)$.
Since the stochastic island of the map $\hat f$ is robust relative link preservation, the map $\hat f_\infty$ also has a stochastic island and, hence, positive metric entropy.

This shows, that arbitrarily close, in $C^r$ for any given $r$, to the original map $f$ there exists a map $\hat f_\infty\in {\rm Diff}^\infty_\omega(\M)$
with positive metric entropy.
$\qed$

\section{Stochastic island}\label{section1}
In this Section we describe a particular example of a stochastic island. It is somewhat similar to the Przytycki's development of the Katok's construction,
in the sense that the holes in the island are bounded by heteroclinic links. A difference with the Przytycki's example is that
the heteroclinic links form smooth circles in our construction. Similar examples were considered by Aubin-Pujals \cite{AP09} in the
non-conservative case.

\subsection{An Anosov map of the torus}\label{Example}
Let $\T^2$ be the torus $\R^2/\Z^2$. Let $\S^1$ be the circle $\R/(2\pi \Z)$. 
We endow $\R^2$ and $\T^2$ with the symplectic form $\omega= dx\wedge dy$. 

Consider the following linear Anosov diffeomorphism of $\T^2$:
\begin{equation}\label{fancart}
F_A: (x,y)\mapsto A (x,y) := (13x+8y,8x+5y)
\end{equation}
($F_A$ is the third iteration of the standard Anosov example $(x,y)\mapsto (2x+y,x+y)$). The map $F_A$
preserves the area form $\omega$ and is uniformly hyperbolic, e.g. its Lyapunov exponents are non-zero.

The map $F_A$  has four different fixed points $\Omega_0=(0,0)$, $\Omega_1=(\frac12,\frac12)$, 
$\Omega_2=(\frac12,-\frac12)$, $\Omega_3=(-\frac12,\frac12)$. Let $\sigma>0$ be the logarithm of the unstable eigenvalue of the matrix $A$: 
$\sigma=\ln (9+4\sqrt{5})$. We may put the origin of coordinates to one of the points $\Omega_i$ and make a symplectic linear transformation of the
coordinates $(x,y)$ such that the map $F_A$ near $\Omega_i$ will become
\[\Omega_i+(x,y)\mapsto \Omega_i+(e^{\sigma} x,\; e^{-\sigma} y)\; .\]
Observe that $F_A$ is the time-$\sigma$ map by the flow of the system 
\[\dot x = \partial_y H_i(x,y), \qquad \dot y = -\partial_ x H_i(x,y)\]
associated with the Hamiltonian
\[ H_i = x y \; .\]

Note that the transition to polar coordinates $(\rho,\theta)$ near $\Omega_i$ by the rule
\begin{equation}\label{symppol}
(x,y) = (\sqrt{2 \rho} \cos(\theta), \sqrt {2 \rho} \sin(\theta) )
\end{equation}
preserves the symplectic form, i.e., we have $\omega=d\rho\wedge d\theta$. The map $F_A$
in these coordinates is the time-$\sigma$ map by the flow defined by the Hamiltonian function
\begin{equation}\label{hamri}
H_i(\theta, \rho) = \rho \sin(2\theta)\;.
\end{equation}
The corresponding Hamiltonian vector field is
\[\frac{d}{dt}\rho= 2\rho \cos(2\theta), \qquad \frac{d}{dt}\theta = - \sin(2\theta).\]
We do not need an explicit expression for the map $F_A$ in the symplectic polar coordinates; just note that
near the point $\Omega_i$ this map has the form $(\rho,\theta)\mapsto(\bar \rho, \bar\theta)$ where
\begin{equation}\label{fapol}
\bar\rho=\rho\; p_0(\theta), \qquad \bar\theta=q_0(\theta),
\end{equation}
and $p_0$ and $q_0$ are $C^\infty$-functions $\S^1\to\R$. Since this map preserves the symplectic form 
$d\rho\wedge d\theta$, its Jacobian $p_0(\theta) \partial q_0(\theta)$ equals to $1$, so $p_0(\theta)\neq 0$ and
$\partial q_0(\theta)\neq 0$.

\subsection{Stochastic island in $\T^2$}
In order to construct a chaotic island, we shall ``blow up'' the four points $\Omega_i$. This means that we will 
take some small $\delta>0$, consider the closed $\delta$-discs $V_i$ with the center at $\Omega_i$ for each $i=0,1,2,3$,
and build a $C^\infty$-diffeomorphism $\Psi$ of $\T^2\setminus \{\sqcup_{0\leq i\leq 3} V_i\}$
onto $\T^2\setminus \{\cup_i \Omega_i\}$. We will do it in such a way that the map $\Psi^{-1}\circ F_A\circ \Psi$ will be smoothly extendable 
to a $C^\infty$-diffeomorphism $\hat F$  
 of $\T^2$. Then the invariant set 
$\hat I=\T^2\setminus \{\sqcup_i V_i\}$ will be a stochastic island for $\hat F$. Moreover, the points $\Omega_i$ will be flat fixed points of
$\hat F$ and, importantly, $\hat F$ will inherit the symmetry with respect to $(-id)$ from the map $F_A$ -- this will be used at the next step in Section \ref{t2s}.

Let $\varepsilon>0$ be such that the map $F_A$ is the time-$\sigma$ map of the Hamiltonian flow defined by (\ref{hamri}) 
in the closed $\varepsilon$-disc $V_i'$ about $\Omega_i$, $i=0,1,2,3$;
we assume that the discs $V_i'$ are mutually disjoint. Let $0<\delta<\varepsilon$ and let $V_i\Subset V_i'$ be the closed $\delta$-discs about $\Omega_i$, $i=0,1,2,3$.

Let $\displaystyle\psi: \left[\frac{\delta^2}2, \frac{\varepsilon^2}2\right] \to \left[0,\frac{\varepsilon^2}2\right]$ be a 
 $C^\infty$-diffeomorphism such that
\begin{equation}\label{smallpsi}
\psi(\rho)= \rho - \frac{\delta^2}2  \quad \text{if $\rho$ is close to }\frac{\delta^2}2\quad \text{and} \quad \psi(\rho)  =\rho  \quad \text{if $\rho$ is close to } \frac{\varepsilon^2}2\; .
\end{equation} 
Let $\Psi_i\in C^\infty(\T^2\setminus int(V_i),\T^2)$ be equal to the identity outside $V_i'$ and let the restriction of $\Psi_i$
to the smaller disc $V_i$ be given by
 \begin{equation}\label{psilarge}\Psi_i|_{V_i}: (\theta, \rho) \mapsto (\theta, \psi(\rho))\end{equation}
in the polar coordinates (\ref{symppol}). 
The radius-$\delta$ circle $\partial V_i$ about $\Omega_i$ is sent by $\Psi_i$ to $\Omega_i$. Note that $\Psi_i$ is 
a diffeomorphism from $\T^2\setminus V_i$ onto $\T^2\setminus \Omega_i$ and $\Psi_i$ commutes with $(-id)$. 
Note also that in a neighborhood of $\partial V_i $ the map $\Psi_i$ preserves the form $\omega=d\rho\wedge d\theta$.

Define 
\begin{equation}\label{surgery}
\text{$\Psi=\Psi_i$ in $V_i'\setminus V_i$ for $i=0,1,2,3$ and $\Psi=id$ in $\T^2\setminus \{\sqcup_i V_i'\}$.}
\end{equation} 
This map is a $C^\infty$-diffeomorphism from $\hat I=\T^2\setminus \{\sqcup_i V_i\}$ onto $\T^2\setminus \{\cup_i\Omega_i\}$
and it commutes with $(-id)$.

Denote $\hat F=\Psi^{-1}\circ F_A\circ \Psi$. By construction, this is a $C^\infty$-diffeomorphism of 
$\hat I$, which commutes with $(-id)$.
In a small neighborhood of the circles $\partial V_i$ the map $\hat F$ preserves symplectic form $d\rho\wedge d\theta$ and, by 
(\ref{hamri}),(\ref{smallpsi}),(\ref{psilarge}), it coincides in this neighborhood with the time-$\sigma$ map of 
the flow defined by the symplectic form $d\rho\wedge d\theta$ and the Hamiltonian
\[\hat H_i=(\rho-\delta^2/2)\sin(2\theta)\;.\]
We can, therefore, smoothly extend $\hat F$ inside $V_i$ (i.e., to $\rho\leq \delta^2/2$) as the time-$\sigma$ map of the flow 
defined by the smoothly extended Hamiltonian $\hat H_i$:
\begin{equation}\label{hamxi}\hat H_i=(\rho-\delta^2/2)\sin(2\theta) \xi(\rho)\;,\end{equation}
where $\xi$ is a $C^\infty$-function, equal to zero at all $\rho$ close to zero and equal to $1$ at all $\rho\ge \delta^2/2$.

We summarize some relevant properties of the map $\hat F$ in the following
\begin{prop}\label{propfhat}
The map $\hat F$ is a $C^\infty$-diffeomorphism of $\T^2$ such that:
\begin{enumerate}
\item $\hat F|_{\T^2\setminus \{\sqcup_i V_i\}}$ is conjugate to $F_A|_{\T^2\setminus \{\cup_i \Omega_i\}}$ via 
a $C^\infty$-diffeomorphism $\Psi$;
\item the set $\hat I:=\T^2\setminus \{\sqcup_i V_i\}$ is invariant with respect to $\hat F$; 
\item $\hat F$ preserves a smooth symplectic form $\hat \omega$;
\item  $\hat F$ commutes with $(-id)$, and $\hat \omega$ is invariant with respect to $(-id)$;
\item $\hat F$ equals to the identity in a small neighborhood of the points $\Omega_i$;
\item each circle $\partial V_i$ is a heteroclinic $4$-link.
\end{enumerate}
\end{prop}

\begin{proof} Claim 1 is given just by construction of $\hat F$. Claim 2 follows from it by continuity of $\hat F$:
since $\Omega_i$ are fixed points of $F_A$, each disc $V_i$ is invariant with respect to $\hat F$. Claim 5 follows since $\hat H_i$ is constant near $\Omega_i$, so
the corresponding vector field is identically zero there.

Claim 3: By claim 1, the map $\hat F$ preserves the symplectic form $\hat \omega=\Psi^*\omega$ in 
$\T^2\setminus \{\sqcup_i V_i\}$. Since $\Psi^*\omega=d\rho\wedge d\theta=\omega$ near $\partial V_i$ (as it follows from (\ref{smallpsi}),(\ref{psilarge})),
we can smoothly extend $\hat \omega$ onto the whole torus by putting $\hat\omega =\omega$ in $\sqcup_i V_i$. 
Since $\hat F|_{V_i}$ is the time-$\sigma$ map by a Hamiltonian flow, the form $\hat\omega=\omega$ inside the discs $V_i$ is preserved by $\hat F$.

Claim 4 follows since both the original map $F_A$ and the conjugacy $\Psi$ commute with $(-id)$, 
the Hamiltonians $\hat H_i$ that defines $\hat F$ inside the discs $V_i$ (see (\ref{hamxi})) are invariant with respect to $(-id): \theta\mapsto\theta+\pi$,
and the symplectic form $\omega$ is invariant with respect to $(-id)$. 

In order to prove claim 6, notice that by (\ref{hamxi}) the map $\hat F$ near $\partial V_i: \{\rho=\delta^2/2\}$
is the time-$\sigma$ map of the system
\begin{equation}\label{sys6}
\frac{d}{dt} \rho= 2 (\rho-\delta^2/2)\cos(2\theta), \qquad \frac{d}{dt} \theta= -\sin(2\theta).
\end{equation}
This system has 4 saddle equilibria on the circle $\rho=\delta^2/2$: $\theta=0, \pi/2, \pi, 3\pi/2$. These equilibria are 
hyperbolic fixed points of $\hat F$, and the invariant arcs of the circle $\rho=\delta^2/2$ between these points are
formed by their stable or unstable manifolds.
\end{proof}

\subsection{The island is robust relative link preservation}

The following statement establishes that the set $\hat I$ is a stochastic island for $\hat F$. It also concerns 
the dynamics of perturbations of $\hat F$. Similar results were obtained by Aubin-Pujals in \cite{AP09} for the non-conservative case. 
The proof is given by a new and shorter argument.

\begin{theo}\label{claim7}
For the map $\hat F$, as well as for every, not necessarily conservative, diffeomorphism $\tilde F$ which is $C^2$-close to 
$\hat F$ and keeps the circles
$\partial V_i$ invariant ($i=0,1,2,3$), all points in $\hat I$ have positive maximal Lyapunov exponent. The maps 
$\tilde F$ and $\hat F$ are topologically conjugate on $\hat I$; all such maps are transitive.
\end{theo}

\begin{proof}
By (\ref{fapol}),(\ref{smallpsi}),(\ref{psilarge}), the map $\hat F$ near $\partial V_i$ can be written as 
$(\rho,\theta)\mapsto (\bar\rho,\bar\theta)$ where
$$\bar\rho=\frac{\delta^2}{2} +(\rho-\frac{\delta^2}{2}) p_0(\theta), \qquad \bar\theta=q_0(\theta)\;.$$
A $C^2$-small perturbation $\tilde F$ of $\hat F$ which keeps the circle $\partial V_i$ invariant must 
send $\rho=\frac{\delta^2}{2}$ to $\bar \rho=\frac{\delta^2}{2}$, so it has the form
$$\bar\rho=\frac{\delta^2}{2} +(\rho-\frac{\delta^2}{2}) (p_0(\theta)+p(\theta,\rho)), \qquad 
\bar\theta=q_0(\theta)+q(\theta,\rho)\;,$$
where the function $p$ is $C^1$-small and $q$ is $C^2$-small.

By reversing our surgery (\ref{surgery}), we obtain a diffeomorphism $\tilde F_A=\Psi\circ \tilde F\circ\Psi^{-1}$
of $I_A:=\T^2\setminus \cup_i \{ \Omega_i\}$, which takes the following form near $\Omega_i$ (see (\ref{smallpsi}),(\ref{psilarge})):
$$\bar\rho=\rho \cdot (p_0(\theta)+p(\theta,\rho+\frac{\delta^2}{2})), \qquad \bar\theta=q_0(\theta)+q(\theta,\rho+\frac{\delta^2}{2})\;.$$

The following Lemma enables us to compare $\tilde F_A$ with $F_A$ defined in (\ref{fancart}).
\begin{lemm}\label{closemap}
In the Cartesian coordinates, the restrictions 
$\tilde F_A$ and $F_A$ to $ I_A=\T^2\setminus \cup_i \{\Omega_i\}$ are uniformly $C^1$-close.
\end{lemm}
\begin{proof}
The transformation
$(\rho,\theta)\mapsto (x,y)=(\sqrt{2\rho}\cos\theta, \sqrt{2\rho}\sin\theta)$ to Cartesian coordinates near
$\Omega_i$ has the following property
\begin{equation}\label{derpolxy}
\|\partial_{(x,y)}\rho\|\leq \sqrt{2\rho}, \qquad \|\partial_{(x,y)}\theta\|\leq \frac{1}{\sqrt{2\rho}}.
\end{equation}
Thus, the map $\tilde F_A$
near $\Omega_i$ takes the form $(x,y)\mapsto (\bar x, \bar y)$ where
$$\bar x= \sqrt{2\rho} \sqrt{p_0+p} \cos(q_0+q), \qquad\bar y= \sqrt{2\rho} \sqrt{p_0+p} \sin(q_0+q).$$
The uniformly-hyperbolic map $F_A$ is given by
$$\bar x= \sqrt{2\rho} \sqrt{p_0} \cos(q_0), \qquad\bar y= \sqrt{2\rho} \sqrt{p_0} \sin(q_0)$$
(see (\ref{fapol})). Recall that $p_0(\theta)\neq 0$ for all $\theta$. 

It follows that, $\tilde F_A(x,y)=F_A(x,y)+\nu \sqrt{2\rho}\;\phi(\rho,\theta)$ near $\Omega_i$ where
$\nu=\|(p,q)\|_{C^1}\sim \|\tilde F-\hat F\|_{C^2}$ is small and
$\phi$ is uniformly bounded along with its first derivatives with respect to $\rho$ and $\theta$. 
By (\ref{derpolxy}), this gives us that near the points $\Omega_i$
$$\|\partial_{(x,y)} (\tilde F_A-F_A)\|=\nu \left\|\frac{\partial_{(x,y)} \rho}{\sqrt{2\rho}}\;\phi+
\sqrt{2\rho}\; \partial_{(\rho,\theta)} \phi \;\partial_{(x,y)}(\rho,\theta)\right\|=O(\nu).$$
\end{proof}
This lemma implies the following:
\begin{lemm}
For the map $\tilde F$, every $z\in \hat I$  displays a positive Lyapunov exponent.
\end{lemm}
\begin{proof}
We have found that the map $\tilde F_A :I_A \to I_A$ is uniformly close in $C^1$ to the uniformly-hyperbolic map $F_A$ 
in a neighborhood of the fixed points $\Omega_i$ (even though the derivative of $\tilde F_A$ may be not defined at the points $\Omega_i$). 
Since the 4 fixed points $\Omega_i$ are the only singularities of the surgery transformation $\Psi$, the derivative of 
$\tilde F_A$ is uniformly close to the derivative of $F_A$ everywhere on $I_A$, i.e., $D\tilde F_A$ is uniformly close
to $(x,y)\mapsto A (x,y) = \left(\begin{array}{cc} 13 & 8 \\ 8 & 5\end{array}\right) (x,y)$ (see (\ref{fancart})). 
In particular, $D\tilde F_A$ takes every vector with positive coordinates to a vector with positive coordinates and at least 4 times larger norm. Hence, 
\begin{equation}\label{estfal}
\|D\tilde F_A^n(P)\| \geq 4^n
\end{equation}
for every point $P\in I_A$.
 
The map $\tilde F: \hat I\to \hat I$ is smoothly conjugate to $\tilde F_A=\Psi\circ \tilde F\circ\Psi^{-1}$, so
\[\|D\tilde F^n(\Psi^{-1}(P))\|\geq \|D\tilde F_A^n(P)\|/(\|D\Psi(\tilde F_A^n(P)\| \cdot \|D\Psi^{-1}(P)\|)\]
for every $P\in I_A$. If the point $P$ is chosen such that the iterations $\tilde F^n(\Psi^{-1}(P))$ do not converge to $\cup_i \partial V_i$,
then there is a sequence $n_j\to+\infty$ such that the iterations $\tilde F_A^{n_j}(P)$ stay away
from the points $\Omega_i$ - the only singularities of the conjugacy map $\Psi^{-1}$. Thus, both $\|D\Psi(\tilde F_A^{n_j}(P)\|$
and $\|D\Psi^{-1}(P)\|$ are bounded away from zero in this case. It follows then from (\ref{estfal}) that
$$\limsup_{n\to \infty} \frac1{n} \log \|D\tilde F^{n}(\Psi^{-1}(P))\|\geq \limsup_{n_j\to \infty} \frac1{n_j} \log \|D\tilde F_A^{n_j}(P)\|\geq \ln 4\; .$$
By definition, this means that the maximal Lyapunov exponent of $\Psi^{-1}(P)$ for the map $\tilde F$ is strictly positive.

In the remaining case, if all iterations of the point $\Psi^{-1}(P)$ by $\tilde F$ converge to $\cup_i \partial V_i$, they must converge
to one of the saddle fixed points that lie in $\cup_i \partial V_i$ (since every $\hat F$-pseudo-orbit which remains close to $\partial V_i$, 
is necessarily eventually close to one of the saddle point of $V_i$).  In this case, the maximal Lyapunov exponent
of $\Psi^{-1}(P)$ equals to the maximal Lyapunov exponent of the saddle point, i.e., it is positive. Thus, in any case, every point of
$\hat I$ has positive maximal Lyapunov exponent for the map $\tilde F$.
\end{proof}
By the uniform hyperbolicity of $F_A$ (see Lemma \ref{closemap}), using a variation of the Moser's technique \cite{Mo69} of the proof 
of Anosov structural stability theorem \cite{An67}, we prove:
\begin{lemm}\label{structstab}
There exists a homeomorphism $h$ of $\T^2$ which conjugates $F_A$ and $\tilde F_A$, and leaves each $\Omega_i$ invariant.   
\end{lemm}
This lemma implies the  topological conjugacy between $\tilde F|_{\hat I}$ and $\hat F|_{\hat I}$. The transitivity of $\tilde F|_{\hat I}$ follows 
from the transitivity of $F_A|_{\T^2\setminus \{\cup_i \Omega_i\}}$ by the conjugacy. This completes the proof of Theorem \ref{claim7}. 
\end{proof}
\begin{proof}[Proof of Lemma \ref{structstab}]
The map $F_A$ induces an automorphism on the Banach space $\Gamma$ of bounded continuous vector fields $\gamma$
vanishing at $\Omega_i$, $i=0,1,2,3$:
\[F_A^\sharp : \gamma \mapsto DF_A \circ \gamma \circ F_A^{-1}\;.\]
The hyperbolicity of $F_A$ (it uniformly expands in the unstable direction and uniformly contracts in the stable direction) implies that 
the linear operator $id-F_A^\sharp$ has a bounded inverse\footnote{As $F_A$ is a linear map, it is easy to provide an explicit formula for 
$id-F_A^\sharp$: if $(id-F_A^\sharp)\gamma=\beta$, then
$\gamma=\sum_{n=0}^\infty e^{-n\sigma} \beta_s\circ F_A^{-n} -\sum_{n=1}^\infty e^{-n\sigma} \beta_u\circ F_A^n$,
where $\beta_s$ and $\beta_u$ are the projections of $\beta$ to the stable and, respectively, unstable directions of $F_A$, and
$e^{\pm\sigma}$ are the eigenvalues of $F_A$, $\sigma>0$.}.
By the implicit function theorem, this implies that the fixed point $\gamma=0$ of $F_A^\sharp$ is unique, and every (nonlinear) 
operator on $\Gamma$ which is $C^1$-close to $F_A^\sharp$ has a unique fixed point close to
$\gamma=0$. In particular, the operator 
\[\gamma \mapsto \tilde F_A \circ ( id + \gamma)\circ  F_A^{-1} - id\;,\]
has a unique fixed point $\gamma$. By the construction, the map $h=id+\gamma$ satisfies:
\[h \circ F_A=\tilde F_A \circ h\; .\]
Thus, $\tilde F_A$ and $F_A$ are semi-conjugate, and we have that
\begin{equation}\label{h12c}
h \circ F_A^n=\tilde F_A^n \circ h
\end{equation}
for every integer $n$, positive and negative. 

In order to prove the topological conjugacy between $F_A$ and $\tilde F_A$, it remains to show that the continuous map
$h$ is injective. This is done as follows: if $h(P)=h(Q)$, then $h(F_A^n P)=h(F_A^n Q)$ for all $n\in\Z$, by (\ref{h12c}).
Since $h$ is uniformly close to identity, it follows that $F_A^n P$ is uniformly close to $F_A^n Q$, i.e., $A^n(P-Q)$ is uniformly
small for all $n\in\Z$. By the hyperbolicity of the matrix $A$, this gives $P=Q$, as required.
\end{proof} 

\subsection{Stochastic island in the disc}\label{t2s} 

It is easy to see that the result of the factorization $\pi$ of the 4-punctured torus $\T^2\setminus \cup_{i=0}^3\{\Omega_i\}$ 
over $(-id): (x,y)\mapsto (-x,-y)$ is a 4-punctured sphere. One can realize the smooth map 
$\pi:\T^2\setminus \cup_i\{ \Omega_i\}\to \S^2$ e.g. by a Weierstrass elliptic function\footnote{or, if we realize the 4-punctured sphere
as the surface $\{Z^2+\eta(X,Y)=1, \; |X|\leq 1,\; |Y|\leq 1\}$ in $\R^3$, where 
$\eta=(X^2+Y^2+\sqrt{1-X^2-Y^2+X^4+Y^4-X^2Y^2})/2$, then $\pi$ can be explicitly defined as
$\pi(x,y)=(X=2|x|-1,Y=2|y|-1, Z={\rm sign}(xy) \sqrt{1-\eta(X,Y)}), \quad |x|\leq 1/2,\; |y|\leq 1/2$.}; see also \cite{Ka79}.

Each fiber of $\pi$ is a pair of points $(x,y)$ and $-(x,y)$. Since $\hat F$ commutes with $(-id)$ in our construction, and 
$\hat F$ is the identity map in a neighborhood of each of the points $\Omega_i$ ($i=0,1,2,3$) where $\pi$ is singular,
the push-forward $\check F=\pi\circ \hat F\circ \pi^{-1}$ of $\hat F$ is a well-defined $C^\infty$-diffeomorphism of the sphere $\mathbb S^2$.
As the symplectic form $\hat \omega$, which is preserved by $\hat F$, is invariant with respect to $(-id)$, it follows that
the push-forward of $\hat\omega$ by $\pi$ is a smooth symplectic form $\check \omega$ on $\S^2\setminus \cup_i \{\Omega_i\}$, 
and $\check \omega$ is invariant by $\check F$. Note that 
the form $\check \omega$ can get singular at the points $\pi \Omega_i$, but we can smoothen
$\check \omega$ in an arbitrary way near these points; since $\check F$ is the identity map there, it preserves any smooth
area form near $\pi\Omega_i$. Hence $\check F$ leaves invariant a smooth symplectic form $\check \omega$. 

By construction, the set $\check I=\pi \hat I$ is the stochastic island for the map $\check F$ (the island $\hat I$ is at a bounded
distance from the singularities $\Omega_i$, so $\pi^{-1}$ realizes a smooth conjugacy between $\check F|_{\check I}$
and $\hat F|_{\hat I}$, which takes heteroclinic links to heteroclinic links and keeps the maximal Lyapunov exponent positive).
Note that the 4-links $\partial V_i$ are invariant with respect to $(-id)$. The map $\pi$ glues the opposite points of 
$\partial V_i$ together, hence the circles $\pi(\partial V_i)$ that bound the island $\check I$ are heteroclinic bi-links. 

Now we will transform the stochastic island $\check I$ on $\S^2$ to a stochastic island for a map of the plane.
Let us identify $\mathbb S^2$ with the one-point compactification of $\R^2$, where $\pi\Omega_0$ is identified with $\infty$;
this can be done e.g. by the stereographic projection $\pi_0: (\S^2\setminus \pi\Omega_0) \to\R^2$.
As $\check F$ is equal to the identity at a neighborhood of $\Omega_0$, after the projection to $\R^2$, the map $\check F$ will be
a $C^\infty$-diffeomorphism and will be equal to the identity at a neighborhood of infinity. The form $\check \omega$ will 
become a symplectic form on $\R^2$ and it will be preserved by $\check  F$. Let $\check\omega = \beta(x,y) dx\wedge dy$ 
for some smooth function $\beta\neq 0$. The diffeomorphism $\pi_1: (x,y)\mapsto (X=x, Y=\int_0^y \beta(x, s) ds)$ of $\R^2$ 
onto a domain $D\subset \R^2$ transforms $\check \omega$ to the standard symplectic form $dX\wedge dY$. The map 
$\check F$ takes $D$ to $D$ in the coordinates $(X,Y)$ and is equal to identity near the boundary $\partial D$, so it can be extended
onto the whole of $\R^2$ as the identity map outside of $D$; it will still preserve the standard form $dX\wedge dY$. By performing an
additional scaling $\pi_2: (X,Y)\mapsto (\kappa X,\kappa Y)$ 
we can achieve that $\check F=id$ everywhere outside the unit disc $\D$.
Thus the image by $\pi_2\circ\pi_1\circ\pi_0$ of the stochastic island $\check I$ on $\S^2$ will lie inside $\D$;
it is a stochastic island $\mathcal I$ for the map $\check F$ (because $\check I$ is separated from $\pi\Omega_0$, so the map $\pi_2\circ\pi_1\circ\pi_0$
is a smooth conjugacy).

We have shown the existence of a diffeomorphism $\check F\in {\rm Diff}^\infty_{dX\wedge dY}(\D)$
with a stochastic island $\mathcal I$. Note that Proposition \ref{KPAP} from Section \ref{section0} is satisfied 
for this island, as easily follows from Theorem \ref{claim7}:
\begin{coro}\label{remf2}  
The stochastic island $\mathcal I$ for the map $\check F \in {\rm Diff}^\infty (\D)$ is robust relative link preservation. 
\end{coro}
\begin{proof} The island $\mathcal I$ is bounded by four heteroclinic bi-links
${\mathcal L}_i=\pi_2\circ\pi_1\circ\pi_0\circ \pi (\partial V_i)$, $i=0,1,2,3$. For every $F$ which is 
$C^2$-close to $\check F$, if $F$ does not split the links, then it is $C^2$-conjugate to a $C^2$-diffeomorphism of $\D$ which
is $C^2$-close to $\check F$ and keeps the links ${\mathcal L}_i$ invariant. Lifting this diffeomorphism to the torus 
$\T_2$ by $\pi^{-1}\circ\pi_0^{-1}\circ\pi_1^{-1}\circ\pi_2^{-1}$, we obtain a diffeomorphism $\tilde F$ of 
$\T_2\setminus\{\cup_{i=0,1,2,3} \Omega_i\}$ which preserves the links $\partial V_i$ and is $C^2$-close to $\hat F$ on $\hat I$.
By Proposition \ref{claim7}, the map $\tilde F$ has positive maximal Lyapunov exponent at every point of $\hat I$.
Since the smooth conjugacy does not change the Lyapunov exponent, the map $F$ has positive maximal Lyapunov exponent
a every point of the island bounded by the continuations of the links $\mathcal L_i$.
\end{proof}

\section{Geometric model for a suitable stochastic island}\label{section2}

Now, we choose a particular coordinate system in $\R^2$ such that the map $\check F$ and its stochastic island $\mathcal I$
we just constructed will acquire certain suitability properties (as given by Definition \ref{suitabledef}; see Fig. \ref{3hetero}).

Let $F\in {\rm Diff}^\infty_\omega (\R^2)$ have a heteroclinic link $L$. 

\begin{defi}
A \emph{fundamental interval} of $L$ for the map $F$ is a closed segment $D_1\subset L$ such that $F(D_1)\cap D_1$ is exactly
one point -- an endpoint both to $D_1$ and $F(D_1)$.
Given $m\ge 1$, an \emph{$m$-fundamental interval} $D_m$ of $L$ is the union $D_m= \cup_{i=0}^{m-1} F^i (D_1)$ of the $m$ first 
iterates of a certain fundamental interval $D_1$.
\end{defi}

Let $(x,y)$ be symplectic coordinates in $\R^2$, so $\omega =dx\wedge dy$. Below we always fix the orientation in $\R^2$
such that the $x$-axis looks to the right and the $y$-axis looks up.
\begin{defi}\label{strd}
An $m$-fundamental interval $D_m$ of a heteroclinic link $L$ will be called {\em straight} if $D_m$ is included in a straight line $y=const$.
\end{defi}

It is a well-known fact (see \cite{GL01}) that any $m$-fundamental interval $D_m$ can be straightened and the so-called
time-energy coordinates can be introduced in its neighborhood, i.e., the map near $D_m$ becomes a translation to a constant vector.
We formulate this result as

\begin{lemm}\label{h-lem}
If $L$ is a link between
two hyperbolic fixed points $P$ and $Q$ for $F\in {\rm Diff}^\infty_\omega(\R^2)$ and $D_m\in L$ is
an $m$-fundamental interval, then 
there exists a symplectic $C^\infty$-diffeomorphism $\phi$ from a neighborhood of $D_m$ into $\R^2$ such that 
$\phi(D_m)$ is a straight $m$-fundamental interval for $\phi\circ F\circ \phi^{-1}$ and, in a neighborhood of $\phi(D_m)$,
\[\phi\circ F\circ \phi^{-1} = : (x,y)\mapsto (x+\tau, y)\]
for some constant $\tau\neq 0$.
\end{lemm}
\begin{proof}
Put $P$ to the origin of coordinates and bring the map to the Birkhoff normal form by a symplectic $C^\infty$ coordinate
transformation \cite[Thm.~1]{Ch83}. This means that we introduce symplectic coordinates $(x,y)$ near $P$
such that the map $F$ near $P$ will be given by
\begin{equation}\label{birkf}
(x,y)\mapsto (\exp(q(xy))\cdot  x, \exp(-q(xy))\cdot  y)
\end{equation}
for some function $q \in C^\infty(\R,\R)$ with $q(0)>0$. Note that the unstable manifold of $P$ is given by $y=0$ in these coordinates. 
By iterating $F$ forward, we can extend the domain of the Birkhoff coordinates to a small neighborhood of any compact subset of $W^u(P)$. 
In particular, we may assume that the map $F$ is given by (\ref{birkf}) near the $m$-fundamental interval $D_m$.
Observe that $D_m:= \{(x,y): \; \tilde x\in [x_0, e^{m q(0)} x_0], \; y=0\}$ in these coordinates, for some $x_0\neq 0$ (by making, if necessary, the coordinate change
$(x,y)\to-(x,y)$, we can always make $x_0>0$).

Let $h=\int q$. Obviously, $F$ is the time-1 map of the flow defined by the Hamiltonian $H(x,y)= h(xy)$. 
Put $\displaystyle X(x,y):= \frac{\ln x}{q(xy)}$ and $Y(x,y)= h(x y)$. The map $(x,y)\mapsto (X,Y)$ is a $C^\infty_\omega$-coordinate 
change near $D_m$ which conjugates $F$ with $(X,Y)\mapsto (X+1,Y)$. 
\end{proof}

Note that the map $(x,y)\mapsto (x+\tau, y)$ is the time-$\tau$ map by the vector filed $\dot x=1,\; \dot y=0$
defined by the Hamiltonian $H(x,y)=y$. Therefore, $x$ plays the role of time and the conserved quantity $y$ can be viewed as energy,
which justifies the ``time-energy'' terminology.

\subsection{Making a bi-link suitable}

A \emph{vertical strip} $V$ is the region $\{(x,y):\; x\in [c_1,c_2], y\in \R\}$ in $\R^2$ for some $c_1<c_2$.

\begin{defi}[Suitable intersection of a heteroclinic bi-link with two strips]
Two vertical strips $V^a$ and $V^b$ intersect a heteroclinic  bi-link  $(\mathring L^a, \mathring  L^b)$ of a dynamics $\mathring  F$  in a \emph{suitable way} if:
\begin{itemize}
\item the intersection of $V^a$ with $\mathring L^a$ is a straight $2$-fundamental interval $D_2^a$;
\item the intersection of $V^b$ with $\mathring L^b$ is the disjoint union of a straight $2$-fundamental interval $\mathring D_2^b$ and
a straight $4$-fundamental interval $\mathring D_4^b$;
\item the strip $V^b$ does not intersect $\mathring L^a$;
\item there exists $\tau>0$ such that the map $\mathring F$ in restriction to a neighborhood of $\mathring D_2^a$ is given by $(x,y)\mapsto (x-\tau, y)$
and, in restriction to a neighborhood of $\mathring D_2^b$, it is given by $(x,y)\mapsto (x+\tau, y)$;
\item there exists $n\ge 1$ such that $\mathring F^n$ sends $\mathring D_2^b \cup \mathring F^2(\mathring D_2^b)$ to $\mathring D_4^b$,
and the restriction of $\mathring F^n$ to a neighborhood of $\mathring D_2^b$ is 
\[\mathring F^n:  (x,y)\mapsto \theta - ( \frac 12  x, 2  y)\]
for some $\theta \in \R^2$. 
\end{itemize}\label{suit}
{\em See Fig. \ref{suitablebilink} for an illustration.}
\end{defi}
\begin{figure}[h!]	\centering
		\includegraphics[width=11cm]{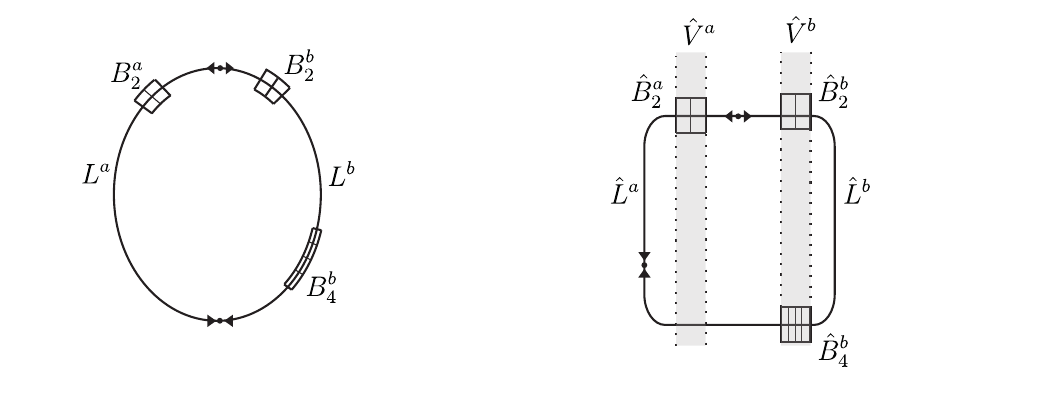}
\caption{Suitable intersection of a heteroclinic bi-link with two  strips}
\label{suitablebilink}
\end{figure}

In Lemma \ref{laprop0} below, we are going to show  that any
conservative diffeomorphism $F$ with 
a  bi-link $(L^a, L^b)$ is  smoothly conjugate to a conservative diffeomorphism $\mathring F$ with a suitable 
bi-link $(\mathring L^a, \mathring L^b)$. This means that we have a lot of geometric freedom in the choice of the bi-link 
$(\mathring L^a, \mathring L^b)$, that we shall explain. 

Take any two parallel straight lines in $\R^2$: $\{y=y_1\}$ and $\{y=y_2\} $ with $y_2<y_1$. 
Let $\tau>0$ and $x_a<x_b$ so that $x_a+\tau < x_b-\tau$. Consider the two vertical strips:
$$V^a:= [x_a-\tau ,x_a+\tau ]\times \R \quad \mbox{ and } \quad V^b:= [x_b-\tau ,x_b+\tau ]\times \R.$$
Consider any $C^\infty$-smooth circle $\mathring L$ in $\R^2$ equal to the union of two curves $\mathring L^a$, $\mathring L^b$ satisfying:
\begin{itemize}
\item $\mathring L = \mathring L^a\cup \mathring L^b$  and $\mathring L^a\cap \mathring L^b = \partial \mathring L^a = \partial \mathring L^b$\; .
\item $\mathring L^a\cap V^a = [x_a-\tau ,x_a+\tau ]\times \{y_1\}$ and $\mathring L^a\cap V^b = \varnothing $.
\item $\mathring L^b\cap V^b = [x_b-\tau ,x_b+\tau ]\times \{y_1, y_2\}$.
 \end{itemize}
Now we can state:
\begin{lemm}\label{laprop0}
Let  $F\in {\rm Diff}^\infty_\omega(\R^2)$ have a heteroclinic bi-link $(L^a,L^b)$.
Then
there exists a symplectic $C^\infty$-diffeomorphism $\mathring\phi$ of a small neighborhood of $L^a\cup L^b$ into $\R^2$ such that:
\begin{itemize}
\item  $\mathring\phi(L^a\cup L^b)=\mathring L$, 
 $\mathring\phi(L^a)=\mathring L^a$,  $\mathring\phi(L^b)=\mathring L^b$
;\item the vertical strips $V^a$ and $V^b$ intersect the bi-link $(\mathring L^a, \mathring L^b)$ of 
$\mathring\phi\circ F\circ \mathring\phi^{-1}$ in a suitable way (with $n=4$ in Definition \ref{suit}). 
\end{itemize}
\end{lemm}
\begin{proof}
Let us take fundamental intervals $D^a$ of $L^a$ and $D^b$ of $L^b$. Observe that:
\begin{itemize}
\item $D_2^a:= D^a\cup F(D^a)$ and $D_2^b:=  D^b\cup F(D^b)$ are 2-fundamental intervals of $L^a$ and $L^b$, respectively;
\item $D_4^b:=  \cup_{i=4}^{7} F^i(D^b)$ is a 4-fundamental interval of $L^b$.
\end{itemize}
We remark that $D_2^b$ and $D_4^b$ are included in the $8$-fundamental interval $D^b_8:=\cup_{i=0}^{7} F^i(D^b)$ of $L^b$. 

By Lemma \ref{h-lem}, there exist symplectic diffeomorphisms 
$\phi^a$ and $\phi^b$ acting from a neighborhood of $\hat D^a_2:= D^a_2\cup F(D^a_2)$ and, respectively, a neighborhood of 
$\hat D^b_8= D_8^b\cup F(D^b_8)$ into $\R^2$ such that:
\begin{itemize}
\item $\phi^a(D^a_2) = [-2,0]\times \{0\}$ and $\phi^a\circ F\circ (\phi^a)^{-1}$ is the translation to $(-1,0)$ in
a neighborhood of $\phi^a(D^a_2)$;
\item $\phi^b(D^b_8) = [0,8]\times \{0\}$ and $\phi^b\circ f\circ (\phi^b)^{-1}$ is the translation to $(1,0)$ in
a neighborhood of $\phi^b(D^b_8)$.
\end{itemize}
Observe that $\phi^b(D^b_2) = [0,2]\times \{0\}$ and $\phi^b(D^b_4) = [4,8]\times \{0\}$.

Let $\delta>0$ be small and denote $J:=[-\delta\tau , \delta\tau ]$ and $J':= [-\frac{\delta\tau }2, \frac{\delta\tau }2]$. Let
\[B_2^a:= (\phi^a)^{-1}([-2,0]\times J),\qquad B_2^b:= (\phi^b)^{-1}([0,2]\times J)\quad, \quad B_4^b:= (\phi^b)^{-1}([4,8]\times J')\; .\]
For $\delta>0$ small enough, the sets $B_2^b$, $B_4^b$ and $B_2^a$ are disjoint. Take two linear area-preserving maps:
\[A_2:= (x,y)\mapsto (\tau x ,\frac{y}\tau)\quad \text{and}\quad  A_4:= (x,y)\mapsto -\; (\frac{\tau }2x  ,\frac{ 2}\tau y)\; .\]
Notice that the maps 
\[\mathring\phi^a_2 := A_2\circ \phi^a, \qquad
\mathring\phi^b_2 := A_2\circ \phi^b, \qquad 
\mathring\phi^a_4 := A_4\circ \phi^b\]
send, respectively, $B_2^a$, $B_2^b$ and  $B_4^b$ onto
translations of $R:= [-\tau,\tau]\times [-\delta,\delta]$.   

Let $\mathring\phi_0$ be a symplectic embedding of a neighborhood of the disjoint union $B_0:= B_2^a\cup B_2^b\cup B_4^b$ into
$\R^2$ such that:
\begin{itemize}
\item $B_2^a$, $B_2^b$ and $B_4^b$ are sent by $\mathring\phi_0$ onto, respectively, (see Fig. \ref{suitablebilink}):
\[\mathring B_2^a := R+(x_a, y_1),\quad
\mathring B_2^b :=R+(x_b, y_1),\quad 
\mathring B_4^b := R+(x_b, y_2);\]
\item the restriction of $\mathring\phi_0$ to neighborhoods of $B_2^a$, $B_2^b$ and $B_4^b$ is 
the composition of respectively $\mathring\phi^a_2$, $\mathring\phi^b_2$ and $\mathring\phi^b_4$ with some translations. 
\end{itemize}
Note that $\mathring\phi_0$ sends the fundamental intervals $D_2^a$, $D_2^b$ and $D_4^b$ onto, respectively,:
$$\mathring D_2^a:= [-\tau,\tau]\times \{0\}+(x_a, y_1),
\quad
\mathring  D_2^b := [-\tau,\tau]\times \{0\}+(x_b, y_1),\quad 
\mathring  D_4^b := [-\tau,\tau]\times \{0\}+(x_b, y_2)\; ,$$
so these images lie in the curve $\mathring L$, in the intersection with the vertical strips
$V^a$ and $V_b$.

The map $\mathring\phi_0\circ F^4\circ \mathring\phi_0^{-1}$ sends $\mathring D_2^b$  into $\mathring D_4^b$ and its restriction 
to a neighborhood of $\mathring D_2^b$ is the composition of a translation with the linear map $(x,y)\mapsto (- x/2, -2y)$, as required by 
Definition \ref{suit} with $n=4$. 
Therefore, to prove the lemma, it suffices to construct a symplectic $C^\infty$-diffeomorphism $\mathring\phi$ of a small neighborhood
of $B_0\cup L^a\cup L^b$ into $\R^2$ which would send $L^a, L^b$ to $\mathring L^a, \mathring L^b$, such that its restriction 
to a neighborhood of $B_0$ would be equal to $\mathring\phi_0$. 

Without the symplecticity requirement, the map $\mathring\phi$ would be given by Whitney extension theorem \cite{Wh34}. Making the
diffeomorphism $\mathring\phi$ symplectic requires an extra effort, as it is done below.

Consider an annulus $\A:= (\R/\Z)\times [-\eta,\eta]$ for a sufficiently small $\eta>0$. Let $t\in \R/\Z$ and $h\in [-\eta,\eta]$
be coordinates in $\A$.  
By the Weinstein's Lagrangian neighborhood theorem \cite{We71}, if $\eta$ is sufficiently small, then there exists an area-preserving 
diffeomorphism $N$ from the annulus $\A$ to a neighborhood
of the bi-link $L^a\cup L^b$, which sends the central circle $S:=\{h=0\}$ to
$L^a\cup L^b$. Similarly, there exists an area-preserving diffeomorphism $\mathring N$ from $\A$ to
a small neighborhood of the curve $\mathring L = \mathring L^a\cup \mathring L^b$, which sends $S$ to $\mathring L$. 

Let $\delta>0$ be small enough, so that the sets $B_0$ and $\mathring\phi_0(B_0)$ will be contained in $N(\A)$ and, respectively,
$\mathring N(\A)$. By Whitney extension theorem, there exists $G\in {\rm Diff}^\infty(\A,\A)$ such that $G(S)=S$ and 
$\mathring N\circ G\circ N^{-1}$ restricted to a neighborhood $U$ of $B_0$ is $\mathring\phi_0$. In particular, $\det DG|_{N^{-1}(U)} =1$. 
The map $G$ is orientation-preserving but it is not, a priori, area-preserving outside of $U$. 

Our goal is to correct $G$ in order to make it area-preserving. More precisely, we are going to construct a $C^\infty$-diffeomorphism 
$\Psi$ of $\A$ such that $\det D\Psi = \det DG$, $\Psi(S)=S$, and the restriction of $\Psi$ to $N^{-1}(B_0)$ is the identity. Then the Lemma
will be proved by taking $\mathring\phi:=\mathring N\circ G\circ \Psi^{-1}\circ N^{-1}$. 

Let us keep fixed the neighborhood $U $ of $B_0$ where $G$ is area-preserving, and let us take $\delta>0$ small. Then $B_0$
can be made as close as we want to $D_2^b\cup D_4^b\cup D_2^a$. 
Hence, for $\delta>0$ small enough, if the image by $N$ of a vertical segment $\{t=const, \; |h|\leq \eta\}$
intersects $B_0$, then it lies entirely in $U$, i.e., $\det DG=1$ everywhere on this segment.
Therefore, if we define the map 
\[\Psi: (t,h) \mapsto (t, \int_{0}^{h} \det DG(t,s)ds),\]
then $\Psi=id$ in the restriction to $N^{-1}(B_0)$. It is also obvious, that $\Psi=id$ in restriction to the central circle 
$S=\{h=0\}$, and $\det D\Psi= \det DG$.\end{proof}

\subsection{Making the stochastic island suitable}
Consider the map $\check F\in {\rm Diff}^\infty_\omega$ with the stochastic island $\mathcal I$.
Recall that $\mathcal I$ is bounded by 4 heteroclinic bi-links $(L_i^a,L_i^b)$, $i=0,1,2,3$, each of which
is a $C^\infty$-smooth circle. We take a convention that $L_0^a\cup L_0^b$ is the outer circle, i.e., the bi-links $(L_i^a,L_i^b)$ with $i=1,2,3$ 
lie inside the region bounded by $L_0^a\cup L_0^b$. 

Below, we will construct symplectic coordinates $\mathring\phi$ in $\R^2$ such that the island $\mathring\phi({\mathcal I})$ will
satisfy the following suitability conditions.

\begin{defi}[Suitable intersection of 4 heteroclinic bi-link-s with 4 pairs of vertical strips]\label{suitabledef}
We say that 4 pairs of vertical strips $(V_i^a , V_i^b)$, $i=0,1,2,3$, intersect bi-links 
$(\mathring L_j^a,\mathring L_j^b)$, $j=0,1,2,3$ in a suitable way if the following conditions hold true.
\begin{enumerate}[$(H_1)$]
\item For every $0\le i\le 3$, the intersection of  $V^a_i\sqcup V^b_i$ with $(\mathring L_i^a,\mathring L_i^b)$ is suitable
in the sense of Definition \ref{suit}, with $n=4$.
\item For every $j\ge 1$ and every $i\not =  j$, the strips $V^a_i$ and $V^b_i$ do not intersect the circle $\mathring L_j^a\cup \mathring L_j^b$
\end{enumerate}
{\em See Fig. \ref{3hetero} for an illustration.}
\end{defi}
\begin{figure}[h!]
	\centering
		\includegraphics[width=11cm]{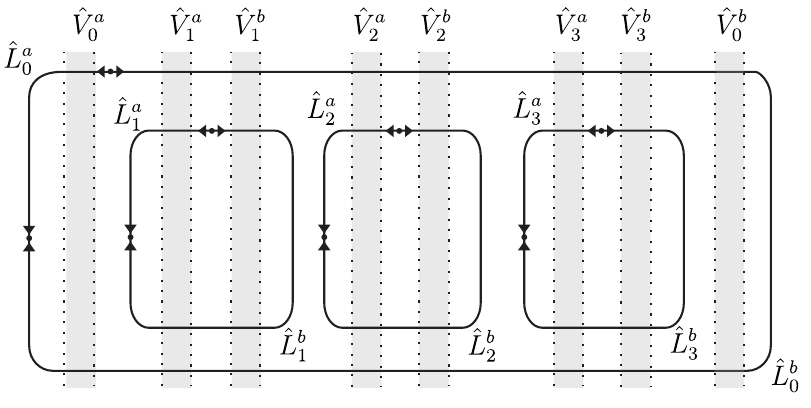}
		\caption{Suitable intersection of the stochastic island and 8 vertical strips}\label{3hetero}
\label{fig:newcoord}
\end{figure}

\begin{prop}\label{suitpro}
There exist $\mathring\phi\in {\rm Diff}^\infty_\omega(\R^2)$ and 4 pairs of   
vertical strips $V_i^a$,  $V^b_i$, $i=0,1,2,3$, 
which intersect the heteroclinic bi-links $(\mathring L_j^a=\mathring\phi(L_j^a), \mathring L_j^b=\mathring\phi(L_j^b))$, $j=0,1,2,3$, 
of the map $\mathring F=\mathring\phi\circ \check F\circ \mathring\phi^{-1}$ in a suitable way.
\end{prop}
\begin{proof}
By Lemma \ref{laprop0}, for every $i=0,1,2,3$ there exists a pair of vertical strips $V_i^a$,  $V^b_i$ and
a symplectic $C^\infty$ diffeomorphism $\mathring\phi_i$ of a small neighborhood of the bi-link $(L_i^a, L_i^b)$
such that the strips $V_i^a$, $V_i^b$ intersect $(\mathring L^a_i, \mathring L^b_i ):= (\mathring\phi_i(L^a_i), \mathring\phi_i(L^b_i))$ of 
the map $\mathring F_i:=\mathring\phi_i\circ \check F\circ \mathring\phi_i^{-1}$ in a suitable way, ensuring the fulfillment of Condition $(H_1)$
of Definition \ref{suitabledef}. 

Note that in Lemma \ref{laprop0} there is a freedom
in the choice of the curve $\mathring L_i=\mathring\phi_i(L^a_i \cup L^b_i)$. So, we take $\mathring L_i$ such that
it will bound a disc of the same area as $L_i=L_i^a\cup L_i^b$ does. Also, by choosing the 
constants $y_1, y_2, x_0, x_1$ in Lemma \ref{laprop0} in an appropriate way for each $i$, we can
assure that $\mathring\phi_i(L_i)$ do not intersect for different $i$ and $\mathring\phi_1(L_1)\cup \mathring\phi(L_2)\cup 
\mathring\phi(L_3)$ lies
inside the disc bounded by $\mathring\phi_0(L_0)$, and the strips $V_i^a$ and $V_i^b$ are positioned where we wish,
thus ensuring Condition $(H_2)$ of Definition \ref{suitabledef}.

Let $\A_i$ be a sufficiently small closed annulus around $L_i$, $i=0,1,2,3$. Let us prove the proposition by showing the existence of
a symplectic extension $\mathring\phi$ of the symplectic maps $\mathring\phi_i$ from a neighborhood of the annuli $\A_i$
to the whole of $\R^2$, i.e., a diffeomorphism $\mathring\phi\in {\rm Diff}^\infty_\omega$ such that
$\mathring\phi|_{\A_i}= \mathring\phi_i|_{\A_i}$ for all $i=0,1,2,3$. Since the curve $\mathring L_i$ bounds the disc of the same area
as $L_i=L_i^a\cup L_i^b$ does, for each $i$, the annuli ${\A}_i$ is necessarily such that the area of each of the connected components of $\R^2\setminus \sqcup_i \A_i$
equals to the area of a corresponding component of $\R^2\setminus \sqcup_i \mathring\phi_i(\A_i)$. 
Then the existence of the  sought symplectic extension $\mathring \phi$ is a standard consequence of Dacorogna-Moser theorem \cite{DM90}, as given by

\begin{cor}[Cor. 4 \cite{Av10}]\label{coroAvila}
Let $K\subset \R^2$ be  a  compact  set, $U$ be  a  neighborhood  of $K$ and  let $\psi\in C^\infty_\omega(U, \R^2)$  be  close to the identity.
Assume that every bounded connected component $W$ of $\R^2\setminus U$  and its corresponding one in $\R^2 \setminus \psi (U )$  have  the  same  area. 
Then  there exists $\phi\in Diff^\infty_\omega(\R^2)$ which is $C^\infty$-close to the identity and such that $\phi|K= \psi|K$. 
\end{cor}
\end{proof}

\section{Restoration of broken heteroclinic links}\label{section3}

Let $\mathring F\in Diff^\infty_\omega(\R^2)$ be the map constructed in the previous Section. It has a stochastic island  $\mathring{\mathcal I}$ bounded by 
4 smooth circles - heteroclinic bi-links $(\mathring L_i^a, \mathring L_i^b)$, $i=0,1,2,3$. The bi-link 
$\mathring L_0^a\cup \mathring L_0^b$ forms the outer boundary of $\mathring {\mathcal I}$.   
By construction, there exist 4 pairs of vertical strips $(V_i^a,  V_i^b)$ which intersect
$(\mathring L_i^a, \mathring L_i^b)$ in a suitable way in the sense of Definition \ref{suitabledef}. We denote
$V_i^a=I_i^a\times \R$ and $V_i^b=I_i^b\times \R$, where $I_i^a$, $I_i^b$ are closed disjoint intervals in the $x$-axis.

In this Section we consider perturbations of the map $\mathring F$ and prove Proposition \ref{closing_bi-link}.
Namely, we show that {\em for every $r\ge 1$, for every $\eta>0$, for every $F\in\Diff^{r+8}_\omega$
which is sufficiently close to $\mathring F$ in $C^{r+8}$, there exists $\psi\in C^r(\R,\R)$, supported in $\cup_i (I_i^a\cup I_i^b)$
and with $C^r$-norm smaller than $\eta$, such that the map $\bar F=S_\psi\circ F$ has 4 heteroclinic bi-links
$(L_i^a, L_i^b)$ close to $(\mathring L_i^a, \mathring L_i^b)$.}
We recall that given a function $\psi$, we denote 
\[S_\psi: (x,y)\mapsto  (x, y+\psi(x))\;.\]
 
It is pretty much obvious that to prove this statement, it suffices to show
\begin{prop}\label{propbilinksimple}
Let $r\ge 1$.  Let $\mathring F\in\Diff^{r+4}_\omega$ have a heteroclinic bi-link $(\mathring L^a, \mathring L^b)$ which 
intersects two vertical strips $V^a=I^a\times \R$ 
and $V^b=I^b\times \R$ in a suitable way. For every $F\in\Diff^{r+4}_\omega$ which is $C^{r+4}$-close to $\mathring F$,
there exists $\psi\in C^r(\R,\R)$ which is $C^r$-small and supported in 
$I^a\sqcup I^b$, such that the map $S_\psi\circ F$ has a bi-link $(L^a, L^b)$ close to $(\mathring L^a,\mathring L^b)$. 
\end{prop}
Proposition \ref{closing_bi-link} is inferred from this statement as follows. 
\begin{proof}[Proof of Proposition \ref{closing_bi-link}]
Let $F$ be a $C^{r+8}$-small perturbation of $\mathring F$. By Proposition \ref{propbilinksimple}, for each $i=1,2,3$
there exists a $C^{r+4}$-small function $\psi_i$ supported in $I^a_i\sqcup I^b_i$ such that
the map $S_{\psi_i}\circ \mathring F$ has a bi-link $(\bar L_i^a, \bar L_i^b)$ close to $(\mathring L_i^a, \mathring L_i^b)$.
By property $(H_2)$ of the suitable intersection (see Definition \ref{suitabledef}), 
the vertical strips $V^a_i$ and $V^b_i$ do not intersect the bi-links $(\mathring L_j^a, \mathring L_j^b)$ for $j\neq i$, $j>0$.
Thus the map $S_{\psi_i}$ is identity near the bi-links
$(\mathring L_j^a, \mathring L_j^b)$ with $j\neq i$, $j>0$. Therefore, the map
$S_{\psi_1+\psi_2+\psi_3}\circ F$ has 3 bi-links $(L^a_i, L^b_i)$ close to $(\mathring L^a_i,\mathring L^b_i)$, respectively.

The map $S_{\psi_1+\psi_2+\psi_3}\circ F$ is an $\omega$-preserving diffeomorphism and is $C^{r+4}$-close to $\mathring F$.
Therefore, by applying Proposition \ref{propbilinksimple} to this map and the link $(\mathring L^a_0, \mathring L^b_0)$,
we obtain that there exists a $C^r$-small function $\psi_0$ localized in $I_0^a\cup I_0^b$ such that the map
$S_{\psi_0}\circ S_{\psi_1+\psi_2+\psi_3}\circ F$ has a bi-link $(L^a_0, L^b_0)$ close to $(\mathring L^a_0, \mathring L^b_0)$.
Since $V^a_0$ and $V^b_0$ do not intersect the bi-links $(L_i^a, L_i^b)$ for $i>0$ (by property $(H_2)$
of the suitable intersection), the map $S_{\psi_0}$ is identity near these bi-links, hence it does not destroy them.
Thus, the map $\bar F=S_\psi\circ F$ with $\psi=\psi_0+\psi_1+\psi_2+\psi_3$ has all 4 bi-links $(L_i^a,L_i^b)$ as required.
\end{proof}

\begin{proof}[Proof of Proposition \ref{propbilinksimple}]
This Proposition follows from the two lemmas below which we prove in Sections \ref{lm1} and \ref{lm2} respectively.
\begin{lemm}\label{propbilinksimple1}
Under the hypotheses of Proposition \ref{propbilinksimple}, for every $F\in\Diff^{k}_\omega$ which is $C^{k}$-close to $\mathring F$, $k\geq 3$,
there exists a $C^{k-2}$-small function $\psi_a$ supported in $I^a$ and such that the map $S_{\psi_a}\circ F$ has a link $L^a$ close to $\mathring L^a$.
\end{lemm}
\begin{lemm}\label{propbilinksimple2}
Under the hypotheses of Proposition \ref{propbilinksimple}, for every $F\in\Diff^{k}_\omega$ which is $C^{k}$-close ($k\geq 3$)
to $\mathring F$ and has a link $L^a$ close to $\mathring L^a$, there exists a $C^{k-2}$-small function $\psi^b$ supported in $I^b$
and such that the map $S_{\psi_b}\circ F$ has a link $L^b$ close to $\mathring L^b$.
\end{lemm}
Indeed, if an $\omega$-preserving diffeomorphism $F$ is a $C^{r+4}$-small perturbation of $F$, then,
by Lemma \ref{propbilinksimple1}, there exists a $C^{r+2}$-small $\psi_a$ such that the map $S_{\psi_a}\circ F$ has the link $L_a$.
This map is $\omega$-preserving and is $C^{r+2}$-close to $\mathring F$. Therefore, applying Lemma \ref{propbilinksimple2} to this map,
we find that there exists a $C^r$-small $\psi_b$ supported in $I^b$ and such that the map
$S_{\psi_b}\circ S_{\psi_a}\circ F= S_{\psi_a+\psi_b}\circ F$ has the link $L_b$.
As the strip $I^b\times \R$ does not intersect $L^a$, the link $L^a$ also persists for the map $S_{\psi_a+\psi_b}\circ \hat F$,
which gives Proposition \ref{propbilinksimple} with $\psi=\psi_a+\psi_b$.
\end{proof}
 
\subsection{Evaluation of the link splitting}\label{evlink}

The map $\mathring F$ has two saddle fixed points $P$ and $Q$ on the circle $\mathring L^a\cup \mathring L^b$. The point $P$ is repelling on the circle,
while $Q$ is attracting on the circle. 

Let $W^a(P;\mathring F)$ and $W^b(P;\mathring F)$ denote the halves of the unstable manifolds of $P$ equal to, respectively, 
$\mathring L^a\setminus \{Q\}$ and $\mathring L^b\setminus \{Q\}$. Let $W^a(Q;\mathring F)$ and $W^b(Q;\mathring F)$ be the halves of the stable manifolds
of $Q$ equal to respectively $\mathring L^a\setminus \{P\}$ and $\mathring L^b\setminus \{P\}$. 

The points $P$ and $Q$ persist for every $C^1$-close map $F$, and depend continuously on $F$.
The corresponding manifolds $W^a(P; F)$, $W^b(P;F)$, $W^a(Q; F)$ and $W^b(Q; F)$ also persist, and depend continuously on
$F$ as embedded curves of the same smoothness as $F$. To avoid ambiguities, we will fix a sufficiently small neighborhood
of the point $Q$ and then $W^a(P; F)$ and $W^b(P;F)$ will denote the two arcs of $W^u(P,F)$ which connect $P$ with the boundary of this neighborhood
and are close, respectively, to $W^a(P;\mathring F)$ and $W^b(P;\mathring F)$. Similarly, $W^a(Q; F)$ and $W^b(Q; F)$
are the arcs in $W^s(Q,F)$ which connect $Q$ with the boundary of a small neighborhood of $P$
and are close, respectively, to $W^a(Q;\mathring F)$ and $W^b(Q;\mathring F)$.

In general, the links are broken when the map $\mathring F$ is perturbed, so $W^a(P; F)$ and 
$W^b(P; F)$ do not need to coincide with, respectively, $W^a(Q; F)$ and $W^b(Q; F)$.

In the next two Sections we will show, for a given $F\in \Diff^\omega_k(\D)$ which is $C^k$-close to $\mathring F$,
 how to find a $C^{k-2}$-function $\psi$ such that each of the unions
$W^a(P; S_\psi\circ F) \cup W^a(Q; S_\psi\circ F)$ 
and  $W^b(P; S_\psi\circ F) \cup W^b(Q; S_\psi\circ F)$ forms a heteroclinic link between $P$ and $Q$.

In this Section, we obtain formulas for the defect of coincidence between
$W^a(P; S_\psi\circ F)$ and $W^a(Q; S_\psi\circ F)$ or $W^b(P; S_\psi\circ F)$ and $W^b(Q; S_\psi\circ F)$.
In order to do that, we shall use the so-called time-energy coordinates
near the fundamental interval $\mathring D_2^a= V_a\cap L_a$ of the link $L_a$ and the fundamental domain 
$\mathring D_2^b\subset V_b\cap L_b$ of the link $L_b$. 
Recall that by the suitability conditions (see Definition \ref{suit}) there exist $\tau>0$ and $(x_a,y_a)\in \D$,
$(x_b,y_b)\in \D$ such that
$\mathring D_2^a= \{x\in [x_a-2\tau, x_a]\} \times \{y=y_a\}$, 
$\mathring D_2^b= \{x\in [x_b, x_b+2\tau]\} \times \{y=y_b\}$, and the map
$\mathring F$ restricted to a small neighborhood $N^a$ of  $\mathring D_2^a$ or a small neighborhood $N^b$ of  
$\mathring D_2^b$ is given by
\begin{equation}\label{mathrtra}
\mathring F|_{N^a} := (x,y) \mapsto (x-\tau, y), \qquad \mathring F|_{N^b} := (x,y) \mapsto (x+\tau, y)\;.
\end{equation}

\begin{defi}
For an $\omega$-preserving map $F$, which is $C^k$-close to $\mathring F$, an $N^a$ time-energy chart $\phi^a$ is 
an $\omega$-preserving diffeomorphism from 
$N^a\cup F(N^a)$ to $\D$ which is $C^{k-1}$-close to identity and satisfies 
\begin{equation}\label{techa}
\phi^a\circ F|_{N^a} = \mathring F\circ \phi^a|_{N^a}\;.
\end{equation}
An $N^b$ time-energy chart $\phi^b$ is an $\omega$-preserving diffeomorphism from 
$N^b\cup F(N^b)$ to $\D$ which is $C^{k-1}$-close to identity and satisfies 
\begin{equation}\label{techb}
\phi^b\circ F|_{N^b} = \mathring F\circ \phi^b|_{N^b}\;.
\end{equation}
\end{defi}
We notice that the identity map is a time-energy chart for $\mathring F$.  The time-energy charts are not uniquely 
defined, so we
will fix their choice below. In our construction the time-energy charts will be identity near $\{x=x_a\}$ and $\{x=x_b\}$.

Once certain time-energy coordinates are introduced in $N^a\cup F(N^a)$, 
the curves $W^a(P; F)\cap \{N^a\cup F(N^a)\}$ and $W^a(Q; F)\cap \{N^a\cup F(N^a)\}$ become graphs of $\tau$-periodic
functions: the manifolds $W^a(P; F)$ and $W^a(Q; F)$ are invariant with respect to $F$ which means that in the 
time-energy
coordinates they are invariant with respect to the translation to $(-\tau,0)$, see (\ref{techa}),(\ref{mathrtra}).
We denote as $w^u_a(F,\phi^a)$ and $w^s_a(F,\phi^a)$ the $\tau$-periodic functions whose graphs are the curves
$\phi^a(W^a(P; F))$ and $\phi^a(W^a(Q; F))$, respectively.
\begin{defi}
The \emph{link-splitting function} $M^a(F, \phi^a)$ associated to $(N^a, F,\phi^a)$ is the $\tau$-periodic function equal to
$w^u_a(F,\phi^a)-w^s_a(F,\phi^a)$ at $x\in [x_a-\tau,x_a]$. 
\end{defi}
Similarly, let $w^s_b(F,\phi^b)$ and $w^u_b(F,\phi^b)$ be the $\tau$-periodic functions whose graphs are
the curves $\phi^b(W^b(Q;F)$ and $\phi^b(W^b(P;F))$.
\begin{defi}\label{lbd}
The \emph{link-splitting function} $M^b(F, \phi^b)$ associated to $(N^b, F,\phi^b)$ is the $\tau$-periodic function
equal to $w^u_b(F,\phi^b)-w^s_b(F,\phi^b)$ at $x\in [x_b,x_b+\tau]$.
\end{defi}
By the definition, the link $L^a$ or $L^b$ is restored when the function $M^a$ or, respectively, $M^b$ is identically 
zero.

We start with constructing a $C^{k-1}$-smooth time-energy chart for the map $F$.
\begin{lemm}\label{forMenilkov}
There exists a small neighborhood $N^a$ of $\mathring D_2^a$ and a small neighborhood $N^b$ of $\mathring D_2^b$ such 
that 
for every $\omega$-preserving diffeomorphism $F$ which is $C^k$-close to $\mathring F$, $k\geq 3$, there 
exists $C^{k-1}$-smooth time-energy chart $\phi^a$ and $\phi^b$, which depend continuously on $F$ and equal
to identity if $F=\mathring F$. 
\end{lemm}
\begin{proof}
We will show the proof only for the existence of $\phi^a$. The proof for $\phi^b$ is identical up to the exchange of index $a$
to $b$ and $(-\tau)$ to $\tau$.

Let $\rho\in C^\infty(\R, [0,1])$ be zero everywhere near $x=x_a$ and $1$ everywhere near $x=x_a-\tau$. 
Let $\phi_0(x,y):= (x,y) (1-\rho(x)) + \rho(x) \mathring F\circ F^{-1}(x,y)$. The map $\phi_0$ is a $C^k$-diffeomorphism 
from a small neighborhood
of $\{x\in [x_a-\tau,x_a], \; y=y_a\}$ into $\D$, it is $C^k$-close to identity and equals to
the identity near $(x_a,y_a)$ and to $\mathring F\circ F^{-1}$ near $(x_a-\tau, y_a)$.

Thus, $\phi_0$ satisfies
\[\phi_0\circ F \circ \phi_0^{-1} (x,y) =(x-\tau, y)\; .\]
in a neighborhood of $(x_a, y_a)$ (see (\ref{mathrtra})). Take a small neighborhood of $D^a$ and define there
$\phi^a(x,y)=\phi_0(x, \sigma(x,y))$ where the $C^{k-1}$-function $\sigma$ satisfies $\sigma(x,y_a)=y_a$ and
$\partial_y \sigma = \det D\phi_0^{-1}(x,\sigma)$. By construction, $\det D\phi^a\equiv 1$, i.e., it is
an $\omega$-preserving $C^{k-1}$-diffeomorphism and, since $\phi_0$ is $C^k$-close to the identity, $\phi^a$ is $C^{k-1}$-close to the identity.
Since $\det D\phi_0=1$ everywhere near $(x_a, y_a)$ and $(x_a-\tau,y_a)$, we have that $\sigma\equiv y$ near these points, so
$\phi^a\equiv \phi_0$ there. In particular,
\[\phi^a\circ F = \mathring F \circ \phi^a\]
near $(x_a,y_a)$.

It follows that we obtain the required time-energy chart if we extend $\phi^a$ to a small neighborhood of 
$\mathring D^a_2 \cup \mathring F \mathring D^a_2$ by the rule
\[\phi^a = : \left\{\begin{array}{cl} \mathring F \circ \phi^a \circ F^{-1} & \text{if } x \in [x_a-2\tau, x_a-\tau],\\
\mathring F^2 \circ \phi^a\circ F^{-2} & \text{if } x \leq  x_a-2\tau\; . \\
\end{array}\right.\]
\end{proof}

Now, take some sufficiently small $\delta>0$. Consider any map $F$ close enough to $\mathring F$ and 
let $\phi^{a,b}$ be the $C^{k-1}$ time-energy charts for $F$,
defined in Lemma \ref{forMenilkov}. Given any close to zero smooth function $\psi(x)$ supported inside 
$[x_a-2\tau+\delta, x_a-\delta]$, we consider the map  
\begin{equation}\label{fbdsp}
\bar F:=  S_{\psi}\circ F
\end{equation}
and define for it the time-energy chart $\phi_\psi^a$ in the open set $N^a\cup F(N^a)$ such that
\begin{equation}\label{mchr1a}
\phi_\psi^a = \left\{\begin{array}{cl} \phi^a \circ F \circ  \bar F^{-1} = \phi^a\circ S_{- \psi} & \text{if } x\geq x_a-\tau, \\
\phi^a\circ F^2\circ \bar F^{-2} & \text{if } x\leq x_a-\tau. \\
\end{array}\right.\end{equation}
Recall that $\psi$ vanishes for $x$ close to $x_a$ 
 and for $x$ close to $x_a-2\tau$. 
Furthermore, if $x$ is close to $x_a-\tau$, then the $x$-coordinate of $\bar F^{-1}(x,y)$ is close to $x_a$. Thus, on a neighborhood of $\bar F^{-1}(x,y)$,  it holds $F\circ \bar F^{-1}= id$ and so 
\[\phi^a\circ F^2\circ \bar F^{-2}(x,y )= \phi^a\circ F\circ \bar F^{-1}(x,y )=\phi^a\circ S_{-\psi}(x,y)\; .\]
Hence $\phi_\psi^a$ has no discontinuities at $x=x_a-\tau$ and the following required conjugacy holds true:
\begin{equation}\label{techar}
\phi_\psi^a \circ \bar F|_{N^a} = \mathring F\circ \phi_\psi^a|_{N^a}\; .
\end{equation}

Similarly, for any close to zero smooth function $\psi(x)$ which is supported inside $[x_b+\delta, x_b+2\tau-\delta]$, 
for the map $\bar F$ given by (\ref{fbdsp}), we define the time-energy chart $\phi_\psi^b$ in $N^b\cup F(N^b)$ by the rule
\begin{equation}\label{mchr1b}
\phi_\psi^b = \left\{\begin{array}{cl} \phi^b \circ F \circ  \bar F^{-1} = \phi^b\circ S_{- \psi} & \text{if } x\leq x_b+\tau, \\
\phi^b\circ F^2\circ \bar F^{-2} & \text{if } x\geq x_b+\tau, \\
\end{array}\right.\end{equation}
such that the identity
\begin{equation}\label{techbr}
\phi_\psi^b \circ \bar F|_{N^b} = \mathring F\circ \phi_\psi^b|_{N^b}
\end{equation}
holds.

With this choice of the time-energy charts, for small $\psi$ we have the link-splitting function $M^a(S_\psi\circ F, \phi_\psi^a)$
(if $\psi$ is supported inside $[x_a-2\tau, x_a]$) or  $M^b(S_\psi\circ F, \phi_\psi^b)$ (if $\psi$ is supported inside $[x_b, x_b+2\tau]$).
This defines the operators $\mathcal M^a: \psi \to M^a(S_\psi\circ F, \phi_\psi^a)$ and $\mathcal M^b: \psi \to M^b(S_\psi\circ F, \phi_\psi^b)$
acting from the space of smooth functions supported inside $[x_a-2\tau, x_a]$ or, respectively, inside $[x_b,x_b+2\tau]$,
to the space of $\tau$-periodic functions of the same smoothness.
\subsection{Regularity of the graph transform operator}
 The regularity of the operators $\mathcal M^{a}$ and $\mathcal M^b$ 
depends on the smoothness class of $\psi$. We choose it to be $C^{k-2}$. For $\delta>0$ small, for every $x\in \R$ we denote by $C^{k-2}_0([x-2\tau +\delta, x-\delta ] ,\R)$
the Banach space of real  $C^{k-2}$-functions supported inside $[x-2\tau +\delta, x-\delta ]$, endowed with the $C^{k-2}$-norm. 
We have the crucial result:
\begin{prop}\label{prepropbililinksimple1}
The operators \[\mathcal M^a: \psi \in C^{k-2}_0([x_a-2\tau +\delta, x_a-\delta ] , \R)\mapsto M^a(S_\psi\circ F, \phi_\psi^a)\in C^{k-2}(\R)\]
\[\mathcal M^b: \psi \in C^{k-2}_0([x_b +\delta, x_a+2\tau-\delta ], \R)\mapsto M^b(S_\psi\circ F, \phi_\psi^b)\in C^{k-2}(\R)\] are of class  $C^1$
in a small neighborhood of zero, and depend continuously on the map $F$.
\end{prop}
 Proposition \ref{prepropbililinksimple1} follows immediately from the two following lemmas:  
\begin{lemm}\label{wubfp}\label{wuafp}
The functions $w^u_a(S_\psi \circ  F, \phi_\psi^a)$ and $w^u_b(S_\psi \circ  F, \phi_\psi^b)$ are independent of $\psi$, and depend continuously on $F$:  
\[w^u_a(S_\psi \circ  F, \phi_\psi^a) =w^u_a(F, \phi^b)\quad \text{and}\quad  w^u_b(S_\psi \circ  F, \phi_\psi^b) =w^u_b(F, \phi^b)\; .\]
\end{lemm}
\begin{lemm}\label{wsf}
The following operators are of class $C^1$ and depends continuously on $F$:  
$$\psi \in C^{k-2}_0([x_a-2\tau +\delta, x_a-\delta ] , \R)\mapsto 
w^s_a(S_\psi \circ  F, \phi_\psi^a)
\in C^{k-2}(\R)\; ,$$
$$\psi \in C^{k-2}_0([x_b +\delta, x_b+2\tau -\delta ] , \R)\mapsto w^s_b(S_\psi \circ  F, \phi_\psi^b)\in C^{k-2}(\R)\; .$$
\end{lemm}
Lemma \ref{wsf} was rather unexpected: the proof works because $S_\psi$ induces a graph transform which is a translation (which is a smooth operator).  

\begin{proof}[Proof of Lemma \ref{wubfp}]
Let $\psi$ be supported inside $[x_a-2\tau+\delta, x_a-\delta]$. Denote $\bar F := S_\psi \circ F$.  The graph of the function
$w^u_a$ is the curve $\phi_\psi^a (W^a(P; \bar F))\cap 
[x_a-\tau, x_a]\times \R$.
It follows from our choice of the chart $\phi_\psi^a$
(see the first line of (\ref{mchr1a})) that this curve is the image by $\phi^a\circ F$ of an arc of the curve 
$\ell^u_a=\bar F^{-1} (W^a(P; \bar F)\cap [x_a-\tau, x_a]\times \R)
$, which is an arc of $W^a(P; \bar F)$ lying at $x\geq x_a$.
The set $W^a(P; \bar F)\cap \{x\ge x_a\}$  is a part of the unstable manifold   of $P$ which depends only
on the dynamics at $ \{x\ge x_a\}$.  Since $\psi$ is zero at $x\geq x_a$, the map $S_\psi$ is identity there,
hence $\bar F|_{ \{x\ge x_a\}}$ equals $F|_{ \{x\ge x_a\}}$, and   $ W^a(P; \bar F)\cap \{x\ge x_a\}$ equals  $W^a(P; F)\cap \{x\ge x_a\}$,
in particular the curve $\ell^u_a$ does not depend on $\psi$, i.e. it is the same for $F$ and $\bar F$.

Thus,
\[\phi_\psi^a (W^a(P; \bar F))|_{x\in [x_a-\tau, x_a]}=(\phi^a\circ S_{-\psi})\circ (S_\psi\circ F) (\ell^u_a)=\phi^a\circ F (\ell^u_a)=\phi^a(W^a(P; F))|_{x\in [x_a-\tau, x_a]}\;,\]
so $w^u_a(\bar F, \phi_\psi^a)=w^u_a(F, \phi^a)$ is the same for all small $\psi$.

Exactly in the same way, just by changing the index $a$ to $b$ and the interval $[x_a-\tau,x_a]$ to $[x_b,x_b+\tau]$, we obtain that
when $\psi$ is supported inside $[x_b+\delta, x_b+2\tau-\delta]$ the function $w^u_b(\bar F, \phi_\psi^b)$ is independent of $\psi$.
\end{proof}

\begin{proof}[Proof of Lemma \ref{wsf}]
For $c\in \{a,b\}$, let us show that $w^s_{c}(\bar F, \phi_\psi^{c})$ is a $C^1$ function of $\psi\in C^{k-2}(\R)$, with $\bar F := S_\psi \circ F$.
Again, we start with the case $c=a$ and $\psi$ supported inside $[x_a-2\tau+\delta, x_a-\delta]$ and derive
the expression for $w^s_a(\bar F, \phi_\psi^a)$ in this case.

The graph of $w^s_a(\bar F, \phi_\psi^a)$ is the curve $\phi_\psi^a \bar \ell_s^a$ where
$\bar \ell_s^a$ is the piece of $W^a(Q; \bar F)$ in the intersection with $(\phi_\psi^a)^{-1}([x_a-\tau, x_a]\times \R)$.
The curve $W^s(Q; \bar F)$ (a halve of the stable manifold of $Q$) is obtained by iterations of its small, adjoining to $Q$ part
by the map $\bar F^{-1}$. The maps $\bar F^{-1}$ and $F^{-1}$ coincide at $x\leq x_a-2\tau+\delta$, and both
are close to the map $\mathring F^{-1}$ which takes the line $\{x=x_a-2\tau+\delta\}\cap N^a$ into $x=x_a-\tau+\delta$, so
the piece of $W^s(Q; \bar F)$ between $Q$ and $x=x_a-\tau+\delta/2$ does not move as $\psi$ varies, i.e., it is the same as for the map $F$.

This means that for all small $\psi$, the image by $\bar F$ of the curve $\bar \ell^s_a$
lies inside the piece of $W^s(Q; F)$ to the left of $x=x_a-\tau+\delta/2$, i.e., $\bar F \bar \ell^s_a \subset F\ell^s_a$,
where the $\psi$-independent curve $\ell^s_a$ is a piece of $W^s(Q; F)$ in the intersection with $\{x\in[x_a-\tau-\delta/2, x_a+\delta/2]\}$. So
\begin{equation}\label{wsia}
\phi_\psi^a (W^a(Q; \bar F))\cap \{x\in [x_a-\tau, x_a]\} \subset \phi_\psi^a \circ \bar F^{-1} \circ F\; \ell^s_a = \phi^a\circ S_{-\psi} \circ F^{-1} \circ S_{-\psi}
\circ F \;\ell^s_a\;,
\end{equation}
as given by (\ref{mchr1a}),(\ref{fbdsp}).

\medskip

Formula (\ref{wsia}) states that the graph of $w^s_a(\bar F, \phi_\psi^a)$ is the image of the $\psi$-independent curve $\ell^s_a$ by the map 
$\phi^a\circ S_{-\psi} \circ F^{-1} \circ S_{-\psi}\circ F$. Thus, the regularity of the
map $\psi\mapsto w^s_a(\bar F, \phi_\psi^a)$ is determined by the regularity of the corresponding {\em graph transform} operators.

The following defines   the \emph{graph-transform operator} $\mathcal F^\#$  associated with a smooth map $\mathcal F$. 
\begin{fact}[See Thm 2.2.5. P. 145 \cite{Ha82}]
Consider a curve ${\mathcal L}=\{(x,y): y=w(x)\}$ in an $(x,y)$-plane, where $w$ is a $C^s$-smooth function defined on some closed interval,
and a $C^n$-smooth map ${\mathcal F}: (x,y)\mapsto (p(x,y), q(x,y))$ ($n\geq s$) defined in a neighborhood $U$ of ${\mathcal L}$.  Then, under the condition
$\partial_x p(x,y) + \partial_y p(x,y) Dw(x)\neq 0$ everywhere in $U$, the image $\mathcal {F L}$ is a curve of the form $y=\tilde w(x)$ where
$\tilde w \in C^s$. Moreover the operator $\mathcal F^\#$ which takes the $C^s$-function $w$ to the $C^s$-function $\tilde w$ is
of  regularity class $C^{n-s}$. 
\end{fact}
 Since $\phi^a(\ell^s_a)$ is the graph of $w^s_a(F, \phi^a)$, we have from (\ref{wsia}) that
\begin{equation}\label{wsifa}
w^s_a(\bar F, \phi_\psi^a)=(\phi^a)^\#\circ (S_{-\psi})^\# \circ (F^{-1})^\# \circ (S_{-\psi})^\#\circ (F\circ (\phi^a)^{-1})^\# \;w^s_a(F,\phi^a)\;.
\end{equation}
Since $\phi^a$ and $F$ are at least of class $C^{k-1}$, the graph-transform operators $(\phi^a)^\#$, $(F^{-1})^\#$, and $(F\circ (\phi^a)^{-1})^\#$
have regularity at least $C^1$ when act from $C^{k-2}$-smooth functions to $C^{k-2}$-smooth functions. We cannot use the same fact for the graph
transform operator induced by  $S_{-\psi}$ since the latter map is only of class $C^{k-2}$. However, the map $S_{-\psi}$ is given by 
$(x,y)\mapsto (x, y-\psi(x))$, so the associated graph
transform operator $(S_{-\psi})^\#$ sends a function $w$ to $w-\psi$. Thus, it is linear in both $w$ in $\psi$, i.e., it is of class $C^\infty$ with respect
to both $w$ and $\psi$ (irrespective of their class of smoothness). 

Altogether, this implies that the map $\psi\mapsto w^s_a(\bar F, \phi_\psi^a)$ given by
(\ref{wsifa}) is of class $C^1$.

\medskip

Now, let us handle the case $c=b$. 

We need to derive the expression for $w^s_b(\bar F, \phi_\psi^b)$. The graph of this function is the curve $\phi_\psi^b \bar \ell^s_b$ where
$\bar \ell^s_b$ is the piece of $W^b(Q; \bar F)$ in the intersection with $(\phi_\psi^b)^{-1}(\{x\in [x_b, x_b+\tau]\})$.
The curve $W^b(Q; \bar F)$ is close to $W^b(Q; \mathring F)$. Recall that $W^b(Q,\mathring F)$ coincides with
$W^b(P,\mathring F)$ and forms a heteroclinic link $\mathring L_b$. It intersects the vertical strip
$V_b=\{x\in [x_b,x_b+2\tau]\}$ twice, along two straight line segments 
$\mathring D_2^b=[x_b,x_b+2\tau]\times \{y=y_b\}$ and 
$\mathring D_4^b=\mathring F^4(\mathring D_2^b)\cup \mathring F^6(\mathring D_2^b)=[x_b,x_b+2\tau]\times \{y=y_b'\}$
for some $y_b'<y_b$. The piece
of $W^b(Q,\mathring F)$ between $Q$ and the left end of $\mathring D_4^b$ lies entirely in the region $x\leq x_b$,
i.e., to the left of the vertical strip $V_b$. Also, the map $\mathring F$ in a small neighborhood of 
$\mathring D_4^b$ is given by
\begin{equation}\label{mfrd4b}
\mathring F: (x,y)\mapsto (x-\tau/2,y)
\end{equation}
(as implied by the link suitability Definition \ref{suit}). It follows that the curve $W^b(Q,\bar F)$ intersects the straight 
line $x=x_b$ at a point $Z_b$ with the $y$-coordinate close to $y_b$ such that
the piece of $W^b(Q,\bar F)$ between $Q$ and $\bar F^7  Z_b$ lies entirely in the region
$x< x_b+\tau/2+\delta/2$.
Since $\psi$ is supported inside $[x_b+\delta, x_b+2\tau-\delta]$, it follows that  $F=\bar F$ near this piece
of $W^b(Q,\bar F)$. Therefore, the piece of $W^b(Q,\bar F)$ between $Q$ and $x= x_b+\tau/2+\delta/2$ is 
unmoved as $\psi$ varies,
so it coincides with the corresponding piece of $W^b(Q,F)$.

This piece contains the point $\bar F^7 Z_b$, hence it contains the curve $\bar F^7 \bar \ell_s^b$ for all small $\psi$.
It follows that there exists a $\psi$-independent curve $\ell^s_b$ so that $\bar F^7 \bar \ell^s_b \subset F^7\ell^s_b$ and
$\ell^s_b$ is a piece of of $W^s(Q; F)$ in the intersection with $\{x\in[x_b-\delta/2, x_b+\tau+\delta/2]\}$. So
\begin{equation}\label{wsib}
\phi_\psi^b (W^b(Q; \bar F))|_{x\in [x_b, x_b+\tau]} \subset \phi_\psi^b \circ \bar F^{-7} \circ F^7\; \ell^s_b = \phi^b\circ S_{-\psi} \circ( \bar F^{-1}\circ
S_{-\psi})^7\circ F^7 \;\ell^s_b\;,
\end{equation}
see (\ref{mchr1b}),(\ref{fbdsp}).

Since $\phi^b(\ell^s_b)$ is the graph of $w^s_b(F, \phi^b)$, we have from (\ref{wsib}) that
\begin{equation}\label{wsifb}
w^s_b(\bar F, \phi_\psi^b)=(\phi^b)^\#\circ (S_{-\psi})^\# \circ ((F^{-1})^\# \circ (S_{-\psi})^\#)^7\circ (F^7\circ (\phi^b)^{-1})^\# \;w^s_b(F,\phi^b)\;.
\end{equation}
Like we did it for the function $w^s_a$ given by (\ref{wsifa}), we obtain that the map $\psi\mapsto w^s_b(\bar F, \phi_\psi^b)$ defined by
(\ref{wsifb}) is of class $C^1$. 
\end{proof}

The continuity of the operators $\mathcal M^{a}$ and $\mathcal M^{b}$ with respect to $F$ allows us to obtain enough information about 
them by computing them for $F=\mathring F$, which we do in the two following lemmas.

\begin{lemm} When $F=\mathring F$ :
\begin{equation}\label{phiforma}
\mathcal M^a(\psi)(x) = \psi(x)+\psi(x-\tau) \quad\mbox{  for } x\in[x_a-\tau,x_a]. 
\end{equation}
\end{lemm}
\begin{proof}
For $F=\mathring F$, the link $L^a$ exists and the fundamental interval $D^a$ is straight, so the curves 
$W^a(P; \mathring F)$ and $W^a(Q; \mathring F)$ coincide for $x\in [x_a-\tau,x_a]$ and lie in the straight line $y=y_a$.
The map $\phi^a$ for $\mathring F$ is identity, so we have $w^u_a(\mathring F, \phi^a) = w^s_a(\mathring F, \phi^a)=y_a$. The map $\mathring F$
is the translation to $(-\tau,0)$ (see (\ref{mathrtra})). Plugging this information into  (\ref{wsifa}) gives 
(\ref{phiforma}) immediately.
\end{proof}

\begin{lemm} When $F=\mathring F$,
\begin{equation}\label{phiformb}\begin{array}{l}
\mathcal M^b(\psi)(x) = \psi(x)+\psi(x+\tau) -\;\frac12\; (\psi(\frac{3x_b+\tau-x}2) + \psi(\frac{3x_b+2\tau-x}2)+\psi(\frac{3x_b+3\tau-x}2)+
\psi(\frac{3x_b+4\tau-x}2))\end{array}
\end{equation}
for $x\in[x_b, x_b+\tau]$. 
\end{lemm}
\begin{proof}
When $F=\mathring F$, we have $\phi^b=id$, so
\begin{equation}\label{prbrfr}
\mathcal M^b: \psi \mapsto w^u_b(\mathring F,id) -  
(S_{-\psi})^\# \circ ((\mathring F^{-1})^\# \circ (S_{-\psi})^\#)^7\circ (\mathring F^7)^\# \;w^s_b(\mathring F,id)
\end{equation}
(see (\ref{wsifb})). Recall that $W^b(Q,\mathring F)$ coincides with $W^b(P,\mathring F)$ and intersects the vertical
strip $V_b$ on two straight line segments, $\mathring D_2^b=[x_b,x_b+2\tau]\times \{y=y_b\}$ and 
$\mathring D_4^b=[x_b,x_b+2\tau]\times \{y=y_b'\}$, so 
\begin{equation}\label{wubwsbr}
w^u_b(\mathring F,id)=w^s_b(\mathring F,id)\equiv y_b \;\;\mbox{   at   } x\in[x_b,x_b+\tau].
\end{equation}
Also, the map $\mathring F^7$ from a small neighborhood of $[x_b,x_b+\tau]\times \{y=y_b\}$
to a small neighborhood of $[x_b,x_b+\tau]\times \{y=y_b'\}$ acts as
\[\mathring F^7: (x,y) \mapsto (-\frac x2+ \frac{3x_b}2 +\frac{\tau}2, -2y +2y_b+y_b')\] 
(see Definition \ref{suit}). Therefore,
\[(\mathring F^7)^\# \;w^s_b(\mathring F,id) \equiv y_b' \;\;\mbox{   at   } x\in[x_b, x_b+\tau/2]\;,\]
and 
\[w_7(x):=(S_{-\psi})^\# (\mathring F^7)^\# \;w^s_b(\mathring F,id)(x) = y_b'-\psi(x) \;\;\mbox{   at   } 
x\in[x_b, x_b+\tau/2]\;.\]
Next, by (\ref{mfrd4b}), we obtain
\[w_6(x):=(S_{-\psi})^\# \circ (\mathring F^{-1})^\# w_7(x) = 
y_b' - \psi(x-\tau/2)-\psi(x) \;\;\mbox{   at   }  x\in[x_b+\tau/2, x_b+\tau]\;.\]
Repeating the same procedure two more times, we obtain
\[w_4(x):=((S_{-\psi})^\# \circ (\mathring F^{-1})^\#)^3 w_7(x) = 
y_b' - \psi(x-\frac 32\tau)- \psi(x-\tau)-\psi(x-\frac \tau 2)-\psi(x) \mbox{   at   }  x\in[x_b+\frac 32 \tau, x_b+2\tau]\;.\]

The map $\mathring F^{-1}$ takes the rectangle $K:=[x_b+3\tau/2,x_b+2\tau]\times [y_b'-\varepsilon, y_b'+\varepsilon]$
(for some small $\varepsilon$) to the right of the strip $V_b$, and its next image $\mathring F^{-2}K$ also lies to the
right of $V_b$. We recall that the function $\psi$ vanishes there and $S_{-\psi}=id$. Therefore, 
$F^{-1}\circ(S_{-\psi}\circ F^{-1})^2=F^{-3}$ on $K$. As the curve $\{y=w_4(x)\}$ lies in $K$ for small $\psi$,
this implies that
\[(F^{-1}\circ(S_{-\psi}\circ F^{-1})^2)^\# w_4=(F^{-3})^\# w_4\;.\]
By the link suitability conditions (Definition \ref{suit}), the map $F^{-3}$ takes $K$ into a small neighborhood
of $\mathring D_2^b$ and is given by
\[\mathring F^{-3}: (x,y) \mapsto (-2x+ 3x_b +5\tau, -\;\frac y2 + \frac{y_b'}2 +y_b).\] 
This gives
\[\begin{array}{l}
((S_{-\psi}\circ F^{-1})^6)^\# w_7(x)=y_b +\; 
\frac12 (\psi(\frac{3x_b+2\tau-x}2) + \psi(\frac{3x_b+3\tau-x}2)+\psi(\frac{3x_b+4\tau-x}2)+
\psi(\frac{3x_b+5\tau-x}2))\\
\qquad\qquad\qquad\qquad\qquad\qquad\qquad\qquad\qquad -\psi(x) \;\mbox{   at   }  x\in[x_b+\tau, x_b+2\tau]\;.
\end{array}\]
Since $\mathring F^{-1}$ is the translation to $(-\tau,0)$ near $\mathring F(D_2^b)$, we finally obtain
that
\[\begin{array}{l}((S_{-\psi}\circ F^{-1})^7)^\# w_7(x)=y_b+ \;
\frac12 (\psi(\frac{3x_b+\tau-x}2) + \psi(\frac{3x_b+2\tau-x}2)+\psi(\frac{3x_b+3\tau-x}2)+
\psi(\frac{3x_b+4\tau-x}2))\\
\qquad\qquad\qquad\qquad\qquad\qquad\qquad\qquad\qquad-\psi(x+\tau)-\psi(x) \;\;\mbox{   at   }  x\in[x_b+\tau, x_b+2\tau]\;,
\end{array}\]
which implies (\ref{phiformb}) by (\ref{prbrfr}),(\ref{wubwsbr}).
\end{proof}

\subsection{Proof of Lemma \ref{propbilinksimple1}: Restoring the link $L^a$}\label{lm1}

In order to prove Lemma \ref{propbilinksimple1}, we show that for every $\omega$-preserving $F$ which is
$C^k$-close to 
$\mathring F$ there exists a $C^{k-2}$-smooth function $\psi$, supported in $x\in (x_a-2\tau, x_a)$, such that
\begin{equation}\label{bsme}
M^a(S_\psi\circ F, \phi_\psi^a)\equiv 0\;,
\end{equation}
where the $C^{k-2}$-smooth time-energy chart $\phi_\psi^a$ for the map $S_\psi \circ F$ is defined by (\ref{mchr1a}).
We will solve this equation for $\psi$ by reducing it to a fixed point problem for a certain contracting operator.

First, we restrict the class of perturbation functions $\psi$; namely, we will take them in the form
\[\psi(x)=\rho(x)\tilde\psi(x)\;,\]
where $\tilde \psi$ is a $\tau$-periodic function, and $\rho\in C^\infty(\R, [0,1])$ has support in $[x_a-2\tau, x_a]$
and satisfies $\rho(x)+\rho(x-\tau)=1$ for every $x\in [x_a-\tau, x_a]$. Then the operator
\[\mathcal M^a_\rho : \tilde \psi \mapsto \mathcal M^a (\rho\tilde\psi)\;,\]
which provides a correspondence between $\tilde \psi$ and the link-splitting function for the map 
$S_{\rho\tilde\psi}\circ F$,
takes a small ball around zero in the space $C^{k-2}(\R/(\tau \Z) ,\R)$ of $\tau$-periodic functions into the same space. 

By Lemma \ref{prepropbililinksimple1}, the operator $\mathcal M^a_\rho$ is of class $C^1$ on $C^{k-2}(\R/(\tau \Z) ,\R)$.
For $F=\mathring F$, we have $\mathcal M^a_\rho=id$. Indeed, by (\ref{phiforma}),
\[\mathcal M^a_\rho(\tilde\psi)(x) = \rho(x)\tilde\psi(x)+\rho(x-\tau)\tilde\psi(x-\tau) = 
[\rho(x)+\rho(x-\tau)]\tilde\psi(x)=\tilde\psi(x)\;.\]
Due to a continuous dependence on $F$, the operator
\[\tilde \psi \mapsto \tilde \psi - \mathcal M^a_\rho(\tilde\psi)\]
is a contraction in a neighborhood of zero in $C^{k-2}(\R/(\tau \Z) ,\R)$ for all $F\in\Diff^k_\omega(\D)$ which 
are $C^k$-close to $\mathring F$.
Therefore, it has a unique fixed point $\tilde\psi$ near zero in $C^{k-2}(\R/(\tau \Z) ,\R)$, for each such $F$. 

The operator  $\mathcal M^a_\rho$ vanishes at this fixed point. By construction, the corresponding function 
$\psi=\rho\tilde\psi$ solves equation (\ref{bsme}),
i.e., the link-splitting function is identically zero for the map $\bar F=S_{\psi}\circ F$, meaning that 
the link $L^a$ persists for this map. Lemma \ref{propbilinksimple1} is proven.$\qed$

\subsection{Proof of Lemma \ref{propbilinksimple2}: Restoring the link $L^b$}\label{lm2}
In order to prove Lemma \ref{propbilinksimple2}, we show that for every $\omega$-preserving $F$ which is
$C^k$-close to $\mathring F$, if the link $L_a$ persists for the map $F$, then there exists a $C^{k-2}$-smooth function $\psi$, 
supported in $x\in (x_b, x_b+2\tau)$, such that
\begin{equation}\label{bsmeb}
M^b(S_\psi\circ F, \phi_\psi^b)\equiv 0\;,
\end{equation}
where the $C^{k-2}$ time-energy chart $\phi_\psi^b$ for the map $\bar F=S_\psi\circ F$ is defined by (\ref{mchr1b}).

In order to resolve equation (\ref{bsmeb}), we will use the following property of the link-splitting function $M^b$.
\begin{lemm}\label{zeromean}
If the link $L^a$ is persistent for the map $F$, then the link-splitting function
$M^b(S_\psi\circ F, \phi_\psi^b)$ has zero mean for every $\psi$ supported by $(x_b,x_b+2\tau)$:
\[\int_{x_b}^{x_b+\tau} M(\bar F, \phi_\psi^b) dx=0\;.\] 
\end{lemm}
\begin{proof} 
We may always assume that the map $\phi_\psi^b$ is the restriction to $N_a$ of an area-preserving diffeomorphism or $\R^2$ 
(the possibility of the symplectic extension of an area-preserving diffeomorphism from a disc to the whole $\R^2$ is a standard
fact; see e.g. Corollary 4 in \cite{Av10}).

Let the points $Z^u\in W^b(P,\bar F)\cap N^b$ and $Z^s\in W^b(Q,\bar F)\cap N^b$ have the same $x$-coordinate $x_b$,
i.e., $Z^u=(x_b, w^u(\bar F, \phi_\psi^b)(x_b))$, $Z^s=(x_b, w^s(\bar F, \phi_\psi^b)(x_b))$. As the map $\bar F$ is the translation
to $(\tau,0)$ in the time-energy coordinates in $N^b$, the points $\bar FZ^s$ and $\bar FZ^u$ have the same $x$-coordinate $x_b+\tau$,
and the image of the vertical segment connecting $Z^s$ and $Z^u$ is the vertical segment connecting $\bar FZ^s$ and $\bar FZ^u$.

Since the support of $\psi$ lies in $(x_b,x_b+2\tau)$, the maps $F$ and $\bar F$ coincide in a neighborhood of $L_a$,
so the link $L_a$ is not split for all $\psi$ under consideration. Thus, we may consider a region $\mathcal D$ bounded
by the link $L^a$, the piece of $W^b(P,\bar F)$ between $Z^u$ and $P$, the piece of $W^b(Q,\bar F)$ between $Q$ and $Z^s$, and
the vertical segment that connects $Z^s$ and $Z^u$. The region $\bar F \mathcal D$ has the same area as $\mathcal D$.
It is bounded by the link $L^a$,
the piece of $W^b(P,\bar F)$ between $\bar F Z^u$ and $P$,  the piece of $W^b(Q,\bar F)$ between $Q$ and $\bar F Z^s$, and
the vertical segment that connects $\bar F Z^s$ and $\bar F Z^u$. The equality of the areas means that the area of the region
between the curves $W^b(P,\bar F): \{y=w^u(\bar F, \phi_\psi^b)(x)\}$ and $W^b(Q,\bar F): \{y=w^s(\bar F, \phi_\psi^b)(x)\}$
at $x\in[x_b,x_b+\tau]$ is zero (see fig. \ref{int:null}).
\begin{figure}[h!]
	\centering
		\includegraphics{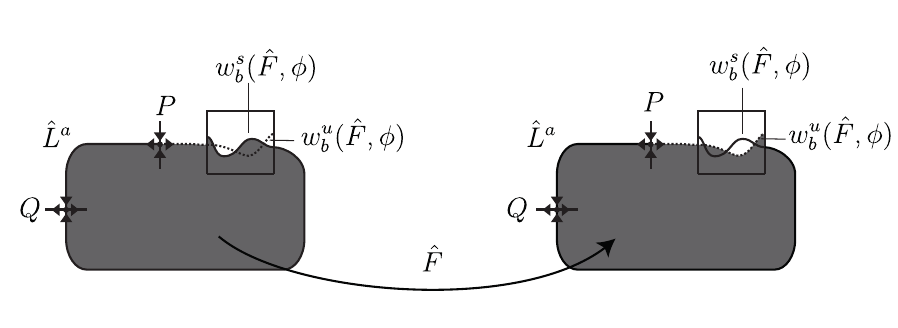}
\caption{the integral of the function  $M(\bar F, \phi)= w^u_b(\bar F,\phi)-w^s_b(\bar F,\phi)$ is null.}
\label{int:null}
\label{fig:oldcoord}
\end{figure}
 This means that $\int_{x_b}^{x_b+\tau} (w^u(\bar F, \phi_\psi^b)-w^s(\bar F, \phi_\psi^b)) dx=0$,
which proves the lemma (see Definition \ref{lbd} of the link-splitting function).
\end{proof}

Now, like in the previous Section, we restrict the class of perturbation functions $\psi$:
\[\psi(x)=\rho(x)\tilde\psi(x)\;,\]
where $\tilde \psi$ is a $\tau$-periodic function with the zero mean, and $\rho\in C^\infty(\R, [0,1])$ has support in $[x_b+\delta, x_b+2\tau-\delta]$ for $\delta>0$ small, 
and satisfies $\rho(x)+\rho(x+\tau)=1$ for every $x\in [x_b, x_b+\tau]$. Then, if the link $L_a$ persists for the map $F$, the operator
\[\mathcal M^b_\rho : \tilde \psi \mapsto \mathcal M^b (\rho\tilde\psi)\;,\]
takes a small ball around zero in the space $C^{k-2}_0(\R/(\tau \Z) ,\R)$ of $\tau$-periodic functions with the zero mean
into the same space (by Lemma \ref{zeromean}).

By Lemma \ref{prepropbililinksimple1}, the operator $\mathcal M^b_\rho$ is of class $C^1$ on $C^{k-2}_0(\R/(\tau \Z),\R)$.
It follows easily from (\ref{phiformb}) that if $F=\mathring F$, then
\begin{equation}\label{phiformbb}
\mathcal M^b_\rho(\tilde\psi)(x) = \tilde\psi(x) -\;\frac12\; (\tilde\psi(\frac{3x_b+\tau-x}2) + \tilde\psi(\frac{3x_b+2\tau-x}2)).
\end{equation}

Note that the space $C^{k-2}_0(\R/(\tau \Z) ,\R)$ is a Banach space when endowed with the following norm:
\begin{equation}\label{0norm}
\|\psi\|_{C^{k-2}_0}= \max_{1\le i\le k-2} \|D^i \psi_1-D^i \phi_2\|_{C^0}\;.
\end{equation}
This norm includes evaluation of the derivatives only (still it is a well-defined norm - if two functions with the zero mean have the
same derivative, they coincide).

By (\ref{phiformb}) and by the continuous dependence of $\mathcal M^b$ on $F$ (see Lemma \ref{prepropbililinksimple1}), the operator
$id-\mathcal M^b_\rho$ is $C^1$-close to the linear operator $\tilde \psi \mapsto \bar\psi$ where
\[\bar\psi(x)=\frac12\; (\tilde\psi(\frac{3x_b+\tau-x}2) + \tilde\psi(\frac{3x_b+2\tau-x}2)).\]
We have
\[D^i\bar\psi(x)=\frac{(-1)^i}{2^{i+1}}\; (D^i\tilde\psi(\frac{3x_b+\tau-x}2) + D^i\tilde\psi(\frac{3x_b+2\tau-x}2)),\]
so this operator is, obviously, a contraction in the norm (\ref{0norm}). Therefore, $id-\mathcal M^b_\rho$ is a contraction on
$C^{k-2}_0(\R/(\tau \Z) ,\R)$ for all $F$ which are $C^k$-close to $\mathring F$, provided $F$ is area-preserving and the link $L_a$ persists for $F$.

Thus, $id-\mathcal M^b_\rho$ has a fixed point $\tilde \psi$. The corresponding function 
$\psi=\rho\tilde\psi$ solves equation (\ref{bsmeb}); Lemma \ref{propbilinksimple2} is proven.$\qed$

\subsection{Proof of Proposition \ref{rtoinf}}\label{proofrtoinf}

The previous results allow for construction of $C^r$-maps with stochastic islands for any finite $r$. Below we prove
Proposition \ref{rtoinf} which deals with the $C^\infty$ case. 
Let $\hat f\in {\rm Diff}^r_\omega(\M)$ have a stochastic island $\mathcal I$ bounded by bi-links $(L^a_i\cup L_i^b)_{i=1}^m$ so that 
each bi-link $C_i:=L^a_i\cup L_i^b$ is a $C^r$-smooth circle (without break points).
Let us show that arbitrarily close in $C^r$ to $\hat f$ there exists a map $\hat f_\infty\in {\rm Diff}^\infty_\omega(\M)$ for which the bi-links persist.

Choose a map $f\in {\rm Diff}^\infty_\omega(\R^2)$ which is sufficiently close in $C^r$ to $\hat f$; such 
exists by Zehnder smoothing theorem \cite[Thm.~1]{Ze77}. The bi-links do not need to persist for $f$. To restore them, the idea is to smoothen 
the circles  $C_i$ to $C^r$-close circles $ \tilde C_i$ which are of class $C^\infty$. Then we will perform a local surgery to construct 
$\hat f_\infty$ of class $C^\infty_\omega$  which is $C^r$-close to $f$ and such that 
$\hat f_\infty(\sqcup_i \tilde C_i)=\sqcup_i \tilde C_i$. 
By local maximality of the hyperbolic continuation $(\tilde P_i,\tilde Q_i)_i$  of saddle points $(P_i,Q_i)_i$ defining the bi-links
$(L^a_i\cup L_i^b)_{i=1}^m$, it follows that each $\tilde P_i,\tilde Q_i$ belongs to $\tilde C_i$. Moreover, as the unique invariant curves which contain $\tilde P_i$ or 
$\tilde Q_i$ are their local stable and unstable manifolds, we obtain that the circles $\tilde C_I$ are heteroclinic bi-links, i.e., the bi-links are persistent for $\hat f_\infty$. 

Consequently, we need only to prove the following
\begin{lemm}\label{restorCinfty}
There exists a collection of $C^\infty$-circles $\tilde C_i$ 
which are $C^r$-close to $C_i$, and a map $\hat  f_\infty$ of class $C^\infty_\omega$ which is $C^r$-close to $f$, such that  
$\hat f_\infty(\sqcup_i \tilde C_i)= \sqcup_i \tilde C_i$.
\end{lemm}
\begin{proof}
Choose $\varepsilon>0$ small enough. Let $(x_i(t),y_i(t))_{t\in \S^1}$, where $(\frac{dx_i}{dt})^2+(\frac{dy_i}{dt})^2\neq 0$,
be a parameterization of a closed $C^\infty$-curve, which is sufficiently close to $C_i$ in the $C^r$-topology. Consider the map $(t,h)\mapsto (x,y)$ defined as
$$x = x_i(t) - z(t,h) \frac{dy_i(t)}{dt}, \qquad y = y_i(t) + z(t,h) \frac{dx_i(t)}{dt},$$
where $z(t,0)=0$ and 
$$\partial_h z=\left((\frac{dx_i}{dt})^2+(\frac{dy_i}{dt})^2 + z (\frac{d^2x_i}{dt^2}\frac{dy_i}{dt}-\frac{d^2y_i}{dt^2}\frac{dx_i}{dt})\right)^{-1}.$$
It is easy to check that $\det (\partial_{t,h}(x,y))\equiv 1$, so this map is a symplectic $C^\infty$-diffeomorphism from $\S^1\times \{|h|\leq\varepsilon\}$
onto a small neighborhood $V_i$ of $C_i$ in $\R^2$ (we think of $(x,y)$ as the Cartesian coordinates in $\R^2$). Thus, $(t,h)$ are symplectic coordinates in $V_i$
and the curve $C_i$ is given by an equation $h=\sigma_i(t)$ where $\sigma_i$ is $C^r$-small.

Let $\tilde C_i$ be a curve $h=\tilde\sigma_i(t)$ where $\sigma_i$ is a $C^\infty$-function, which is sufficiently close to $\sigma_i$ in $C^r$.
By adding a small constant to $\tilde\sigma_i$, if necessary, we can ensure that the disc bounded by the curve $\tilde C_i$ in $\R^2$ has
the same area as the disc bounded by the curve $C_i$. We do this for all $i=1, \dots, m$. Recall that the collection of curves $C_i$ is invariant
with respect to the map $\hat f$ and every curve $C_i$ is invariant with respect to some iteration of $\hat f$. So, there is $j$ such that
$C_i=\hat f(C_j)$, and $C_j$ is the image of $C_i$ by some iteration of $\hat f$. In particular, $C_j$ and $C_i$ bound discs of equal area in $\R^2$.
Therefore, the curves $\tilde C_i$ and $\tilde C_j$ also bound discs of the same area. 

Since $\tilde C_j$ is close to $C_j$ and $\tilde C_i$ is close to $C_i=\hat f(C_j)$, we can always assume that the symplectic
$C^\infty$-diffeomorphism $f$ (that $C^r$-approximates the map $\hat f$) is chosen such that the closed curve $f(\tilde C_j)$ lies in $V_i$
and has an equation $h=\mathring\sigma_i(t)$ with some $C^r$-small function $\mathring\sigma_i$. The discs bounded
by the curves $\tilde C_i$ and $f(\tilde C_j)$ in $\R^2$ have the same area, so the area between the curves $\tilde C_i$ and $f(\tilde C_j)$ is zero, which means that
$$\int_{\S^1} (\tilde\sigma_i(t)-\mathring\sigma_i(t)) dt =0.$$
Therefore, on $\S^1$ there exists a ($C^{r+1}$-small) function $\Phi_i(t)$ such that
$$\frac{d}{dt} \Phi_i(t) = \tilde\sigma_i(t)-\mathring\sigma_i(t).$$

Now, we define a symplectic diffeomorphism $\phi_i :V_i\to V_i$ by a generating function $S(t,\bar h)= t \bar h -\Phi_i(t) \xi(\bar h)$,
where $\xi$ is a $C^\infty$-function which vanishes identically for $|\bar h|\geq \varepsilon$ and equals identically to $1$ on the range
of values of the function $\tilde\sigma_i$. Namely, we define $\phi_i: (t,h)\mapsto (\bar t,\bar t)$ implicitly by the rule
$\bar t =\partial_{\bar h} S$, $h = \partial_t S$,
or
\begin{equation}\label{implgf}
\bar t =t - \Phi_i(t) \frac{d\xi(\bar h)}{d\bar h}, \qquad h = \bar h - (\tilde\sigma_i(t)-\mathring\sigma_i(t)) \xi(\bar h).
\end{equation}
These formulas indeed define a single-valued map, $C^r$-close to identity, because $\tilde\sigma_i(t)-\mathring\sigma_i(t)$ is  
$C^r$-small. By construction, the map $\phi_i$ equals to identity at the boundary of $V_i$ (at $|h|=\varepsilon$), so $\phi_i$ 
is indeed a diffeomorphism of $V_i$; it is also immediate that it preserves the symplectic form $dt\wedge dh$. 

It remains to note that $\phi_i$ sends the curve $f(\tilde C_j): \bar h=\mathring \sigma_i(\bar t)$ to the curve $\tilde C_i: h=\tilde\sigma_i(t)$
(since $\xi$ is identically $1$ everywhere in the range of $\tilde\sigma_i$, it follows from (\ref{implgf}) that $\bar t= t$,
$h-\mathring\sigma_i(t)=\bar h - \tilde\sigma_i(\bar t)$ for every point $(\bar t,\bar h)\in \tilde C_i$). Thus, if we perform this construction for all $i$ and define a $C^r$-close
to identity symplectic $C^\infty$-diffeomorphism $\psi:\R^2\to \R^2$ as $\phi_i$ inside $V_i$ and identity outside $\sqcup_i V_i$, then
the map $\hat f_\infty=\psi\circ f$ wil send the set of the $C^\infty$-smooth curves $\tilde C_i$ into itself, as required.
\end{proof}

\section{Rescaling Lemma}\label{section5} 
We recall that given $U\subset \R^2$, the set  $Diff^\infty_\omega(U, \R^2)$ denotes the set of symplectic diffeomorphisms $\phi$ from $U$ onto their image. We endow this space with the complete metric:
$$d_{C^\infty}(\phi, \tilde \phi)=\sum_r 2^{-r} \min(1, \|\phi-\tilde \phi\|_{C^r})\, .$$
 
Consider a symplectic $C^{\infty}$-diffeomorphism $f$ of $\R^2$ which has a saddle periodic point $O$. Assume that there exists a {\em homoclinic band}, 
i.e., the intersection of the stable and unstable manifolds of $O$ contains a closed interval $J$ of non-zero length. 
We assume that $J$ is sufficiently small (if not, take its sufficiently small subinterval), so that $J$  lies entirely inside 
a fundamental domain, implying that $f^m(J) \cap J=\emptyset$ for all $m\neq 0$. By considering, if necessary, an iterate of $J$, we also assume that $J$ lies close to $O$ in a local stable manifold of $O$.

Given a function $\psi\in C^r(\R)$, we define a symplectic H\'enon-like map $H_{\psi}$ as
\begin{equation}\label{henonpsi}
H_{\psi}(x,y)= (y,  -x +\psi(y)).
\end{equation}

The following statement enables to produce a perturbation of the dynamics displaying a renormalization arbitrarily 
close to any composition of H\'enon-like maps. It is an improved and more developed\footnote{The current formulation includes many rounds near a homoclinic tangency; see also Remark \ref{important remark}.} version of the rescaling lemma 
from \cite{GST07}.  
We recall that $\D$ denotes the closed unit disk of $\R^2$.\\

\begin{RL} For every odd $N\in \N$, for every $r\ge 1$,  for all  $L,\delta>0$, and for every neighborhood $U_J$ of $J$, there exist:
\begin{itemize}
\item an integer $n>1$,
\item  sympletic  diffeomorphisms   $\Phi_{1}, \dots, \Phi_{N}\in Diff^\infty_\omega (\R^2)$ which are $\delta$-close to the identity.
\item  a $C^\infty$-diffeomorphism $Q:\R^2\to \R^2$ with constant Jacobian satisfying $Q(\D)\subset U_J$, 
\end{itemize}
such that for any  functions $\psi_1, \dots, \psi_N\in C^r(\R)$ whose $C^r$-norms are bounded by $L$, there exists a symplectic 
$C^r$-diffeomorphism $\hat f$ which coincides with $f$ on the complement of $U_J$ and satisfies:
\begin{equation}\label{dltclsfhat}
\|\hat f - f\|_{C^r} < \delta,
\end{equation}
\begin{equation}\label{rescaling}
Q^{-1} \circ \hat f^{n} \circ Q|_{\D} = H_{\psi_N}\circ \Phi_{N}\circ \dots \circ H_{\psi_1} \circ \Phi_{1}|_{\D}\; ,
\end{equation}
where $H_{\psi_i}$ are the symplectic H\'enon-like maps defined by \eqref{henonpsi}. Moreover:
\begin{equation}\label{disjoint1}
\hat f^m\circ Q(\D)\cap Q(\D) = \varnothing\quad \mbox{for al} \quad  m=1, \dots, n-1 .
\end{equation}\end{RL}

\medskip

\begin{rema}\label{important remark}
It is important for us that the maps $\Phi_i$ in formula (\ref{rescaling}) are independent
of the choice of the functions $\psi_1,\dots, \psi_N$.
\end{rema} 

\begin{cor}\label{qq} Let $ r\ge 1$. Take any symplectic $C^\infty$-diffeomorphism $F: {\D} \to \R^2$. 
For any $\delta>0$, $L>0$, and a neighborhood $U_J$ of $J$, there exist:
\begin{itemize}
\item   a $C^\infty$-diffeomorphism $Q:\R^2\to \R^2$ with constant Jacobian satisfying $Q(\D)\subset U_J$, 
\item a symplectic diffeomorphism $\hat F: {\D} \to \R^2$ which is $\delta$-close to $F$ in the $C^\infty$-metric,
\end{itemize}
such that  for every function $\psi\in C^r( \R)$ with $C^r$-norm bounded by $L$,  there exists a symplectic 
diffeomorphism $\hat f$, which is $\delta$-close to $f$ in the $C^r$-norm, coincides with $f$ on the complement of $U_J$, 
and its renormalization at the disc $Q(\D)$ for some $n\ge 1$ satisfies:
\begin{equation}\label{rescalingpsi}
Q^{-1}\circ \hat f^n\circ Q|_{\D} = S_{\psi}\circ \hat F,
\qquad
\text{with }S_\psi: (x,y)\mapsto (x, y +\psi(x)) ,
\end{equation}
\begin{equation}\label{disjoint2}
\hat f^m\circ Q(\D)\cap Q(\D) = \varnothing\quad \mbox{for all} \quad m=1, \dots, n-1 .
\end{equation}
\end{cor}

\begin{proof} By Theorem 2 of \cite{Tu03}, every symplectic $C^\infty$-diffeomorphism of any two-dimensional disc to 
$\R^2$ can be arbitrarily well approximated by a composition of an {\em even} number of H\'enon-like maps of 
the form (\ref{henonpsi}). In particular, for every $\delta>0$, for the given $C^\infty$-diffeomorphism $F$ there exists 
an even number $N'$ of functions $\psi_1,\dots,\psi_{N'}$
such that
$$d_{C^\infty}(H_0^{-1}\circ F , H_{\psi_{N'}}\circ \dots \circ H_{\psi_1})<\frac{\delta}{2},$$
where $H_0(x,y)=(y,-x)$ is the H\' enon-like map associated to the zero function $\psi$ in (\ref{henonpsi}). 
Since $H_0$ is just a linear rotation, it follows that 
\begin{equation}\label{aprx}
d_{C^\infty}(F , H_{0}\circ H_{\psi_{N'}}\circ \dots \circ H_{\psi_1})<\frac{\delta}{2},
\end{equation}
Note that the map $S_\psi$ defined in (\ref{rescalingpsi}) satisfies $H_\psi=S_\psi\circ H_0$.
Since $N'+1$ is odd, we can apply the Rescaling Lemma with $N=N'+1$ and with the sequence 
of functions $\psi_1, \dots, \psi_{N'}$ and $\psi_{N'+1}\equiv \psi$. This gives us the perturbation $\hat f$ to $f$ 
such that  (\ref{rescaling}) is fulfilled:
\begin{equation}\label{prerescalingpsi}  Q^{-1} \circ \hat f^{n} \circ Q|_{\D} = H_{\psi_{N'+1}}\circ \Phi_{N'+1}\circ \dots \circ H_{\psi_1} \circ \Phi_{1}|_{\D}=
S_{\psi}\circ H_0 \circ \Phi_{N'+1}\circ \dots \circ H_{\psi_1} \circ \Phi_{1}|_{\D} ,
\end{equation}
for some $n>1$. We recall that the Rescaling Lemma asserts that $\hat f$ coincides with $f$ 
on the complement of a small neighborhood $U_J$ of $J$  and that $Q$ is a diffeomorphism with constant Jacobian such that $Q(\D)\subset U_J$.  
By \eqref{prerescalingpsi}, we notice that \eqref{rescalingpsi} is satisfied with:
$$\hat F=H_0 \circ \Phi_{N'+1}\circ  H_{\psi_{N'}} \dots \circ H_{\psi_1} \circ \Phi_1|_{\D}.$$
Since the maps
$\Phi_i$ can be made as close to identity as we want, we can make
$$d_{C^\infty}(\hat F , H_{0}\circ H_{\psi_{N'}}\circ \dots \circ H_{\psi_1}) <\frac{\delta}{2},$$
which implies, by (\ref{aprx}), that $\hat F$ lies in the $\delta$-small $C^\infty$-neighborhood of $F$. Since the maps $\Phi_i$ are
independent of the choice of $\psi$, it follows that $\hat F$ is independent of the function $\psi$ either. Note that \eqref{disjoint2} follows from \eqref{disjoint1}.
\end{proof}

\begin{proof}[Proof of Proposition \ref{proprescalling}]\label{rm842}
Take any $L>0$, $\varepsilon>0$, and any integer $r >0$. Consider the set 
$\mathcal F_{ r,\varepsilon, L}$ of maps $\hat F\in Diff^\infty_\omega(\D, \R^2)$ such that for every function $\psi\in C^r( \R)$ whose $C^r$-norm is bounded by $L$, in the 
$\varepsilon$-neighborhood of $f$ in $C^r$ there exists a symplectic diffeomorphism $\hat f$ 
for which some renormalized iteration equals to $S_\psi\circ \hat F$. By Corollary \ref{qq},
the set $\mathcal F_{ r,\varepsilon, L}$ is dense in $Diff^\infty_\omega(\D, \R^2)$. 

\begin{Claim} The set $\mathcal F_{ r,\varepsilon, L}$ is open in $Diff^\infty_\omega(\D, \R^2)$. 
\end{Claim}
\begin{proof} This is inferred from (\ref{disjoint2})
by the following observation. Let $\hat F$ satisfies \eqref{rescalingpsi} for some $\psi$, $Q$, and $\hat f$:
$$\hat f|_{Q(\D)} = \hat f^{-(n-1)} \circ Q \circ S_{\psi}\circ \hat F\circ Q^{-1}|_{Q(\D)}.$$
Then, for any map $\tilde F$ close to $\hat F$, we take a map $\tilde f$ such that
it equals to $\hat f$ outside a small neighborhood of $Q(\D)$ and
$$\tilde f|_{Q(\D)} = \hat f^{-(n-1)} \circ Q \circ S_{\psi}\circ \tilde F\circ Q^{-1}|_{Q(\D)}.$$
If $\tilde F$ is close enough to $\hat F$, then $\tilde f$ is close to $\hat f$ too. So, by (\ref{disjoint2}),
the discs $\tilde f^m\circ Q(\D)$ stay at a sufficient distance from $Q(\D)$ for all $m=1, \dots, n-1$, which 
means that $\tilde f\equiv \hat f$ in the disc $\tilde f^m\circ Q(\D)$ for all $m=1, \dots, n-1$. By construction, this gives us
$$Q^{-1}\circ \tilde f^n\circ Q|_{\D} = S_{\psi}\circ \tilde F,$$
i.e., $\tilde F$ is a renormalized iteration of $\tilde f$. As the map $\tilde f$ can be taken as close as we want
to $\hat f$ if $\tilde F$ is close enough to $\hat F$, we also have that $d_{C^\infty}(\tilde f,  f)<\varepsilon$.
Thus, every $\hat F$ which is close enough to $\hat F$ lies in $\mathcal F_{ r,\varepsilon, L}$.\end{proof}

Now, we define the topologically generic set $\mathcal F \in Diff^\infty_\omega(\D, \R^2)$ from Proposition \ref{proprescalling} as
 \[\mathcal F= \bigcap_{(r, m, L)\in \N^3} \mathcal F_{r, \frac{1}{m}, L}.\]
\end{proof}

In the following, we complete the proof of the main theorem by proving the Rescaling Lemma.

\subsection{Local behavior near a nonlinear saddle.}  
Denote as $T_0$ the restriction of $f^s$ to a sufficiently small neighborhood of $O$.
One can introduce symplectic $C^{\infty}$-coordinates $(x,y)$ such that the local stable and unstable manifolds $W^s_{loc}(O)$ 
and $W^u_{loc}(O)$  get straightened (i.e., they acquire equations $y=0$ and $x=0$, respectively) and the restrictions
of $T_0$ onto $W^s_{loc}(O)$  and $W^u_{loc}(O)$ become linear (see Section 2.1 in \cite{GST07}). This means that the map 
$T_0:(x,y)\mapsto (\bar x, \bar y)$ takes the following form
\begin{equation}\label{t0form}
\bar x= \lambda x + p(x,y)x, \qquad   \bar y = \lambda^{-1} y + q(x,y)y,
\end{equation}
where $0< |\lambda| <1$ and
\begin{equation}\label{pq0}
p(x,0)=0, \qquad \qquad q(0,y)=0.
\end{equation}
Since $T_0$ is symplectic, $\det (DT_0)\equiv 1$, so
$$(\lambda +\partial_x(p(x,y)x))(\lambda^{-1} + \partial y (q(x,y)y)) -\partial_y p(x,y) \partial_x q(x,y) xy\equiv 1,$$ 
and it is a trivial exercise to check that this identity and (\ref{pq0}) imply
\begin{equation}\label{pq1}
p(0,y)=0, \qquad q(x,0)=0.
\end{equation}

Identities (\ref{pq0}), (\ref{pq1}) are important, because they imply nice uniform estimates for arbitrarily long iterations of $T_0$. 
Namely, the following result holds true:
\begin{lemm}[Lem. 7 in \cite{GST08}] Let \eqref{t0form}-\eqref{pq1} hold. 
There exist $\alpha>0$ and a sequence of functions
$\xi_k$, $\eta_k$ from $U_0:=[-\alpha,\alpha]^2$ into $\R^1$, such that for all $k$ large enough, we have for all $(x,y)$, $(\tilde x, \tilde y)\in U_0$:
\begin{equation}\label{xy0k}
T_0^k (x,y)= (\tilde x,\tilde y)\qquad\Longleftrightarrow\qquad \left[ 
\tilde x=\lambda^k x + \xi_k(x,\tilde y), \;\;  y=\lambda^k \tilde y + \eta_k(x,\tilde y)\right]\; .
\end{equation}
Moreover, the functions $\lambda^{-k}\xi_k$ and $\lambda^{-k}\eta_k$ are $C^\infty$-small when $k$ is large:
\begin{equation}\label{xieta}
\forall \rho\ge 1, \quad \|\xi_k, \eta_k\|_{C^\rho}=o(\lambda^k) \mbox{   as     }  k\to+\infty\; .
\end{equation}
 \end{lemm}

This lemma means that arbitrarily long iterations of $T_0$ are well approximated by the iterations of
its linearization. 

We will look for the perturbed map $\hat f$ in the form $\hat f = g\circ f$ where $g$ is a close to identity symplectic diffeomorphism.
In our construction $(g-id)$ will be supported outside of $f(U_0)$, so formulas (\ref{xy0k}) 
will stay valid for the perturbed map  $\hat f=g\circ f$ with the same $\xi_k$, $\eta_k$, and $\lambda$.

\subsection{Formulas for the iterations near the homoclinic band}

Recall that the homoclinic band $J$ is assumed to be small enough, so that $J\cap f^m(J) = \emptyset$ for every $m\not = 0$. 
This also implies that there exists an iteration $J^+$ of $J$ that lies in $W^s_{loc}(O):= (-\alpha,\alpha)\times\{0\}$ and 
an iteration $J^-$ that lies in $W^u_{loc}(O):=\{0\}\times (-\alpha,\alpha)$
such that $f^{\bar m}(J^-)=J^+$ for some $\bar m>0$ and $f^j(J^-) \cap U_0 = \emptyset$  for $j=1,\dots,\bar m-1$. We take $\alpha$ sufficiently small, 
so $\bar m$ is sufficiently large, in particular $\bar m>s$ (where $s$ is the period of $O$). Let $U_+\subset U_0$ and 
$U_-\subset U_0$ be sufficiently small neighborhoods of $J_+$ and $J_-$. As $f^{-1}(J_+)$ lies outside $U_0$, it follows that 
$U_+\cap f(U_0)=\emptyset$; similarly, $f^{-1}(U_0)\cap U_-=\emptyset$. Note also that
\begin{equation}\label{smu-0}
f^j(U_-)\cap  U_0=\emptyset \mbox{    for   }  j=1,\dots,\bar m-1.
\end{equation}
As $\bar m>s$, we also have that
\begin{equation}\label{smu+0}
U_+\cap T_0(U_0)=\emptyset.
\end{equation}
Denote
\[T_1:= f^{\bar m}|_{U_-}\; .\]

Choose $N$ different points $M_i^-:=(0,y_i^-)\in J^-$ and put  $M_i^+:=(x_i^+,0):=T_1 M_i^-$. 
The perturbation $g=\hat f\circ f$ will be supported in a small neighborhood of the points $M_i^+$. Thus, $\hat f$ will coincide with $f$ outside of $U_+$.
We could from the very beginning choose $J$ lying in $W^s_{loc}(O)$ sufficiently close to $O$, i.e., assume that $J_+=J$, Then the neighborhood $U_J$
in the formulation of the Rescaling Lemma is the neighborhood $U_+$, so the claim of the lemma that $\hat f$ coincides with $f$ in the complement of $U_J$
will be fulfilled.

As $J^-\subset \{0\}\times (-\alpha,\alpha)$ is sent to $J^+\subset (-\alpha,\alpha)\times \{0\}$, the map 
$T_1: (\tilde x,\tilde y)\mapsto (x,y)$ near the point 
$M_i^-=(0,y_i^-)$ can be written in the form
\begin{equation}\label{t1map}
x = x_i^+ + b_i (\tilde y-y_i^-) + \varphi_{1\, i}(\tilde x,\tilde y-y_i^-), \qquad 
y = (c_i +d_i (\tilde y-y_i^-)+ \varphi_{2\, i}(\tilde x,\tilde y-y_i^-))\, \tilde x,
\end{equation}
where the $C^{\infty}$-functions $\varphi_{1\, i}$ and  $\varphi_{2\, i}$ satisfy
\begin{equation}\label{phi12}
\varphi_{1\, i}(0) = 0, \qquad \partial_y \varphi_{1\, i}(0)=0,\qquad \varphi_{2\, i}(0) = 0, \qquad \partial_y \varphi_{2\, i}(0)=0.
\end{equation}
Obviously, this implies that
\begin{equation}\label{phi12j}
\varphi_{j\, i}(\tilde x, \tilde y -y_i^-) = O(|\tilde x|)+o(|\tilde y-y_i^-|), \qquad \partial_y \varphi_{j\, i}(\tilde x, \tilde y -y_i^-)=o(1)_{(\tilde x,\tilde y-y_i^-)\to0}\; , \qquad j=1,2.
\end{equation}

Note also that the area-preservation property of $T_1$ implies that the coefficients $b_i$, $c_i$ in (\ref{t1map}) satisfy
\begin{equation}\label{bc1}
b_ic_i = - \det (DT_1(M_i^-))=-1.
\end{equation}

\subsection{Scaling transformation}
We will further assume that the indices $i$ are defined modulo $N$, i.e. hereafter $i+1\equiv 1$ if $i=N$ and $i-1\equiv N$  if $i=1$.
We will need a sequence of positive real numbers $R_i$, $i=1,\dots N$, that satisfy
\begin{equation}\label{rc}
R_1=1, R_{i+1}=-c_{i+1}b_i R_{i-1}.
\end{equation}
Such sequence indeed exists when $N$ is odd: we define $R_i$ by (\ref{rc}) inductively: $R_1$, $R_3$, ... , $R_N$, $R_2$, ... ,$R_{N-1}$,
until we arrive to $R_{N+1}=(-1)^N \prod_{i=1}^N c_i \prod_{i=1}^N b_i$ and notice that the constraint $R_{N+1}=R_1$
is satisfied by virtue of (\ref{bc1}).
 
 \begin{figure}[h!]
	\centering
		\includegraphics[width=9cm]{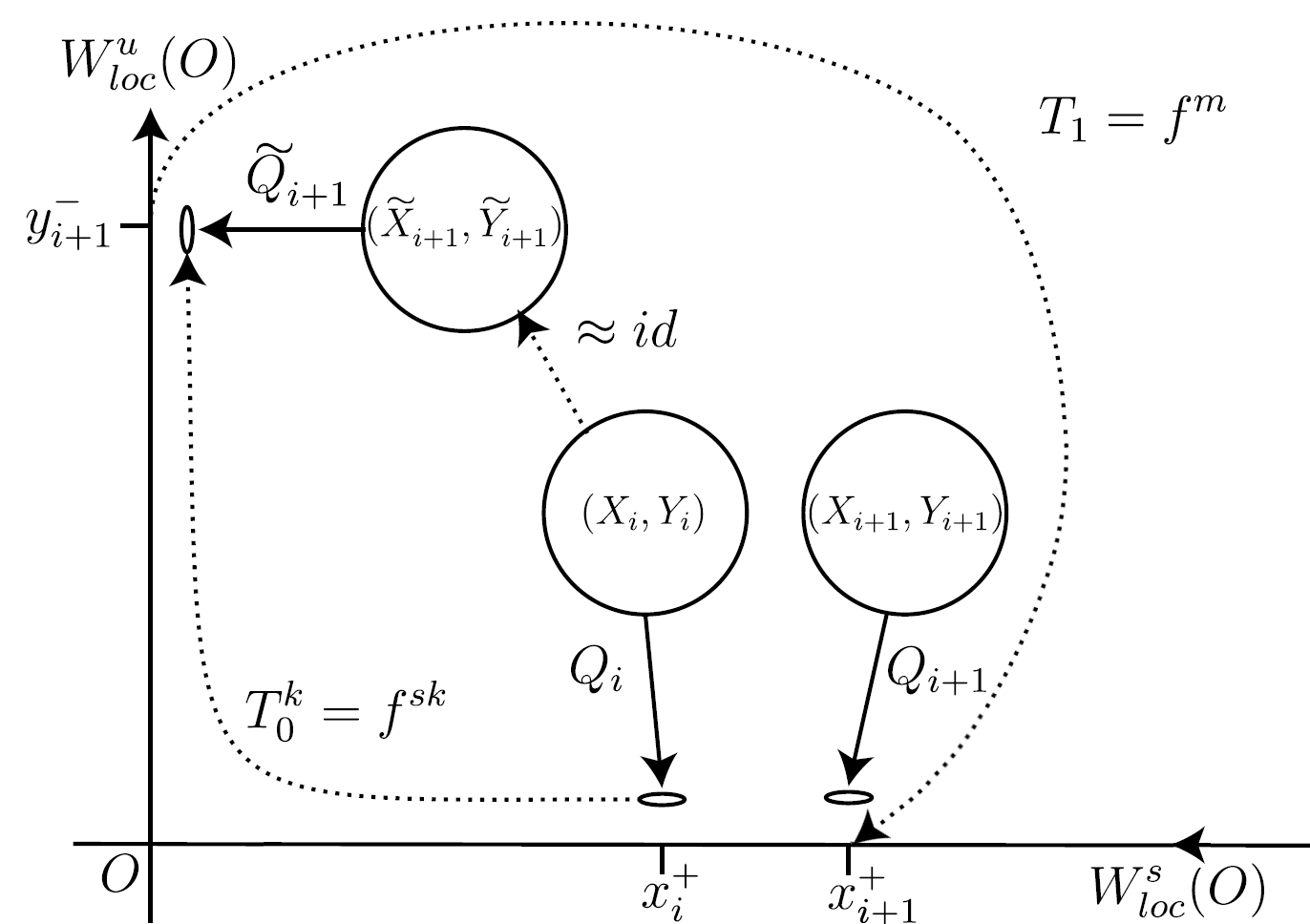}
\caption{Rescaling coordinates}\label{fig:island}
\label{fig:oldcoord}
\end{figure}

We now perform affine rescaling of the coordinates $(x,y)$ near the points $M_i^+$ and $(\tilde x,\tilde y)$ near the points $M_i^-$.
First, we chose a decreasing sequence $(\mu_k)_k$ of scaling constants $\mu_k \in(0,1)$ which converges to 0 sub-exponentially. In particular, for every $\rho\ge 1$, it holds:
\begin{equation}\label{lamu}
\mu_k\to 0 \qand \lambda^k\cdot \mu_k^{-\rho} \to 0 \text{ when }k\to \infty\; .
\end{equation}

Next, we take a sufficiently large integer $k$ and define coordinate transformations
\[ Q_{ik}: (X_i, Y_i)\mapsto (x, y)\quad \text{and}\quad  \tilde Q_{ik}: (\tilde X_i,\tilde Y_i)\mapsto (\tilde x, \tilde y)\; ,\]
by the following rule:
\begin{equation}\label{resc}
x=x_i^+ + b_i R_{i-1}\mu_k X_i, \qquad y=\lambda^k(y_{i+1}^- + \gamma_{i\, k}+R_i \mu_k Y_i),
\end{equation}
\begin{equation}\label{resc-}
\tilde x=\lambda^k(x_{i-1}^++\beta_{i\, k}+b_{i-1} R_{i-2}\mu_k \tilde X_i), \qquad \tilde y=y_i^- +R_{i-1} \mu_k \tilde Y_i,
\end{equation}
where the constant terms $\gamma_{i\, k}$ and $\beta_{i\, k}$ are given by
\begin{equation}\label{gabe}
\beta_{i\, k}= \xi_k(x_{i-1}^+,y^-_{i})\lambda^{-k},\qquad
\gamma_{i\, k}=\eta_k(x_i^+,y_{i+1}^-)\lambda^{-k}.
\end{equation}
It is obvious that given any real number $\hat L$, if $k$ is large enough, then the map $Q_{ik}$ sends 
the disc $\hat\D:= \{X^2+Y^2 \leq \hat L^2\}$ into a small neighborhood of 
$M_i^+$, whereas the map $\tilde Q_{ik}$ sends $\hat\D$ into a small neighborhood of $M_i^-$. Indeed,
$\gamma_{i\, k}\to 0$ and $\beta_{i\, k}\to0$ as $k\to+\infty$, by virtue of (\ref{xieta}), and $\lambda^k\to 0$, $\mu_k\to 0$ by \eqref{lamu}. Therefore, when $k\to \infty$, 
for any $(X_i,Y_i)$ and $(\tilde X_i,\tilde Y_i)$ in $\hat \D$, when $k\to \infty$, the following limits holds true:  $(x,y)\to  (x_i^+,0)=M_i^+$ in (\ref{resc}) and $(\tilde x,\tilde y)\to  (0,y_i^-)=M_i^-$ in (\ref{resc-}). 

Note that the map $Q_{1k}$ will be the map $Q$ in the statement of the Rescaling Lemma.

\bigskip

Next, we are going to derive formulas (see (\ref{xlylk2}),(\ref{nonp3})) for the maps $Q_{i+1,k}^{-1} T_1 \tilde Q_{i+1,k}$ and
$\tilde Q_{i+1,k}^{-1}T_0^k Q_{ik}$
for the original map $f$ and the perturbed map $\hat f$. As the composition:
 $$
 (Q_{N+1,k}^{-1} T_1 \tilde Q_{N+1,k}) (\tilde Q_{N+1,k}^{-1}T_0^k Q_{Nk})\circ \cdots (Q_{i+1,k}^{-1} T_1 \tilde Q_{i+1,k}) (\tilde Q_{i+1,k}^{-1}T_0^k Q_{ik})\circ \cdots \circ 
 (Q_{2,k}^{-1} T_1 \tilde Q_{2,k}) (\tilde Q_{2,k}^{-1}T_0^k Q_{1k})
   $$ equals
to $Q_{1k}^{-1} \circ \hat f^n \circ Q_{1k}$ for $n=N(ks +\bar m)$, we will obtain formula (\ref{rescaling}) in this way
and, thus, prove the Rescaling Lemma. 

\subsection{Renormalized iterations of the unperturbed map}

We start with the map $\tilde Q_{i+1,k}^{-1}T_0^k Q_i$, i.e., the map $T_0^k$ in the rescaled coordinates. 

\begin{lemm}\label{lemma6.3} Take any $\hat L>0$ and let $\hat\D$ be the disc $\{X^2+Y^2 \leq \hat L^2\}$ in $\R^2$.
For every sufficiently large $k$ there exist two real functions $\hat\xi_{ik}$, $\hat\eta_{ik}$ on $\hat\D$, which vanish at zero and satisfy for every $\rho\ge 1$:
\begin{equation}\label{htxieta}
\|\tilde\xi_{ik}, \tilde\eta_{ik}\|_{C^\rho} =o(1)_{k\to+\infty},
\end{equation}
such that for every $(X_i,Y_i)\in \hat\D$ the points $T_0^j Q_{ik}(X_i,Y_i)$ lie in $U_0$ for all $j=1,\dots k$, and 
$(\tilde X_{i+1},\tilde Y_{i+1})=\tilde Q_{i+1,k}^{-1}T_0^k Q_{ik}$ is given by
\begin{equation}\label{xlylk2}\begin{array}{l}
\tilde X_{i+1}= X_i+\tilde \xi_{ik}(X_i,Y_i), \\
\tilde Y_{i+1}= Y_i+\tilde \eta_{ik}(X_i,Y_i).\end{array}
\end{equation}
\end{lemm}
\begin{proof}
We will use the notation $(x_{i}, y_{i})= Q_{ik}(X_{i},Y_{i})$ and 
$(\tilde x_{i+1}, \tilde y_{i+1})= \tilde Q_{i+1,k}(X_{i+1},Y_{i+1})=T_0^k (x_{i}, y_{i})$. By 
(\ref{xy0k}), (\ref{resc}) and (\ref{resc-}), the first $k$ iterations of $(x_{i}, y_{i})$ by $T_0$ lie in $U_0$ if and only if:
\begin{equation}\left\{\begin{array}{rl}
\tilde x_{i+1}=\lambda^k x_i + \xi_k(x_i, \tilde y_{i+1}), &  y_i=\lambda^k \tilde y_{i+1} + \eta_k(x_i,\tilde y_{i+1}),\\
x_i=x_i^+ + b_i R_{i-1}\mu_k X_i, & y_i=\lambda^k(y_{i+1}^- + \gamma_{i\, k}+R_i \mu_k Y_i),
\\
\tilde x_{i+1}=\lambda^k(x_{i}^++\beta_{i+1\, k}+b_{i} R_{i-1}\mu_k \tilde X_{i+1}), & \tilde y_{i+1}=y_{i+1}^- +R_{i} \mu_k \tilde Y_{i+1}.
\end{array}\right.
\end{equation}
Thus,
\begin{equation}\label{xlylk}\begin{array}{l}
\lambda^k (\beta_{i+1\, k}+b_i R_{i-1}\mu_k \tilde X_{i+1})=
\lambda^k b_i R_{i-1}\mu_k X_i + \xi_k(x_i^+ + b_i R_{i-1}\mu_k X_i,y_{i+1}^- +R_i \mu_k \tilde Y_{i+1}), \\
\lambda^k (\gamma_{i\, k}+ R_i  \mu_k Y_i)=\lambda^k R_i \mu_k  \tilde Y_{i+1} 
+ \eta_k(x_i^+ + b_i  R_{i-1}\mu_k X_i,y_{i+1}^-+R_i \mu_k \tilde Y_{i+1}).\end{array}
\end{equation}
Using the definition of $\beta_{i+1\, k}$ and $\gamma_{i\, k}$, given in (\ref{gabe}), we can see from formula (\ref{xlylk}) 
that the zero value of $(X_i,Y_i)$ corresponds to the zero value of $(\tilde X_{i+1},\tilde Y_{i+1})$. So, we can rewrite (\ref{xlylk}) as
$$\mu_k \tilde X_{i+1}=\mu_k X_i + \xi_{ik}^\prime(\mu_k X_i,\mu_k \tilde Y_{i+1}), \qquad
\mu_k Y_i = \mu_k \tilde Y_{i+1}  + \eta_{ik}^\prime(\mu_k X_i, \mu_k \tilde Y_{i+1}),$$
where, as follows from \eqref{xieta} and \eqref{lamu}, the functions $\xi_{ik}^\prime, \eta_{ik}^\prime$ uniformly tend to zero in the $C^\infty$-topology as $k\to+\infty$;
moreover, they vanish when $(X_i, \tilde Y_{i+1})=0$. So,  we have that
\begin{equation}\label{xlylk1}
\tilde X_{i+1}=X_i + \tilde \xi_{ik}^\prime(X_i, \tilde Y_{i+1}), \qquad
Y_i = \tilde Y_{i+1}  + \tilde \eta_{ik}^\prime(X_i, \tilde Y_{i+1}),
\end{equation}
where the functions $\tilde \xi_{ik}^\prime(\cdot)=\mu_k^{-1} \tilde \xi_{ik}^\prime(\mu_k\cdot)$ and 
$\tilde \eta_{ik}^\prime(\cdot)= \mu_k^{-1} \tilde \eta_{ik}^\prime(\mu_k\cdot)$ are $C^\infty$-small, and these functions vanish at zero. 
By the implicit function theorem, the second equation of (\ref{xlylk1}) defines $\tilde Y_{i+1}$ as a vanishing at zero function of 
$(X_i,Y_i)$ which is uniformly smooth. Therefore, formula (\ref{xlylk1}) gives us the result of the lemma (i.e., relations (\ref{htxieta}) and 
(\ref{xlylk2})) with $\tilde \xi_{ik}(X_i,Y_i):=\tilde \xi_{ik}^\prime(X_i, \tilde Y_{i+1})$ and
$\tilde \eta_{ik}(X_i,Y_i):=-\tilde \eta_{ik}^\prime(X_i, \tilde Y_{i+1})$.
\end{proof}

As the next step, we consider the map $Q_{i+1,k}^{-1}T_1 \tilde Q_{i+1,k}$ (the map $T_1$ in the rescaled coordinates)
for the {\em unperturbed map} $f$.
\begin{lemm}\label{lemma6.4}
There are two real functions $\tilde\phi_{1ik}$, $\tilde\phi_{2ik}$ on the disc $\hat \D:=\{\tilde X^2+\tilde Y^2 \leq \hat L^2\}$ satisfying:
\begin{equation}\label{htpheta}
\tilde \phi_{1ik}
{\overset{ C^\infty }{\longrightarrow}}
  0\qand
\tilde \phi_{2ik}{\overset{ C^\infty }{\longrightarrow}}0\quad \text{as}\;  k\to \infty,
\end{equation}
 such that for every $(\tilde X_{i+1},\tilde Y_{i+1})\in \hat\D$ sent to $(X_{i+1},Y_{i+1})$ by $Q_{i+1,k}^{-1}T_1 \tilde Q_{i+1,k}$  we have
\begin{equation}\label{nonp2}
X_{i+1} =  \tilde Y_{i+1} + \tilde\phi_{1ik}(\tilde X_{i+1},\tilde Y_{i+1}),\qquad
Y_{i+1} = C_{ik}\mu_{k}^{-1} + A_i \tilde Y_{i+1} - \tilde X_{i+1} +\tilde\phi_{2ik}(\tilde X_{i+1},\tilde Y_{i+1}),
\end{equation}
where the constants\footnote{see (\ref{t1map}) and (\ref{resc}) for the definition of the coefficients in these formulas} $C_{ik}=[c_{i+1}(x_i^++\beta_{i+1\, k}) - y_{i+2}^- - \gamma_{i+1\, k}]/R_{i+1}$ and 
$A_i=d_{i+1} (x_i^++\beta_{i+1,k})R_i/R_{i+1}$ 
are uniformly bounded for all $k$.
\end{lemm}
\begin{proof} By definition, $Q_{i+1,k}(X_{i+1},Y_{i+1})=T_1 \tilde Q_{i+1,k}(\tilde X_{i+1},\tilde Y_{i+1})$.  
Let $(x_{i+1}, y_{i+1}):= Q_{i+1,k}(X_{i+1},Y_{i+1})$ and $(\tilde x_{i+1},\tilde y_{i+1})= \tilde Q_{i+1,k}(X_{i+1},Y_{i+1})$. By 
(\ref{t1map}), (\ref{resc}), and (\ref{resc-}) 
\[\left\{\begin{array}{c}x_{i+1} = x_{i+1}^+ + b_{i+1} (\tilde y_{i+1}-y_{i+1}^-) + \varphi_{1\, i+1}(\tilde x_{i+1},\tilde y_{i+1}-y_{i+1}^-),\\
y_{i+1} = \tilde x_{i+1} (c_{i+1}+d_{i+1} (\tilde y_{i+1}-y_{i+1}^-)+\varphi_{2\, i+1}(\tilde x_{i+1},\tilde y_{i+1}-y_{i+1}^-)),\\
\begin{array}{rl}
x_{i+1}=x_{i+1}^+ + b_{i+1} R_{i}\mu_k X_{i+1}, & y_{i+1}=\lambda^k(y_{i+2}^- + \gamma_{i+1\, k}+R_{i+1} \mu_k Y_{i+1})\; ,
\\
\tilde x_{i+1}=\lambda^k(x_{i}^++\beta_{i+1\, k}+b_{i} R_{i-1}\mu_k \tilde X_{i+1}), & \tilde y_{i+1} =y_{i+1}^-+R_{i} \mu_k \tilde Y_{i+1}\; .
\end{array}\end{array}\right.
\]
We replace $x_{i+1}$ and $y_{i+1}$ in the first and second lines by their expressions from the third line, and then eliminate the
 term $x_{i+1}^+$ which now appears on both sides of the first line. This gives us
\[\left\{\begin{array}{c}
 b_{i+1} R_{i}\mu_k X_{i+1}
=  b_{i+1} (\tilde y_{i+1}-y_{i+1}^-) + \varphi_{1\, i+1}(\tilde x_{i+1},\tilde y_{i+1}-y_{i+1}^-),\\
 \lambda^k(y_{i+2}^- + \gamma_{i+1\, k}+R_{i+1} \mu_k Y_{i+1})= \tilde x_{i+1}
(c_{i+1}+d_{i+1} (\tilde y_{i+1}-y_{i+1}^-)+\varphi_{2\, i+1}(\tilde x_{i+1},\tilde y_{i+1}-y_{i+1}^-)),\\
\begin{array}{rl}
\tilde x_{i+1}=\lambda^k(x_{i}^++\beta_{i+1\, k}+b_{i} R_{i-1}\mu_k \tilde X_{i+1}), & \tilde y_{i+1}-y_{i+1}^- =R_{i} \mu_k \tilde Y_{i+1}\; .
\end{array}\end{array}\right.
\]
We can now isolate the terms $X_{i+1}$ and $Y_{i+1}$, and replace $\tilde y_{i+1}-y_{i+1}^-$ by its expression $R_{i} \mu_k \tilde Y_{i+1}$ from the last line:
\[\left\{\begin{array}{c}
X_{i+1} = \tilde  Y_{i+1} + \frac{1}{b_{i+1} R_{i}\mu_k}\varphi_{1\, i+1}(\tilde x_{i+1},R_{i} \mu_k \tilde Y_{i+1}) ,\\
Y_{i+1}= -\frac{y_{i+2}^- + \gamma_{i+1\, k}}
{R_{i+1} \mu_k }+ \frac{\tilde x_{i+1}}{\lambda^k R_{i+1} \mu_k }
(c_{i+1}+d_{i+1}R_{i} \mu_k \tilde Y_{i+1}+\varphi_{2\, i+1}(\tilde x_{i+1},R_{i} \mu_k \tilde Y_{i+1})),\\
\tilde x_{i+1}=\lambda^k(x_{i}^++\beta_{i+1\, k}+b_{i} R_{i-1}\mu_k \tilde X_{i+1})\; .
\end{array}\right.
\]

Replacing $\tilde x_{i+1}$ in the first and second line by the right-hand side of the third line, we obtain
\begin{equation}\label{nonp1}\!
X_{i+1}\! =  \tilde Y_{i+1} + \tilde \phi_{1ik}(\tilde X_{i+1},\tilde Y_{i+1}),\quad
Y_{i+1}\! = C_{ik}\mu_k^{-1}\! +\!  A_i \tilde Y_{i+1}\! +\frac{c_{i+1}b_iR_{i-1}}{R_{i+1}} \tilde X_{i+1}\! +\tilde \phi_{2ik}(\tilde X_{i+1},\tilde Y_{i+1})\; ,
\end{equation}
where the coefficients $C_{ik}$ and $A_i$ are as defined in the statement of the lemma, and
\begin{equation}\label{ph1ikds}
\tilde \phi_{1ik}(\tilde X_{i+1},\tilde Y_{i+1})=\frac{1}{b_{i+1} R_i} \tilde\varphi_{1ik}(\tilde X_{i+1},\tilde Y_{i+1}),\end{equation}
\begin{equation}\label{ph2ikds}
\tilde \phi_{2ik}(\tilde X_{i+1},\tilde Y_{i+1})=
\frac{x_{i}^++\beta_{i+1\, k}+b_{i} R_{i-1}\mu_k \tilde X_{i+1}}{R_{i+1}}\tilde\varphi_{2ik}(\tilde X_{i+1}, 
\tilde Y_{i+1})+\frac{b_{i} R_{i-1}\tilde X_{i+1}} {R_{i+1}} d_{i+1}R_{i} \mu_k \tilde Y_{i+1},
\end{equation}
where
\begin{equation}\label{phjkd}
\tilde\varphi_{jik}(\tilde X_{i+1},\tilde Y_{i+1}) =
\mu_k^{-1}\varphi_{j\, i+1}(\lambda^k(x_i^++\beta_{i+1\, k}+b_i R_{i-1}\mu_k \tilde X_{i+1}), R_i \mu_k \tilde Y_{i+1}), \qquad j=1,2.
\end{equation}

Since the coefficient $\frac{c_{i+1}b_iR_{i-1}}{R_{i+1}}$ in the second equation of (\ref{nonp1}) equals to $-1$ by (\ref{rc}), 
formula (\ref{nonp1})
gives us the lemma (cf. formula (\ref{nonp2})), once we show that the functions $\tilde\phi_{jik}$ defined by (\ref{ph1ikds}),(\ref{ph2ikds}) 
satisfy (\ref{htpheta}). To do this, we just need to notice that for every given $\rho$, the $C^\rho$-norms of the functions $\tilde \varphi_{ijk}$ defined by 
(\ref{phjkd}) tend to zero as $k\to+\infty$. Indeed,  since the $C^\infty$-functions $\varphi_{j\, i+1}$ are independent of $k$, all their derivatives up to the order $\rho$ are uniformly bounded. 
Therefore, the derivatives $\frac{\partial^{\rho} \tilde\varphi_{jik}}{(\partial\tilde X_{i+1})^{\rho-\sigma}(\partial \tilde Y_{i+1})^\sigma}$
are estimated as $O(\lambda^{(\rho-\sigma)k} \mu^{\sigma-1}_k)$. Hence, by (\ref{lamu}), all the derivatives of order 2 and higher tend 
to zero as $k\to+\infty$, as well as the first derivative with respect to $\tilde X_{i+1}$. From (\ref{phi12j})-(\ref{lamu}), we also have that 
$\tilde\varphi_{jik}=o(1)_{k\to+\infty}$ and 
$\partial_{\tilde Y_{i+1}}\tilde\varphi_{jik}=o(1)_{k\to+\infty}$, and the required $C^\infty$-smallness of the functions 
$\tilde \varphi_{ijk}$ follows.
\end{proof}

\subsection{Construction of the perturbation $g$}
The perturbation we will now add to the map $f$ does not change the map $T_0$ (because the perturbed map $\hat f$ equals to
$g\circ f$ where $g$ will be supported in $U_+$ which does not intersect $f(U_0)$, so the difference $\hat f-f$ is supported outside of $U_0$). 
Therefore, Lemma \ref{lemma6.3} and its formula (\ref{xlylk2}) for $\tilde Q_{i+1,k}^{-1}T_0^k Q_{ik}$
remains the same when $f$ is replaced by the perturbed map $\hat f$. The map $T_1$ will be affected by the
perturbation, hence Lemma \ref{lemma6.4} and its formula (\ref{nonp2}) for the map $Q_{i+1,k}^{-1}T_1  \tilde Q_{i+1,k}$ will be modified.
In order to construct the perturbation map $g$, we take the functions $\psi_1,\dots,\psi_N$ from the statement of the lemma
and define
\begin{equation}\label{gpsi}
\hat\psi_{i+1}: x\mapsto 
-\lambda^k C_{ik}R_{i+1} - \lambda^k \frac{A_iR_{i+1}}{b_{i+1} R_i} (x- x_{i+1}^+) 
+ \lambda^k\mu_k R_{i+1} \psi_i(\mu_k^{-1}(x- x_{i+1}^+)/(b_{i+1}R_i)),
\end{equation}
where the constants $C_{ik}$, $A_i$ are the same as in (\ref{nonp2}). As condition (\ref{lamu}) implies that $ |\lambda|^k=o(\mu^r_k)$, we have
$$\|\hat\psi_{i+1}\|_{C^r}=o(\mu_k)\to 0  \mbox{    as    } k\to+\infty.$$

Recall that we use the notation $(x, y)$ for the non-rescaled coordinates near the set 
$\{M_i^+; 1\le i\le N\}= \{(x_i^+,0); 1\le i\le N\}\subset W^u_{loc}(O)$. 
Take $\varepsilon>0$ and define neighborhoods $V_i$ and $V_i'\subset V_i$ of the homoclinic points $M_i^+$: 
\[V_i := [-\varepsilon,\varepsilon]^2+M_i^+=[-\varepsilon+x_i^+,\varepsilon+x_i^+]\times [-\varepsilon,\varepsilon]\qquad 
V_i' := [-\varepsilon/2,\varepsilon/2]^2+M_i^+.\] 
We assume that $\varepsilon$ is small enough, so the discs $V_i$ all lie in $U_+$ and are mutually disjoint. Denote $V'=\sqcup_i V_i'$ and
$V=\sqcup_i V_i$, so $V'\subset V$. We choose $k$ large enough, so $Q_{ik}(\hat \D) \subset V_i'$, where $\hat \D$, the domain 
of the rescaling coordinates, is the disc from the formulation of Lemmas \ref{lemma6.3} and \ref{lemma6.4}. Below we define the perturbation
map $g$ such that the support of $g-id$ lies in $V$.

Let $\Psi\in C^{\infty}(\sqcup_i [-\varepsilon+x_i^+,\varepsilon+x_i^+],\R)$ be such that its derivative $D\Psi$, in restriction 
to $[-\varepsilon+x_i^+,\varepsilon+x_i^+]$,  satisfies $D\Psi=  \hat \psi_i$. Let $\rho\in C^\infty(\R,[0,1])$ be the bump function 
equal to $1$ identically in $V'$ and equal to $0$ identically outside of $V$. Let the perturbation map $g$ be equal to 
the time-1 map for the differential equation
$$\dot x = - \Psi(x) \partial_y \rho(x,y), \qquad \dot y = D\Psi (x) \rho(x,y) + \Psi(x)\partial_x \rho(x,y).$$
This is a Hamiltonian system in $\R^2$ (with the Hamiltonian $ - \Psi(x)\rho(x,y)$), so $g$ is a symplectic diffeomorphism. 
Since $\|\hat\psi_i\|_{C^r}\to 0$ as $k\to+\infty$, it follows that $\|g-id\|_{C^r}$ can be made
smaller than any given constant $\delta$ if $k$ is taken sufficiently large. Note also that the support of $g-id$ lies in $V\subset U_+$.
All this is in agreement with the statement of Rescaling Lemma.

By construction, the restriction of $g$ to $V_i'$ is given by
\begin{equation}\label{defgp}
g|_{V_i'}:(x, y)\mapsto (x, y + \hat\psi_i(x)).
\end{equation}
Plugging this into (\ref{t1map}), we find that for the perturbed map $\hat f=g\circ f$ the transition map 
$T_1=\hat f^m=g\circ f^m$ 
from a neighborhood of $M_{i+1}^-$ to a neighborhood of $M_{i+1}^+$ is given by
\begin{equation}\label{t1maphat}\begin{array}{ll}
x = x_{i+1}^+ +b_{i+1} (\tilde y-y_{i+1}^-) + \varphi_{1\, {i+1}}(\tilde x,\tilde y-y_{i+1}^-), \\
y = \hat\psi_{i+1}(x) + \tilde x(c_{i+1} + d_{i+1}(\tilde y-y_{i+1}^-)+\varphi_{2\,{i+1}}(\tilde x,\tilde y-y_{i+1}^-)),
\end{array}
\end{equation}
with the same functions $(\varphi_{1\, i+1},\varphi_{2\, i+1})$ and coefficients $x^+_{i+1}$, $y^-_{i+1}$, $b_{i+1}$, $c_{i+1}$, $d_{i+1}$ 
as for the unperturbed map $f$, and with the single additional term $\hat\psi_{i+1}(x)$ in the second equation. Therefore, because
the rescaling map $Q_{i+1,k}$ is affine, the only correction to the map $Q_{i+1,k}^{-1}T_1 \tilde Q_{i+1,k}$ due to this perturbation will be the additional term
$$DQ_{i+1,k}^{-1}\circ (0,\hat \psi_{i+1})\circ Q_{i+1,k}(X_{i+1},Y_{i+1}) = 
(0,\; \lambda^{-k}\mu_k^{-1}R_{i+1}^{-1}\hat\psi_{i+1}(x_{i+1}^+ + b_{i+1} R_i\mu_k X_{i+1}))$$
in the right-hand side of the respective equations in (\ref{nonp2}) (we use here the rescaling formula (\ref{resc}) with $i$ replaced 
by $i+1$; $DQ$ denotes the linear part of the affine map $Q$). By (\ref{gpsi}), this formula can be rewritten as
$$DQ_{i+1,k}^{-1}\circ (0,\hat \psi_{i+1})\circ Q_{i+1,k}( X_{i+1}, Y_{i+1})=(0,\; -\mu_k^{-1} C_{ik} - A_i X_{i+1} + \psi_i(X_{i+1})),$$
so the equation (\ref{nonp2}) for the map $Q_{i+1,k}^{-1}T_1 \tilde Q_{i+1,k}$ changes to
\begin{equation}\label{nonp3}\begin{array}{l}
X_{i+1} =  \tilde Y_{i+1} + \tilde \phi_{1ik}(\tilde X_{i+1},\tilde Y_{i+1}),\\
Y_{i+1} = - \tilde X_{i+1} +\psi_i(X_{i+1})+\tilde  \phi_{2ik}(\tilde X_{i+1},\tilde Y_{i+1}) - A_i\tilde \phi_{1ik}(\tilde X_{i+1},\tilde Y_{i+1}).
\end{array}
\end{equation}

By this formula and formula (\ref{xlylk2}) from Lemma \ref{lemma6.3}, the map $T_1T_0^k=\hat f^{\bar m+sk}=g\circ f^{\bar m+sk}$ 
from a small neighborhood of $M_i^+$ to a small neighborhood of $M_{i+1}^+$ is written, in the rescaled coordinates, as
\begin{equation}\label{ifin}
X_{i+1} = Y_i + \phi_{1ik}(X_i, Y_i), \qquad
Y_{i+1}= - X_i +\psi_i(X_{i+1})+\phi_{2ik}(X_i, Y_i),
\end{equation}
where
\begin{equation}\label{longlong}
\phi_{1ik}= \hat \eta_{ik} + \tilde \phi_{1ik}\circ (id+(\xi_{ik},\eta_{ik})),\qquad
\phi_{2ik}= -\hat \xi_{ik}+(\tilde  \phi_{2ik} - A_i\tilde \phi_{1ik})\circ (id+(\xi_{ik},\eta_{ik})).
\end{equation}
By (\ref{htxieta}),(\ref{htpheta}),
the functions $\phi_{1ik},\phi_{2ik}$ tend to zero uniformly in $C^\infty$ on any compact
as $k\to+\infty$. Importantly, the functions $(\phi_{1ik},\phi_{2ik})_i$ do not depend on the choice of the perturbation functions
$(\psi_i)_i$.

Formula (\ref{ifin}), in fact, completes the proof of the Rescaling Lemma. Indeed, it can be rewritten as
$$(X_{i+1},Y_{i+1}) = H_{\psi_i}\circ \Phi_i (X_i,Y_i),$$
where
$$\Phi_i(X_i,Y_i)=(X_i-\phi_{2ik}(X_i,Y_i), Y_i + \phi_{ ik}(X_i,Y_i)).$$
Thus, the map $(T_1 T_0^k)^N$ from a small neighborhood of $M_1^+$ takes indeed the required form 
(\ref{rescaling}) in the coordinates $(X,Y)=Q_{1k}^{-1}(x,y)$. In order to prove relation (\ref{disjoint2}) (with $Q=Q_{1k}$ and $n=N(ks+\bar m)$),
note that the functions $\psi_i$ are uniformly bounded by the constant $L$, and the functions $\phi_{jik}$ are small, so for any given constant $L_i$ 
the image of the disc $\{X_i^2+Y_i^2 \leq L_i^2\}$ by the map (\ref{ifin})
lies inside the disc  $\{X_{i+1}^2+Y_{i+1}^2 \leq L_{i+1}^2\}$ where $L_{i+1}=L_i+L+1$. Therefore, the images of the unit disc $\D$
by the first $N$ iterations of (\ref{ifin}) lie in the disc $\hat \D$ from Lemmas \ref{lemma6.3} and \ref{lemma6.4} if its radius $\hat L$
satisfies $\hat L > N(L+1)$. In other words, $\hat f^{i(ks+\bar m)}\circ Q_{1k}(\D) \subset Q_{i+1,k}(\hat \D)\subset V_{i+1}\subset U_0$ for all $i=0,\dots,N-1$.
Since the discs $V_i$ are mutually disjoint, we have that
$\hat f^{i(ks+\bar m)}\circ Q_{1k}(\D)\cap Q_{1k}(\D)$ for all $i=0,\dots,N-1$. The images $\hat f^m \circ Q_{1k}(\D)$ with $m\neq i(ks+\bar m)$ cannot intersect 
$Q_{1k}(\D)$ too, because they lie outside of $U_+$ by construction. Indeed, by Lemma \ref{lemma6.3}, 
$\hat f^{ls} Q_{ik}(\hat \D)=T_0^l Q_{ik}(\hat \D) \subset U_0\cap T_0 U_0$ for $l=1,\dots,k$ -- this implies that 
$\hat f^{j} Q_{ik}(\hat \D) \cap U_+ = \emptyset$ for $1\leq j\leq ks$ (see (\ref{smu+0})), and $\hat f^{ks} Q_{ik}(\hat \D) \subset U_-$ --
this implies that $\hat f^{ks+j} Q_{ik}(\hat \D) \cap U_+ = \emptyset$ for $j=1, \dots, m-1$ (see (\ref{smu-0})). 
Thus, relation (\ref{disjoint2}) is fulfilled for all $m=1,\dots, N(ks+\bar m)-1$, as required.
\bibliographystyle{alpha}
\bibliography{references}
\end{document}


\def\Max{\text{max}}
\def\P{\mathbb P}
\def\R{\mathbb R}
\def\T{\mathbb{T}}
\def\N{\mathbb N}
\def\Z{\mathbb Z}
\def\C{\mathbb C}
\def\D{\mathbb D}
\def\M{\mathbb M}
\def\a{{\underline a}}
\def\b{{\underline b}}
\def\n{{\underline n}}
\def\Log{\text{log}}
\def\loc{\text{loc}}
\def\inta{\text{int }}
\def\det{\text{det}}
\def\exp{\text{exp}}
\def\Re{\text{Re}}
\def\lip{\text{Lip}}
\def\leb{\text{Leb}}
\def\dom{\text{Dom}}
\def\diam{\text{diam}\:}
\def\supp{\text{supp}\:}
\newcommand{\ovfork}{{\overline{\pitchfork}}}
\newcommand{\ovforki}{{\overline{\pitchfork}_{I}}}
\newcommand{\Tfork}{{\cap\!\!\!\!^\mathrm{T}}}
\newcommand{\whforki}{{\widehat{\pitchfork}_{I}}}
\newcommand{\marginal}[1]{\marginpar{{\scriptsize {#1}}}}
\def\sR{{\mathfrak R}}
\def\sM{{\mathfrak M}}
\def\sA{{\mathfrak A}}
\def\sB{{\mathfrak B}}
\def\sY{{\mathfrak Y}}
\def\sE{{\mathfrak E}}
\def\sP{{\mathfrak P}}
\def\sG{{\mathfrak G}}
\def\sa{{\mathfrak a}}
\def\sb{{\mathfrak b}}
\def\sc{{\mathfrak c}}
\def\se{{\mathfrak e}}
\def\sg{{\mathfrak g}}
\def\sd{{\mathfrak d}}
\def\sr{{\mathfrak {sr}}}
\def\ss{{\mathfrak {s}}}
\def\sD{{\mathfrak {p}}}
\def\sp{{\mathfrak {p}}}
\def\arr{\overleftarrow}
\def\u{\underline}
\def \S{\mathbb S}
\def \A{\mathbb A}
\def \Diff{{\rm Diff}}
\def \qand{\quad \text{and} \quad}
\newtheorem{prop}{Proposition} [section]
\newtheorem{thm}[prop] {Theorem}
\newtheorem{conj}[prop] {Conjecture}
\newtheorem{defi}[prop] {Definition}
\newtheorem{lemm}[prop] {Lemma}

\newtheorem{prob}[prop] {Problem}

\newtheorem{sublemm}[prop] {Sub-Lemma}
\newtheorem{cor}[prop]{Corollary}
\newtheorem{theo}{Theorem}
\newtheorem{theoprime}{Theorem}
\newtheorem{Claim}[prop]{Claim}
\newtheorem{fact}[prop]{Fact}

\newtheorem{coro}[theo]{Corollary}
\newtheorem{defprop}[prop]{Definition-Proposition}
\newtheorem{propdef}[prop]{Proposition-Definition}
\newtheorem{RL}[prop]{Rescaling Lemma}

\newtheorem{question}[prop]{Question}
\newtheorem{conjecture}[prop]{Conjecture}

\theoremstyle{remark}
\newtheorem{exam}[prop]{Example}
\newtheorem{rema}[prop]{Remark}

\renewcommand{\thetheo}{\Alph{theo}}
\renewcommand{\thetheoprime}{\Alph{theo}$'$}
\title{Corrigendum to `On Herman's Positive Entropy Conjecture'}

\author{Pierre Berger\footnote{CNRS-Sorbonne Universit\'e-Université Paris Cité}\; and Dimitry Turaev\footnote{Imperial College London}}

\date{\today}

\maketitle

In the paper, in Section 5, we built a certain, sufficiently smooth, area-preserving diffeomorphism $F$ of $\R^2$, for which the stable and unstable manifolds of some periodic orbits are close on some fundamental domain (we call this a ``broken link").

There are two such links $L^a$ and $L^b$ to restore by a small perturbation of $F$. We express how they are broken in terms of Melnikov functions, which are  periodic functions from $\R/\tau\Z$ to $\R$. When the first link $L^a$ is restored, the second Melnikov function belongs to the space  $C^r_0(\R/\tau \Z, \R)$ of $C^r$-maps with zero mean.

Given a  function $\psi \in C^r(\R, \R)$, we construct a perturbation $S_{\psi} \circ F$  of  $F$ using the vertical shear $S_{\psi}$.  We also consider inclusions $C^r(\R/\tau\Z, \R)\to  C^r(\R, \R)$ so that the Melnikov functions induce two operators, which
 we name
\[ \mathcal M^a_\rho:  C^r(\R/\tau\Z, \R) \to C^r(\R/\tau\Z, \R)\qand  {\mathcal M}^b_\rho:  C^r_0(\R/\tau\Z, \R) \to C^r_0(\R/\tau\Z, \R).\]

Then the perturbation that restores the links corresponds to the zeros for these operators. 

A mathematician, anonymous to us, sent Marie-Claude Arnaud an ``AI audit'' of our paper, which identified an inaccuracy in our argument at this point. 
 In fact, it is a minor mistake, which does not affect the structure of the argument.

In the paper, we (correctly) determined that $\mathcal M^a_\rho$ is a perturbation of the identity  and $\mathcal M^b_\rho$ is the identity plus a perturbation of a contracting linear operator. We were wrong to write that $id- \mathcal M^a_\rho$  and $id- \mathcal
 M^b_\rho$ are {\em $C^1$} operators and then refer to the Banach contraction mapping principle to establish the existence and smallness of the fixed points.  These operators are in general not differentiable, only continuous (contrary to what is stated in
 5.8, 5.10 and 5.11).

This is a standard mistake when one deals with a graph-transform operator, and the remedy is also standard: Our computations in the proofs of lemmas 5.12 and 5.13 give, in fact, that the operators $id- \mathcal M^a_\rho$  and $id- \mathcal M^b_\rho$ send a
 $\delta$-$C^r$-ball into itself. The continuity of these operators in $C^{r-1}$ implies that they take  the $C^{r-1}$-closure of this ball into itself. This is a convex and compact (in the $C^{r-1}$ topology) set, so the Schauder principle gives the existence
 of the desired $\delta$-small  fixed points. Note that this argument loses one additional derivative, but this does not affect the overall proof.